\documentclass[usletter]{amsart} 

\usepackage{amsmath,amsthm,amssymb,amsfonts,mathrsfs,color,hyperref, mathtools,crop, graphicx, enumitem, todonotes,subcaption}
\usepackage[left = 3cm, right = 3cm]{geometry}

\theoremstyle{plain}
\begingroup
\newtheorem{theorem}{Theorem}[section]
\newtheorem*{theorem*}{Theorem}
\newtheorem*{"theorem"}{``Theorem''}
\newtheorem{corollary}[theorem]{Corollary}
\newtheorem{conjecture}[theorem]{Conjecture}

\newtheorem{lemma}[theorem]{Lemma}
\endgroup

\theoremstyle{definition}
\begingroup
\newtheorem{definition}[theorem]{Definition}
\endgroup

\theoremstyle{remark}
\begingroup

\newtheorem{example}[theorem]{Example}
\endgroup 

\numberwithin{equation}{section}
\newenvironment{pde}{\left\{\begin{array}{rll} } {\end{array}\right.}

\newcommand{\N}{\mathbb N} 
 
\newcommand{\Z}{\mathbb Z} 
\newcommand{\R}{\mathbb R}

\newcommand{\sign}{\mathrm{sign}}

\newcommand{\M}{{\mathcal M}}

\renewcommand{\L}{{\mathcal L}}

\newcommand{\C}{{\mathcal C}}
\newcommand{\F}{{\mathcal F}}

\newcommand{\LRa} {\Leftrightarrow}
\newcommand{\Ra} {\Rightarrow}
\newcommand{\wto}{\rightharpoonup}
\newcommand{\embeds}{\xhookrightarrow{\quad}}

\renewcommand{\d}{\mathrm{d}}

\newcommand{\dx}{\,\mathrm{d}x}

\newcommand{\ds}{\,\mathrm{d}s}
\newcommand{\dt}{\,\mathrm{d}t}

\newcommand{\eps}{\varepsilon}
\newcommand{\average}{{\mathchoice {\kern1ex\vcenter{\hrule height.4pt
width 6pt depth0pt} \kern-9.7pt} {\kern1ex\vcenter{\hrule
height.4pt width 4.3pt depth0pt} \kern-7pt} {} {} }}

\allowdisplaybreaks

 \makeatletter
\@namedef{subjclassname@2020}{2020 Mathematics Subject Classification}
\makeatother

\DeclareMathOperator*{\argmin}{argmin}

\begin{document}

\title[Overparametrized polynomial interpolation]{The double descent and Runge phenomena in overparametrized polynomial interpolation}

\author{Jason Wein}
\address{Jason Wein\\
Department of Mathematics\\
University of Pittsburgh\\
Pittsburgh, PA 15213}
\email{JSW116@pitt.edu}

\author{Stephan Wojtowytsch}
\address{Stephan Wojtowytsch\\
Department of Mathematics\\
University of Pittsburgh\\
Pittsburgh, PA 15213}
\email{s.woj@pitt.edu}

\date{\today}

\subjclass[2020]{41A10, 62J07}
\keywords{Overparametrization, regularization, double descent, polynomial interpolation, Runge phenomenon}

\begin{abstract}
The Runge phenomenon in polynomial interpolation is often considered a classical analogue of the double descent phenomenon in machine learning. In this note, we explore overparameterized polynomial interpolation in three popular polynomial bases: Monomial, Chebyshev and Legendre basis with coefficients that are minimal in the $\ell^2$-norm (and, for the monomial basis, also those minimal in the $\ell^1$-norm). We present our results primarily for equidistant and Chebyshev data points, but many results are independent of the exact form of sampling.
\end{abstract}

\maketitle


\section{Introduction}

Modern deep learning employs function models whose number of tunable parameters may exceed the number of available data points by several orders of magnitude. This overparametrization is crucial for the trainability of neural networks as the loss landscape of smaller networks is often plagued by the presence of spurious local minima \cite{safran2018spurious, chizat2018global, wojtowytsch2020convergence, chizat2020implicit, jacot2018neural, weinan2020comparative}.

From the perspective of statistical generalization on the other hand, highly expressive function models require subtle distinctions. For instance, neural networks of sufficient size are able to fit any collection of prescribed values $y_i$ not only at the known data points $x_i$, $i=1,\dots,n$, but also a number of additional points $x_{n+1}, \dots, x_{n+N}$. With heavy overparametrization, $N$ may even be significantly larger than $n$. Depending on the choice of $y_{n+1}, \dots, y_{n+N}$, a model trained on the augmented set $\{(x_i, y_i) : i = 1, \dots, n+N\}$ may generalize well to previously unseen data (in the sense of the original problem), or it may output meaningless noise. Thus, there are both `good' empirical risk minimizers and bad ones over the original set $\{(x_i, y_i) : i =1,\dots, n\}$.

In the spirit of structured risk minimization, it has been shown in various settings that minimizers of regularized risk functionals composed of a data fidelity term and a Tikhonov regularizer generalize well to new data if the labels are generated $y_i = f^*(x_i)$ with a target function $f^*$ in a suitable function class \cite{weinan2018priori}. As the number of labeled data points approaches infinity, these solutions approach minimum norm interpolants (in the target function space) also outside of the support of the data distribution \cite{park2023minimum}. Even when relying on implicit regularization by a training algorithm, ERMs often generalize if a `flat' minimum (in the parametrization loss landscape) is found, but the generalization gap may be subject to the curse of dimensionality \cite{liang2026stable}.

Despite classical statistical guidance that suggests avoiding models which can overfit noisy data (or adding a regularizer of positive strength), it has been overparametrized functions in fact generalize better than underparametrized functions in a variety of settings, assuming that the correct `minimum norm' interpolating solution is selected using (explicit or implicit) regularization. While the data is fit exactly (which can lead to `overfitting' the noise), any overfitting is `benign' and does not degrade generalization performance. The performance outside the training dataset tends to be worst in an intermediate regime where it is possible to exactly fit the labeled data, but the interpolating solution is unique. Independent of prior works \cite{loog2020brief}, this {\em double descent phenomenon} (where population loss falls on both sides of the interpolating threshold) was popularized by \cite{belkin2019reconciling, belkin2020two}.

Double descent is often likened to the {\em Runge phenomenon} in numerical analysis where the unique polynomials $P_{n,f}$ of degree $n$ which interpolate a target function $f$ at the $n+1$ equidistant data points $x_0, \dots, x_n$ in an interval $[a,b]$ satisfy $\lim_{n\to\infty}\|f-P_n\|_{L^\infty(a,b)} = +\infty$. The phenomenon occurs even for innocuous functions such as $f(x) = 1/(25x^2+1)$ in the interval $[-1,1]$ \cite{epperson1987runge, stackexchange_runge}. We can think of the Runge function as an entirely deterministic mechanism for overfitting.

In the classical Runge setting, the number of parameters is tethered precisely to the interpolation threshold. If the polynomial degree goes to infinity at a rate $d_n \ll n$ and the least squares solution $P_{d_n, n} = \argmin_{P\in \mathcal P_{d_n}} \frac1n\sum_{i=1}^n |P(x_i) - f(x_i)|^2$ is selected, the phenomenon is not observed. The question addressed in this note is: {\em In the simple setting of polynomial interpolation in one dimension, is there a double-descent phenomenon as $d_n\gg n$? And how does this depend on the selection mechanism for overparametrized solutions?}

Following machine learning practice, we consider regularizers which operate on {\em parameter space}, not the function space directly. More precisely, we select a basis $\mathcal Q_d = \{q_0, \dots, q_d\}$ for the space of polynomials $\mathcal P_d$ of degree at most $ d$ and select a regularizer $R:\R^{d+1}\to [0, \infty)$ like $R(a) = \|a\|_p^p$ for the represented function $P_{\mathcal Q,a}= \sum_{k=0}^d a_k q_k$. Given a dataset $\mathcal X = \{(x_i, y_i): i=0,\dots, n\}$, the polynomial under consideration is $P_{\mathcal Q_d, a^*}$ where
\[
 a^* \in \argmin \left\{R(a) : a\in\R^{d+1} \text{ such that } P_{\mathcal Q_d, a}(x_i) = y_i\text{ for all } i = 0,\dots, n\right\}.
\]
The set is guaranteed to be non-empty if $d\geq n$ (and the points $x_i$ are distinct) and the minimizer is unique if $R(a) = \|a\|_2^2$, but not necessarily if $R(a) = \|a\|_1$.

In this note, we focus primarily on two settings:

\begin{enumerate}
\item $\mathcal Q$ is the monomial basis $q_k(x) = x^k$ and $R(a) = \|a\|_1$ is an $\ell^1$-regularizer.

\item $\mathcal Q$ is the basis of Chebyshev polynomials and $R(a) = \|a\|_2^2$ is an $\ell^2$-regularizer.
\end{enumerate}

Partial results are also obtained for $\ell^2$-minimal coefficient interpolants in monomial basis and Legendre basis. Noted in \cite{ongie2026representation}, in the overparametrized setting, common parameter space regularizers correspond to a norm (or at least quasi-norm) regularizer on a function space. This is also the case for us, in two very different ways.

Consider the Chebyshev basis first. As Chebyshev polynomials are (essentially) orthonormal in $L^2(\nu)$ for a suitable measure $\nu$ on $(-1,1)$, the parameter space regularizer corresponds to a function space regularizer as $d\to\infty$ with
\[
P_{d} = \argmin_{P\in\mathcal P_d} \left\{ \int_{-1}^1 \frac{P(x)^2}{\sqrt{1-x^2}}\dx + \left(\int_{-1}^1 \frac{P(x)}{\sqrt{1-x^2}}\dx\right)^2 : P(x_i) = y_i\text{ for all } i=0,\dots, n\right\}.
\]
If $d$ is large enough compared to the number of data points, we observe that it is possible to fit a given collection of input/output pairs $\{(x_i, y_i) : i=0,\dots, n\}$ by a polynomial of vanishingly small $L^2(\nu)$-norm. Consequently $\lim_{d\to \infty} P_d = 0$ in $L^2(-1,1)$ and in the weighted $L^2$-space in which the Chebyshev basis is orthogonal. Immediately, we find that the population risk remains bounded as $d\to \infty$ and
\[
\lim_{d\to \infty} \|P_{d,\mathcal X} - f\|_{L^2(-1,1)}^2 = \|f\|_{L^2(-1,1)}^2.
\]
If the points $x_i$ are equidistant in $(-1,1)$ and the labels are given by $y_i = f(x_i)$ for a function $f$ which suffers from the Runge phenomenon, this shows that a descent exists towards heavy overparametrization, but with a suboptimal limiting value: Even the zeroth-degree polynomial (constant) approximation $P_0 = \frac1n\sum_{i=1}^nf(x_i)$ typically does better at approximating $f$ than the high degree limit (i.e.\ the constant zero function) if $\int_{-1}^1f(x)\dx \neq 0$. 

Closer to practice, we observe the non-commutation of limits as both the number $n$ of data points $(x_i, f(x_i))$ and the polynomial degree $d$ approach infinity:
\[
0 = \lim_{d\to\infty} \lim_{n\to\infty} \|f-P_{d,n}\|_{L^2(-1,1)} \neq \lim_{n\to\infty}\lim_{d\to\infty}\|f-P_{d,n}\|_{L^2(-1,1)} = \|f\|_{L^2(-1,1)},
\]
for instance if the points $x_i$ are equidistant or distributed independently according to the uniform distribution on $(-1,1)$. This is easy to see as $\lim_{n\to\infty} P_{d,n}$ is simply the $L^2$-orthogonal projection of $f$ onto the space $\mathcal P_d$ of degree $d$ polynomials, independently of the choice of basis for $\mathcal P_d$.

We conjecture that the situation of the Legendre basis is similar, but significantly more subtle, and we do not fully settle the question here. The difference arises as Legendre polynomials form an orthogonal basis whose norm approaches zero at the rate $\|p_d\|^2 = O(1/d)$, which vanishes just slowly enough to prevent effective regularization. We conjecture that in Legendre basis, the $\ell^2$-regularizer on the basis coefficients acts as an $H^s$-type regularizer for the critical case $s=1/2$ which just fails to guarantee continuity. For even large finite degree, empirical regularization is observed, but we conjecture that the large degree limit degenerates in the same way as for Chebyshev polynomials albeit at a slower rate.

In contrast, the monomial basis is possibly most intuitively viewed through the lens of complex analysis. Specifically, we note that if $(a_i)_{i\in \mathbb N_0}$ is an absolutely summable complex sequence, then the power series $f(x) = \sum_{i=0}^\infty a_iz^i$ converges absolutely and uniformly on the unit disk since
\[
\left|\sum_{k=0}^D a_k z^k - \sum_{i=0}^d a_kz^k\right| \leq \sum_{k= \min\{d,D\}+1}^{\max\{d,D\}} |a_k| \,|z|^k \leq \sum_{k= \min\{d,D\}}^\infty |a_k| \qquad \forall\ z\in \overline{B_1(0)}\subset \mathbb C.
\]
In particular, $f$ is a continuous function on the closed unit disk of the complex plane (and holomorphic in its interior). This indicates a bias which is very different from that of the Chebyshev basis and tied to analytic function representation. It is less compatible with an $L^p$- or $C^k$-perspective since the map $a\mapsto \sum_i a_iz^i$ is continuous, but its inverse is not with respect to either topology: two functions may be very close with very different power series, at least when only data on the real axis is observed.

For instance, if $f:[0,1]\to\R$ is a continuous functions satisfying $f(0) = 0$, then we can extend $f$ to $[-1,1]$ by even reflection as $f_e$ or odd reflection as $f_o$. Approximating $f_e, f_o$ by polynomials $P_e, P_o$ in $C^0([-1,1])$ by the Stone-Weierstrass Theorem, we may assume that $P_e$ is even and $P_o$ is odd since
\begin{align*}
\left| f_e - \frac{P_e(x) + P_e(-x)}2 \right| &= \left| \frac{f_e(x) + f_e(-x)}2 - \frac{P_e(x) + P_e(-x)}2 \right| \leq \frac{|f_e(x) - P_e(x)| + |f_e(-x) - P_e(-x)|}2\\
    &\leq \frac{2\,\|f_e-P_e\|_{C^0([-1,1])}}2 = \|f_e-P_e\|_{C^0([-1,1])},
\end{align*}
i.e.\ the even part of $P_e$ approximates the even function $f_e$ at least as well as the generic polynomial $P_e$, and analogously for $P_o, f_o$. Thus, both $P_e$ and $P_o$ approximate $f$ well on the original interval $[0,1]$, but their coefficients $a_{k, e}$ and $a_{k, 0}$ are entirely different: $a_{k,e} =0$ if $i \in 2\mathbb N +1$ and $a_{k, o} = 0$ if $i\in 2\mathbb N$. By integration, the same discrepancy can be achieved approximating $f$ in $C^m([0,1])$ for any $m\in \mathbb N_0$ if $f$ has $m$ continuous derivatives such that we simply approximate $f^{(m)}$.

To understand the selection principle and approximation properties in the monomial basis, we therefore focus on data $y_i = f(x_i)$ for functions which have a convergent power series representation $f(x) = \sum_{k=0}^\infty a_kx^k$. Eschewing the complex plane and using entirely real analytic methods, we demonstrate that heavily overparametrized interpolating polynomials with $\ell^1$-minimal monomial basis coefficients converge rapidly to the target function also between data points if the power series coefficients of $f$ are real and non-negative. For (real or complex) $\ell^2$-minimal coefficients in the monomial basis, we draw connections to the Hardy space $H^2$ of harmonic analysis.

This note illustrates the delicate behavior in a popular introductory example for overparametrization, bias/complexity trade-off and double descent \cite[Section 15.1]{petersen2024mathematical}, \cite[Section 1.3]{ma2020towards}. It matches subtleties in deep learning, where even the choice of penalizing the biases in addition to the weights of a neural network can have significant impact \cite{boursier2023penalising, wojtowytsch2024optimal}, and it matches the observation that stable models often offer generally mediocre performance while less stable models may perform spectacularly well on a smaller set of compatible problems, but not necessarily generic problems.

The article is organized as follows. In the remainder of the introduction, we review classical material on the Runge phenomenon. In Section \ref{section monomial ell1}, we study minimum norm interpolants with respect to the $\ell^1$-norm of monomial basis coefficients. Section \ref{section chebyshev} is dedicated to the study of interpolants with $\ell^2$-minimal coefficients for the basis of Chebyshev polynomials of the first kind (or, more generally, suitable orthogonal bases in general). In Section \ref{section others}, we provide a partial analysis of interpolants with $\ell^2$-minimal coefficients in the monomial basis and the Legendre basis.

To conclude the introduction, let us note that we do not prove that overparametrized interpolants outperform the Lagrange polynomial $P\in\mathcal P_n$, i.e.\ the unique degree $n$ polynomial which interpolates all given data $\{(x_i, y_i) : i=0,\dots, n\}$, in situations when it converges (although we empirically observe an instance of such behavior in Figure \ref{figure monomial l1 and l2}). Rather, our results prove that the quality of solutions does not deteriorate (monomial basis, $\ell^1$-minimal coefficients, and in a restrictive class of target functions) or that the solutions approach the zero function and thus at least do not blow up (Chebyshev basis, $\ell^2$-minimal coefficients). The goal of the article is not to develop new techniques in polynomial interpolation, but rather to explore and explain the nuanced behavior of an example which is often presented as introductory, but remains poorly understood.

In a companion article \cite{second_article}, we illustrate that a double descent phenomenon does arise when we select the solution of least coefficient space $\ell^2$-norm with respect to a scaled Chebyshev basis
\[
q_{k, s}(x) = \frac{T_k(x)}{\sqrt{k^{2s}+1}}, \qquad\text{where }T_k(x) = \cos\big(k\,\arccos(x)\big)
\]
to interpolate the Runge type function $f(x) = 1/((6x)^2+1)$ based on equidistant data points in $[-1,1]$ -- see also Figure \ref{figure second project}. As the technical details are quite distinct, we focus on standard polynomial bases here.

\begin{figure}
\centering
\includegraphics[width = .8\textwidth]{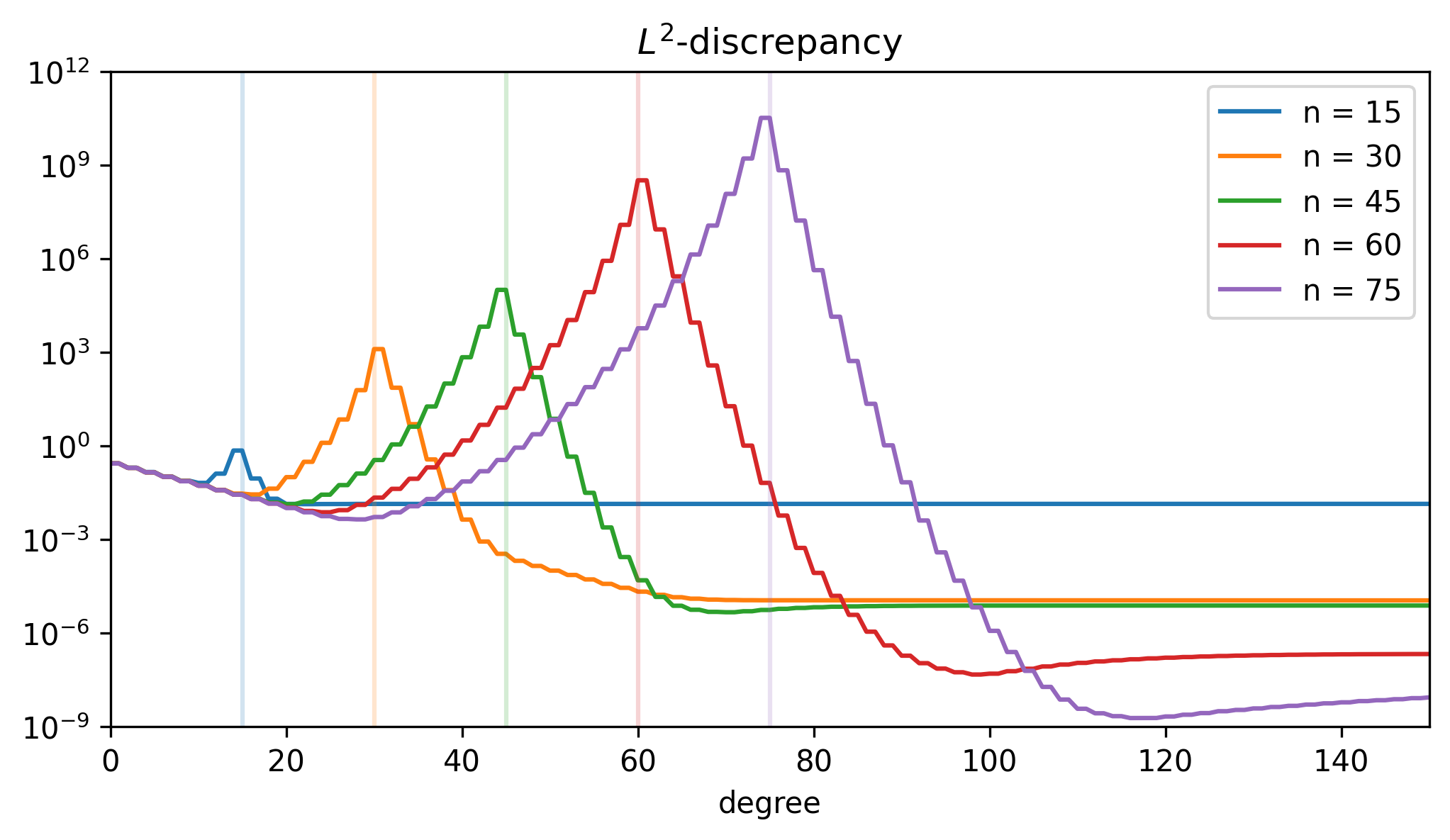}
\caption{\label{figure second project}
$L^2(-1,1)$-discrepancy between the target function $f(x) = (36x^2+1)^{-1}$ and a least squares fit polynomial (underparametrized case) and the unique polynomial which minimizes the $\ell^2$-coefficient norm with respect to a suitably scaled Chebyshev polynomial basis. The vertical line indicates the interpolation threshold for the plot of the corresponding color. For details, see \cite{second_article}.
}
\end{figure}

\subsection{The Runge phenomenon}

For $n+1$ distinct points $\mathcal X_n = \{x_0, \dots, x_n\}$ in a fixed interval $[a,b]$ and a function $f:[a,b]\to\R$, there always exists a unique polynomial $P_n$ of degree $n$ such that $P_n(x_i) = f(x_i)$ for $i= 0,\dots, n$ (the `Lagrange polynomial'). In between the data points, $P$ may or may not be close to $f$, and Faber's Theorem \cite{faber1914interpolatorische, trefethen2019approximation} guarantees that there is no family of finite sets $\mathcal X_n$ such that the Lagrange polynomial converges to $f$ for all continuous functions $f$.

A particularly important instance is that of equidistant grids $\mathcal X_n = \{0, 1/n, \dots, (n-1)/n, 1\}$. Bernstein showed that Lagrange polynomials for $f(x) = |x-1/2|$ based on equi-distant points diverge everywhere except at $1/2$ and the interval ends -- see \cite{revers2000lagrange} for extensions and historic context. The `Runge phenomenon' asserts that $P_n$ does not converge to $f$ for functions which are real analytic on the interval \cite{epperson1987runge}. Convergence guarantees are only available for analytic functions in suitable neighborhoods of the interval \cite{epperson1987runge}.

For interpolants at Chebyshev nodes, on the other hand, convergence is guaranteed, for instance, for all one-dimensional `Barron functions' $f$ (see \cite[Section 4.1]{wojtowytsch2022representation} for context), i.e.\ all continuously differentiable functions whose derivative is of bounded variation \cite[Chapter 7]{trefethen2019approximation}. In particular, this includes all functions which are at least $C^{1,1}$-smooth. 
 

\subsection{Notation}

If $k_1, k_2\in\mathbb N_0$, we denote $[k_1:k_2] = \{k_1, k_1 +1, \dots, k_2\} = [k_1,k_2]\cap \mathbb Z$.

\section{Monomial Basis Regularization} \label{section monomial ell1}

In this section, we demonstrate that polynomials $P_d(x) = \sum_{k=0}^d b_{k,d}x^k$ of high degree $d$ which interpolate data $y_i = f(x_i)$ are close to $f$ (more precise statement below) if 
\begin{enumerate}
\item the target function $f$ belongs to the restrictive function class of {\em absolutely monotonic} and analytic functions,
\item we select $P_d$ such that the coefficient norm $\|b_d\|_{\ell^1} = \sum_{k=0}^d |b_{k,d}|$ is minimized, and
\item $x_i\geq 0$ for all $i$ and $1\in \{x_0, \dots, x_n\}$ (or the data is `balanced').
\end{enumerate}
The final condition arises as any derivative of a polynomial or power series with non-negative coefficients is non-negative on $[0,\infty)$, but not necessarily on $\R$.
Our main result is the following.

\begin{theorem} \label{main theorem}
Let $0 = x_0 < x_1 < \dots < x_n = 1$ such that $x_i - x_{i-1} < K/n$ for some $K\geq 1$ and all $i \in [1:n]$. Assume that $f(x) = \sum_{k=0}^\infty a_kx^k$ is a power series which converges on $[0,1]$ such that $a\in \ell^1$ and $a_k\geq 0$ for all $k\in\mathbb N.$

For $d\geq n$, the `interpolating parameter family'  
\[
\mathcal F_d :=\left\{ b\in \R^{d+1} : \sum_{k=0}^d b_k x_i^k = f(x_i)\ \text{ for } i \in [0:n]\right\},
\qquad \M_d = \argmin_{b\in \F_d}\|b\|_1
\]
is non-empty and if we denote $P_b(x) = \sum_{k=0}^d b_kx^k$ for $b\in\F_d$, then
\begin{enumerate}
\item $\lim_{d\to\infty} \sup_{P_b\in\M_d}\|f-P\|_{L^1(0,1)} \leq K\frac{f(b) - f(a)}n$.
\item If $f^{(m)}(1)<\infty$ for $m \leq n$, then
\[
\lim_{d\to\infty} \sup_{b\in\M_d}\|f-P_b\|_{L^\infty(0,1)} \leq 2\,f^{(m)}(1) \left(\frac Kn\right)^m.
\]
\item In general, if $\zeta = x_J \in (0,1)$ and $J\geq m$, then
\[
\lim_{d\to\infty} \sup_{b\in\M_d}\|f-P_b\|_{L^\infty(0,\zeta)} \leq \frac{2\,m!\,f(1)}{(1-\zeta)^{m+1}}\, \left(\frac Kn\right)^m.
\]
\end{enumerate}
The same bounds hold if $\mathcal X:= \{x_0,\dots, x_n\}$ contains $0, 1$ and is {\em balanced}, i.e.\ $x\in\mathcal X\ \Ra -x\in\mathcal X$. In this case, $K\geq 2$ on the larger interval $[-1,1]$ and we consider $(-\zeta,\zeta)$ instead of $(0,\zeta)$.
\end{theorem}

In other words, if an $\ell^1$-minimal solution (with respect to monomial basis coefficients) in the interpolating set is selected, then it is quantitatively close to the function $f$. The result could be stated more generally on an interval $[a,b]\subseteq[0,\infty)$ or $[-a,a]$. Using this, we can easily compare the overparameterized to the underparameterized setting.

\begin{corollary} \label{comparison corollary}
Let $[a,b]\subset [0,\infty)$ and $a = x_0 < x_1 < \dots < x_n = b$ with $1\in \mathcal X_n := \{x_0, \dots, x_n\}$. Assume that $K\geq b-a$ such that $x_{i}-x_{i-1}\leq \frac{K}{n}$ for all $i\in\{1,\ldots, n\}.$ 

Let $f(x) = \sum_{k=0}^\infty a_kx^k$ be an analytic function with convergent power series on $[a,b]$ and non-negative power series coefficients $a_k\geq 0$. Denote by $P_d$

\begin{enumerate}
\item the unique polynomial $P\in \mathcal P_d$ which minimizes the mean squared error $P\mapsto \frac1{n+1}\sum_{i=0}^n \big|P(x_i) - f(x_i)\big|^2$ over the dataset if $d\leq n$ and
\item any polynomial $P(x) = \sum_{k=0}^d b_kx^k$ of degree at most $d$ such that $P(x_i) = f(x_i)$ for all $i=0, \dots, n$ and 
\[
\|b\|_{\ell^1} \in \argmin \left\{\|c\|_{\ell^1} : \sum_{k=0}^d c_k x_i^k = f(x_i) \:\:\forall\ i = 1,\dots, n\right\}
\]
if $d\geq n$.
\end{enumerate}
If $f$ is not constant and
\[
n_0 = K\,\frac{\big(f(b) - f(a)\big)^2}{\int_a^b \big(f- \langle f\rangle\big)^2\dx} \qquad\text{where}\quad  \langle f\rangle = \frac1{b-a}\int_a^bf(x)\dx
\]
then $\limsup_{d\to\infty}\|P_d - f\|_{L^2(a,b)} < \|P_0 - f\|_{L^2(a,b)}$ for all $n> n_0$.  
\end{corollary}

The corollary can be seen as a very weak (single) descent guarantee: Very high degree polynomials perform better than degree zero polynomials, if suitable minimum norm interpolants are selected. However, the corollary does not describe an initial ascent, and the decrease in $d$ may be monotonic. In particular, it does not compare $P_d$ for $d\gg n$ to the Lagrange polynomial $P_n$, i.e.\ the unique interpolating polynomial of minimal degree. We present it in part for comparison with Corollary \ref{corollary suboptimal descent} for the Chebyshev basis.

Similarly, if $f$ is not a polynomial of degree $d_0$ and $n$ is large enough, we could guarantee that high degree interpolants outperform $P_{d_0}$. We do not consider a direct comparison between $P_n$ and $P_d$ for $d\gg n$ here, but note that Figure \ref{figure monomial l1 and l2} suggests a possible benefit to overparametrization for certain target functions and data samples.

\subsection{\texorpdfstring{Minimum coefficient-$\ell^1$-norm interpolants}{Minimum coefficient-l1-norm interpolants}}\label{section ell^1-norm existence}

Our first lemma illustrates that if the data is generated by an analytic function $f$ whose power series has non-negative coefficients, then minimum coefficient norm interpolants of the given data have the same form.

\begin{lemma} \label{min ell-1 interpolant}
Assume that $f(x) = \sum_{k=0}^\infty a_kx^k$ is a power series with non-negative coefficients $a_k\geq 0$ which converges on an (open or closed) interval $I$. Assume further that $\mathcal X\subseteq I$ is a subset with $1\in\mathcal X$. Denote
\[
\mathcal C = \left\{ b\in \ell^1 : \sum_{k=0}^\infty b_k x^k = f(x) \:\:\forall\ x\in \mathcal X\right\}.
\]
Then
\[
\argmin_{b\in \mathcal C}\|b\|_{\ell^1} = \{ b\in \mathcal C : b_k\geq 0\:\:\forall\ k\in\mathbb N_0\}.
\]
\end{lemma}

\begin{proof}
We observe that 
\[
\|a\|_{\ell^1} = \sum_{k=0}^\infty a_k = f(1) = \sum_{k=0}^\infty b_k \leq \|b\|_{\ell^1} 
\]
with equality if and only if $b_k\geq 0$ for all $k$. Thus $f(1)$ is a lower bound for the norm which is attained whenever all coefficients are non-negative. The set is non-empty by assumption. 

The convergence of the power series $\sum_{k=0}^\infty b_kx^k$ for $b\in \mathcal C$ is assured on $[-1,1]$ by the summability of $b$ as seen in the introduction. Due to the interpolation condition, it converges on a potentially larger set $(-r,r)$ where $r = \sup\{|x| : x\in\mathcal X\}$ since convergence at $x\in \mathcal X$ implies convergence on the interval $(-|x|, |x|)$.
\end{proof}

Naturally, the same proof applies to more general expansions $\sum_{k=0}^\infty b_k p_k(x)$ if $p_k(1) = 1$ for all $k\in \mathbb N_0$, e.g.\ Legendre or Chebyshev polynomials. 

In a second Lemma, we illustrate that polynomials of sufficiently high degree can interpolate data $(x, f(x))$ for $x\in \mathcal X$ with a finite set $\mathcal X$ such that all coefficients in the monomial basis are non-negative. Thus, we do not require a true power series.

\begin{lemma} \label{existence lemma}
Let $f(x) = \sum_{k=0}^\infty a_kx^k$ be a power series such that $a_k\geq 0$ for all $k$ and $\mathcal X\subset [0,\infty)$ a finite set. Then there exist $d\in \N$, $b\in \R^{d+1}$ such that $b_k\geq 0$ for all $k$ and $P(x) := \sum_{k=0}^d b_k x^k$ satisfies $P(x) = f(x)$ for all $x\in \mathcal X$.
\end{lemma}

\begin{proof}
In the case that $a_k\neq 0$ for only finitely many indices $k$, $f$ is already a polynomial and we take $P = f.$

Otherwise, there are infinitely many indices such that $a_k > 0.$ We write the finite dataset as $\mathcal X = \{x_0, \ldots, x_n\}$.
Then, we take $n+1$ indices $k_0, k_1, \ldots, k_n \in \mathbb N _0$ such that $a_{k_i}>0$ for $i= 0,\dots, n$.
We consider the polynomial $f_d(x)=\sum_{k=0}^d a_kx^k$ for $d\geq \max\{k_0, \dots, k_n\}$ and let $\varepsilon_{i, d}=f(x_i)-f_d(x_i).$ Since $|f_d(x_i)-f(x_i)|\to 0$ as $d\to \infty,$ we have $\varepsilon_{i,d}\to 0$ as $d\to \infty.$

We construct $P(x)=\sum_{k=0}^d b_kx^k$ as follows:
\begin{itemize}
\item For $k\geq d+1$ set $b_k=0$.
\item For $k\in [0:d]\setminus \{k_0, \ldots, k_n\},$ let $b_k=a_k.$
\item For $k\in \{k_0, \ldots, k_n\},$ we claim that there exist unique values $b_{k} = a_{k}+\delta_k$ such that $P(x_i) = f(x_i)$ for all $i$.
\end{itemize} 
We want to show that for $d$ sufficiently large, we have $|\delta_k| \leq a_{k}$, which implies that $b_{k} = a_{k} + \delta_k \geq 0$.
We must solve
\[
P(x_i)=\sum_{k=0}^d b_kx^k_i=f_d(x_i)+\sum_{j=0}^n\delta_{k_j}x_i^{k_j}=f(x_i),\]
or equivalently
\[\varepsilon_{i,d}=f(x_i)-f_d(x_i)=\sum_{j=0}^n\delta_{k_j}x_i^{k_j}.
\]
As such, we have a system of $n+1$ linear equations for the flexible $n+1$ coefficients. Write the generalized Vandermonde matrix $V=(x_i^{k_j})_{i,j=0,\ldots n},$ where without loss of generality, $0 \leq x_0<x_1<\ldots<x_n$ and $k_0<k_1<\ldots<k_n$. By \cite[Section 4.2]{pinkus2010totally}, $V$ is invertible since all data points are non-negative.

Let $\delta=(\delta_0, \delta_1,\ldots, \delta_n)^T$ and $\varepsilon=(\varepsilon_0, \varepsilon_1,\ldots, \varepsilon_n)^T.$
Our system of linear equations can then be written as
$V\delta=\varepsilon,$ and admits a unique solution $\delta=V^{-1}\varepsilon$ which satisfies 
\[
\|\delta\|_\infty \leq \|V^{-1}\|_{op}\|\eps_d\|_\infty \leq \left(\max_{0\leq i\leq n}\sum_{j=0}^n \big|(V^{-1})_{ij}\big|\right) \|\eps_d\|_\infty \leq \min_{0\leq i\leq n} |a_{k_i}|
\]
for the $\ell^\infty$-operator norm and assuming that $d$ is large enough.
\end{proof}

The degree of the polynomial with non-negative coefficients in the preceding proof is implicit, either as the degree of $f$ (if $f$ itself is a polynomial) or in the rate of convergence of the power series. We demonstrate that the degree may indeed be large even if we only require interpolation at $n=2$ points. 

\begin{example}[Minimum degree]\label{example minimum degree}
    Let $D\geq 1$ and $f(x)=\sum_{k=D}^\infty b_kx^k$ be an analytic function with nonnegative coefficients, so for each $k\in[D,\infty), b_k\geq 0.$
    If $f\not \equiv 0$, there does not exist a polynomial $p(x)=\sum_{k=0}^d a_kx^k$ of degree $d<D$ such that 
    \begin{enumerate}
        \item $a_k\geq 0$ for $k= 0,\dots, d$ and 
        \item $f(1)=p(1)$ and $f(\xi)=p(\xi)$ for some $\xi\in(0,1).$
    \end{enumerate}

    Assume for the sake of contradiction that such a polynomial $p$ existed. Since $f\not\equiv 0$, there exists $b_k>0$ for some $k$ and thus $f'>0$ in $(0,1]$. In particular, we have $0<f(\xi) < f(1)$. We observe that $\sum_{i=D}^\infty b_i=f(1)=p(1)=\sum_{i=0}^d a_i.$
    We then have 
    $$p(\xi)=\sum_{i=0}^da_i\xi^i
    \geq \sum_{i=0}^da_i\xi^d
    =\sum_{i=D}^\infty b_i\xi^d 
    > \sum_{i=D}^\infty b_i\xi^i
    =f(\xi),$$
    which is a contradiction to $f(\xi)=p(\xi).$ Thus, even with two data points, the degree of $p$ may have to be arbitrarily high to allow for an interpolant with non-negative coefficients in the monomial basis.
\end{example}

\begin{example}[Uniqueness and non-uniqueness]
For a given set $\mathcal X$, the interpolant is usually {\em not} unique: For $f(x) = \exp(x)$ for instance, we can select an arbitrary subset of the coefficients to modify and let the sum run to any $d\geq d_0$ with $d_0$ depending only on the selected arbitrary subset.

By commonly exploited links between the $\ell^1$-norm and sparsity, we can guarantee that there exists an interpolating polynomial $P$ with $\ell^1$-minimal coefficients in the monomial basis such that at most $|\mathcal X|$ coefficients are non-zero (although not necessarily with any fixed degree according to Example \ref{example minimum degree}). In Corollary \ref{corollary equal for polynomial} below, we will show conversely that the minimum norm interpolant (in the coefficient $\ell^1$-sense) is indeed unique if $f$ is a polynomial and we have sufficiently many sampling points.
\end{example}

Summarizing the positive results, there exists $d_0\in \mathbb N$ (depending on both $f$ and the sample set $\mathcal X$) such that any minimum {\em coefficient $\ell^1$-norm} polynomial interpolant $P\in \mathcal P_d$ for $d\geq d_0$ of the input output pairs $x_i, f(x_i)$ is a polynomial with non-negative coefficients. It remains to show that those resemble the function $f$.

\subsection{Geometry of interpolants}\label{section ell1 convergence}

Recall the following terminology.

\begin{definition}
Let $I \subseteq \R$ be an (open or closed) interval. A function $f:I\to\mathbb R$ is called {\em absolutely monotonic} on $I$ if $f$ is continuous on $I$, infinitely smooth in $I^\circ$ and $f^{(n)}(x) \geq 0$ for all $n\in\mathbb N$ and $x\in I^\circ$.
\end{definition}

In particular, an analytic function $f$ whose power series representation $f(x) = \sum_{k=0}^\infty a_kx^k$ around the origin converges on an interval $I$ and satisfies $a_k\geq 0$ for all $k\in \mathbb N_0$ is absolutely monotonic on $I \cap[0,\infty)$. Conversely, any absolutely monotonic function is analytic \cite[Theorem 2.1]{szabo2025completely} and has non-negative power series coefficients when expanded around any point in its domain.

Thus, if $\mathcal X\subseteq [0,\infty)$ is a finite dataset, $1\in\mathcal X$ and $f$ is absolutely monotonic and represented by a convergent power series on an interval containing $\mathcal X$, then (for sufficiently high degree), a polynomial $P$ which interpolates $f$ on $\mathcal X$ and has minimal coefficients (in the sense of $\ell^1$-norm in the monomial basis) is also absolutely monotonic on the convex hull of $\mathcal X\cup \{0\}$. Since there is no question of power series convergence for polynomials, $P$ is in fact absolutely monotonic on $[0,\infty)$.

We show that $f,g$ taking the same values on a dataset $\mathcal X$ must be close also in between data points if $f$ and $g$ have good geometric properties like monotonicity, convexity, or absolute monotonicity.

To obtain rates of convergence, we compare high degree absolutely monotonic polynomial interpolants to `Lagrange polynomial' interpolants, i.e.\ to the unique minimal degree interpolant of a dataset $\{(x_i, f(x_i)) : i=0, \dots, n\}$ and its subsets.

\begin{lemma}\label{lemma sign lagrange polynomials}
Let $f:[a,b]\to\R$ be a continuous function, $k+1$ times continuously differentiable in $(a,b)$ such that $f^{(k+1)}(\xi)\geq 0$ for all $\xi \in (a,b)$.
Let $x_0,\dots, x_k$ distinct points in $[a,b]$ and $P$ the unique polynomial of degree $k$ which satisfies $P(x_i) = f(x_i)$ for $i=0,\dots, k$. Then if $f(x) \neq P(x)$, we have
\[
\sign\big(f(x) -P(x)\big) = \sign\left(\prod_{i=0}^k (x-x_i)\right) \qquad\forall\ x\in (a,b) \setminus \{x_0,\dots, x_k\}.
\]
\end{lemma}

\begin{proof}
Let $x\in (a,b)$. By the standard identity of polynomial interpolation \cite[Theorem 2 in Chapter 6]{kincaid2009numerical}, we have
\[
f(x) - P(x) = \frac{f^{(k+1)}(\xi)}{(k+1)!} \,\prod_{i=0}^k (x - x_i)
\]
for some $\xi\in (a,b)$. The derivatives of $f$ may become infinite at $a,b$, but the proof by Rolle's Theorem/ the intermediate value theorem goes through regardless. Since $f^{(k+1)}\geq 0$ and $f(x) - P(x)\neq 0$ by assumption, the result follows. Notably, while the estimate is most commonly used when the point $x$ lies in the convex hull of the points $x_i$, this is not necessary.
\end{proof}

\begin{corollary}[Monotonic Functions]\label{corollary monotone}
Assume that $f,g:[a,b]\to\R$ are monotonic and $f(x_i) = g(x_i)$ or $i=1,\dots,n$. Then $f(x_{i-1}) \leq f(x),\ g(x)\leq f(x_i)$ for $i=1,\dots, n$ and
\begin{enumerate}
    \item $\|f-g\|_{L^1(a,b)} \leq \frac{f(b)-f(a)}n$.
    \item If $0 \leq f'\leq L$ on $(a,b)$ and $x_{i}-x_{i-1}\leq K/n$ for $i\in [1:n]$, then $\|f-g\|_{L^1(a,b)} \leq \frac{KL}n$.
\end{enumerate}
\end{corollary}

\begin{proof}
The proof is essentially trivial, but we demonstrate how to derive the $L^\infty$-estimate from Lemma \ref{lemma sign lagrange polynomials} for future use. Consider the zeroth degree polynomials $P_i(x) = f(x_i)$ which match $f$ at the point $x_i$. Then by Lemma \ref{lemma sign lagrange polynomials}, either $f(x) = P(x)$ or
\[
\sign\big(f(x) - P_i(x)\big) = \sign\big( x-x_i\big) \qquad \Ra\quad f(x) - P_i(x) \begin{cases} < 0 &\text{if }x<x_i\\ >0 &\text{if }x>x_i.\end{cases}
\]
Putting both cases together, non-strict inequality follows.
If $f, g$ are monotonic functions which coincide at $x_i$ for $i\in [0:n]$, we get from the error estimate of polynomial interpolation
\begin{align*}
\left|f(x) - g(x)\right| &\leq \big|P_i(x) - P_{i-1}(x)\big| \leq \big|f(x) - P_i(x)\big| + \big|f(x) - P_{i-1}(x)\big| \\
    &\leq \frac{\sup_{\xi \in (x_{i-1},x_i)}f'(\xi)}{1!}\big(|x-x_i| + |x-x_{i-1}|\big) = (x_i - x_{i-1})\sup_{\xi \in (x_{i-1},x_i)}f'(\xi)\leq \frac{KL}n.
\end{align*}
The error estimate for polynomial interpolation is geared towards $L^\infty$-convergence guarantees, so we derive the $L^1$-estimate by hand: Since $f(x_{i-1}) \leq f(x),\ g(x) \leq f(x_i)$ for $x\in(x_{i-1}, x_i)$, we find that
\begin{align*}
    \int_a^b|f(x)-g(x)|dx &= \sum_{i=0}^{n}\int_{x_i}^{x_{i+1}}|f(x)-g(x)| \dx \\
    &\leq \sum_{i=0}^{n}\int_{x_i}^{x_{i+1}} (f(x_{i+1})-f(x_i)) \dx \\
    &= \sum_{i=0}^{n} (f(x_{i+1})-f(x_i)) (x_{i+1}-x_i) \\
    &= \frac{K}{n}\sum_{i=0}^{n} (f(x_{i+1})-f(x_i)) \\
    &= \frac{K(f(b)-f(a))}{n},
\end{align*}
since the sum telescopes.
\end{proof}

The $L^1$-convergence estimate is sufficient to prove Corollary \ref{comparison corollary}.

\begin{proof}[Proof of Corollary \ref{comparison corollary}]
Assume that $d_0>n$ is so large that the interpolating polynomials $P_d$ have non-negative coefficients for $d\geq d_0$, using Section \ref{section ell^1-norm existence} and the fact that $1\in\mathcal X$. In particular, both $f$ and $P_d$ are monotonically increasing. By Corollary \ref{corollary monotone}, we find that
\begin{align*}
\|f-P_d\|_{L^2(-1,1)}^2 &\leq \|f-P_d\|_{L^1(-1,1)}\|f-P_d\|_{L^\infty(-1,1)}\\
    &\leq K\,\frac{f(b) - f(a)}n\cdot \max_i \big(f(x_{i+1}) - f(x_i)\big)\\
    &\leq K \,\frac{\big(f(b) - f(a)\big)^2}n\\
    &< \int_{-1}^1 \big(f(x) - \langle f\rangle\big)^2\dx\\
    &\leq \|f-P_0\|_{L^2(-1,1)}^2.
\end{align*}
if $n > n_0$ is large enough.
\end{proof}

Naturally, the estimates are coarse and the condition on $n$ could be relaxed to
\[
n > K \sqrt{\frac{L\big(f(b) - f(a)\big)}{\|f- \langle f\rangle\|_{L^2(-1,1)}^2}}
\]
if $f$ is $L$-Lipschitz. Exploiting absolute monotonicity rather than just monotonicity and strengthening regularity assumptions on $f$, the value of $n$ could be lowered further. The assumption $x_{i+1}-x_i \leq K/n$ holds
\begin{itemize}
    \item with $K =b-a$ if the points are equidistant.
    \item with $K = \pi(b-a)/2 <2(b-a)$ if the points are Chebyshev nodes $x_i = \frac{b+a}2 + \frac{b-a}2\,\cos(\pi\, i/n)$ since $\cos$ is $1$-Lipschitz.
    \item with $n$-dependent $K_n = (b-a)(\log n + C)$ if the points are $x_1, \dots, x_{n-1}$ are independent uniform random samples in $(a,b)$ and $x_0 = a, x_n = b$ are fixed. The estimate only holds in expectation or with high probability -- see \cite{holst1980lengths} for precise statement and derivation on $(a,b) = (0,1)$. The constant $C$ must be large enough; at least Euler's constant $\gamma$.
\end{itemize}

Lemma \ref{lemma sign lagrange polynomials} can be used to derive the zeroth order convexity condition from the second order condition, as well as an approximate first order convexity condition where a derivative is approximated by a difference quotient.

\begin{corollary}[Convex Functions]\label{corollary convex}
Assume that $f''\geq 0$, $x_{i-1}<x_i < x_{i+1}$ and $x\in (x_{i}, x_{i+1})$. Then
\[
L(x):= f(x_i) + \frac{f(x_i) - f(x_{i-1})}{x_i - x_{i-1}}(x-x_i) \leq f(x) \leq \frac{x-x_i}{x_{i+1} - x_i}\, f(x_i) + \frac{x_{i+1}-x}{x_{i+1}-x_i}\,f(x_{i+1}) =: U(x).
\]
\end{corollary}

\begin{proof}
The functions $L, U$ are the unique linear polynomials which coincide with $f$ at $x_{i-1}, x_i$ and $x_i, x_{i+1}$ respectively. The polynomial $(x-x_i)(x-x_{i-1})$ is positive on $(x_i, x_{i+1}) \subseteq (x_i, \infty)$ while $(x-x_i) (x- x_{i+1})< 0$ for $x\in (x_i, x_{i+1})$. The result follows from Lemma \ref{lemma sign lagrange polynomials} as in the proof of Corollary \ref{corollary monotone}.
\end{proof}

The usual first order convexity condition can be derived taking $x_{i-1}\nearrow x_i$ if desired. Alternatively, the lower bound
\[
\tilde L(x) = f(x_{i+1}) + \frac{f(x_{i+2}) - f(x_{i+1})}{x_{i+2} - x_{i}}(x-x_{i+1})
\]
could be used for $x_{i+2}>x_{i+1}$ since $x\in(x_i, x_{i+1})$ does not lie between the zeros of $(x-x_{i+1})(x-x_{i+2})$. Under additional regularity assumptions, Corollary \ref{corollary convex} could be used to obtain quantitative rates of convergence, much like Corollary \ref{corollary monotone}. We proceed more generally.

\begin{corollary}[Absolutely Monotonic Functions]
Let $f: [a,b]\to \R$ absolutely monotonic. Let $\mathcal X = \{x_0, \dots, x_n\}$. Given and $I \subseteq [0:n]$, denote by $P_I$ the Lagrange polynomial for $\mathcal Z_I = \{ (x_i, f(x_i)) : i\in I\}$ and by $Q_I = \prod_{i\in I} (x-x_i)$. Then
\[
\max\left\{P_I(x) : I\subseteq [0:n]\text{ s.t. }Q_I(x) \geq 0\right\}
     \ \leq \ f(x) \ \leq \
\min\left\{P_I(x) : I\subseteq [0:n]\text{ s.t. }Q_I(x) \leq 0\right\}.
\]
\end{corollary}

\begin{proof}
We proved that if $Q_I(x) \leq 0$, then $f(x) - P_I(x) \leq 0$, so $f(x) \leq P_I(x)$. It suffices to take the maximum and minimum pointwise.
\end{proof}

It is easy to see that the index sets for both the upper and lower bound are non-empty if $n\geq 2$. So, in a quantitative way, high degree absolutely monotonic polynomial interpolants do not worse than low degree polynomial interpolants in a precise and localizable sense. 

As a particular application, we find weaker bounds which are easier to interpret and only involve sets $I$ of $k$ points, much like we did above for monotone and convex functions.

\begin{corollary}\label{corollary absolutely monotonic max degree}
Assume that $f, g:[a,b]\to\R$ are absolutely monotonic. Let $\mathcal X = \{x_0, \dots, x_n\}$ be a set of (ordered) distinct points in $[a,b]$ such that $f(x) = g(x)$ if $x\in \mathcal X$. Then
\[
\big|f(x) - g(x)\big| \leq (x_n -x_0)\,\frac{f^{(n)}(x_n)}{n!} \ \prod_{i=1}^{n-1}|x-x_i| \qquad \forall\ x\in[x_0, x_n].
\]
\end{corollary}

\begin{proof}
We consider the index sets $I_1 = [0:n-1]$ and $I_2 = [1:n]$ as competitors for the minimum and maximum. In each sub-interval $(x_i, x_{i+1})$ polynomials $Q_{I_1}$ and $Q_{I_2}$ have opposite signs since the number of factors for which $x< x_i$ is odd for one and even for the other. Thus one of the polynomials $P_{I_1}$, $P_{I_2}$ serves as an upper bound for both $f$ and $g$ while the second serves as a lower bound for both functions. So, since $f^{(n)}\geq 0$, we have
\begin{align*}
\big|f(x) - g(x)\big| &\leq \big|P_{I_1}(x) - P_{I_2}(x)\big|\\
    &\leq \big|P_{I_1}(x) - f(x)\big| + \big|P_{I_2}(x) - f(x)\big|\\
    &\leq \frac{\max_{\xi\in(x_0,x_{n-1})}f^{(n)}(\xi)}{n!} \prod_{i=0}^{n-1}|x-x_i| + \frac{\max_{\xi\in(x_1,x_{n})}f^{(n)}(\xi)}{n!} \prod_{i=1}^{n}|x-x_i|\\
    &\leq \frac{f^{(n)}(x_n)}{n!} \left(\prod_{i=1}^{n-1}|x-x_i|\right)\big(|x_n-x| + |x-x_0|\big)
\end{align*}
since $f^{(n)}$ is monotone increasing.
\end{proof}

Crucially, the closeness depends only on the regularity of $f$, but not $g$. If $f(x) = \sum_{k=0}^\infty a_kx^k$ with $a\in \ell^1$, then the radius of convergence for the analytic functions $f^{(m)}$ is at least 1 for all $m\in \mathbb N_0$, but convergence up to the boundary is only guaranteed for $m=0$ as the dilogarithm function $Li_2(x) = \sum_{k=1}^\infty \frac{x^k}{k^2}$ illustrates. If $f$ is known to be more regular and have finite derivatives of degree $m$, the same is not automatically true for other minimum norm interpolants of the same data as in Section \ref{section ell^1-norm existence}.

As an application we deduce that if $f$ is a polynomial and we are given sufficient sample points, we can identify $f$ uniquely.

\begin{corollary}[Identity for polynomials]\label{corollary equal for polynomial}
In addition to the assumptions of Corollary \ref{corollary absolutely monotonic max degree}, assume that $f$ is a polynomial of degree at most $n-1$. If $g=P$ is another absolutely monotonic polynomial and $f(x_i) = P(x_i)$ for all $i =0, \dots, n$, then $P\equiv f$.
\end{corollary}

In particular, the minimum $\ell^1$-norm interpolant of an absolutely monotonic power series is unique if and only if the power series is in fact a polynomial (assuming sufficient data and high enough degree).

\begin{proof}
We have $f^{(n)}(\xi) = 0$ for all $\xi$, so $f\equiv g$ on $(x_0, x_n)$. The identity holds on the potentially larger interval $[a,b]$ since $f, P$ are analytic and coincide in a non-empty open set \cite[Theorem 2.4.8]{stein2010complex}.
\end{proof}

\begin{corollary}[Error estimate for Chebyshev points]
Let $[a,b]\subset [0,\infty)$, $f:[a,b]\to\R$ an analytic function with convergent power series $f(x) = \sum_{k=0}^\infty a_kx^k$ on $[a,b]$ where $a_k\geq 0$ for all $k\in\mathbb N_0$. Let $x_0, \dots, x_n$ Chebyshev nodes (of the first kind) in $[a,b]$. If $f^{(n)}(b)<\infty$ and $g$ is an absolutely monotonic function such that $f(x_i) = g(x_i)$ for $i=0,\dots, n$, then
\[
|f(x) - g(x)| \leq \frac{f^{(n)}(b)}{n!}\,\left(\frac{b-a}4\right)^{n}\,\frac{1}{(b-x)(x-a)}.
\]
\end{corollary}

The corollary in particular applies when $g$ is an interpolating polynomial of high degree and with $\ell^1$-norm minmial monomial basis coefficients.

\begin{proof}
The estimate follows directly by noting that $f^{(n)}$ is monotone increasing, so $\sup_{\xi \in (a,b)} f^{(n)}(\xi) = f^{(n)}(b)$ and
\[
\left|\prod_{i=1}^{n-1}(x-x_i)\right| = \frac{\left|\prod_{i=0}^n(x-x_i)\right|}{(b-x)(x-a)} = \frac{2^{-n}\,|T_{n;a,b}(x)|}{(b-x)(x-a)}\leq \frac{2^{-n}\,\big((b-a)/2\big)^n}{(b-x)(x-a)}
\]
where $T_{n;a,b}$ is the $n$-th Chebyshev polynomial of the first kind on $[a,b]$.
\end{proof}

Despite the use of Chebyshev points, the error estimate deteriorates at the boundary due to several crude approximations. Still, we find that at least in the interior of $(a,b)$ we should expect {\em exponentially fast} convergence of $P_{d,n}\to f$ as the number of data points increases, assuming that $d>n$ is large enough.

Next, we prove {\em polynomial order convergence} also at the interval boundary. To this end, we consider competitors far below the interpolation threshold and general data points. If $f^{(n)}$ remains uniformly bounded in $L^\infty(a,b)$ (e.g.\ for the exponential function), then we establish convergence of any polynomial order.

\begin{proof}[Proof of Theorem \ref{main theorem} for non-negative data]
The $L^1$-estimate follows directly from our proof of Corollary \ref{comparison corollary}.

For the $L^\infty$-estimate, note that if $x\in (0, x_J)$ and $J\geq m+1$, then we can find two index $I=\{j, \dots, j+m-1\}$ and $J=\{j+1,\dots, j+m\}$ of cardinality $m$ each such that $x\in [x_j, x_{j+m}] \subseteq[x_0, x_J]$. Without loss of generality $x\notin \{x_i: i\in [0:J]\}$. The polynomials $Q_I, Q_J$ take opposite signs at $x$, so $P_I \leq f, g\leq P_J$ for any two absolutely monotonic functions $f,g$ which coincide at all $x_j, \dots, x_{j+m}$. In particular, this applies to all minimum norm interpolants $P$ if $d$ is large enough. We conclude that
\[
|f-P| \leq |P_I - P_J| \leq |f-P_I| + |f-P_J| \leq \max_{\xi\in (0,\zeta)}\frac{f^{(m)}(\xi)}{m!}\left(\prod_{k=0}^{m-1}(x-x_{j+k}) +\prod_{k=1}^m (x-x_{j+k})\right).
\]
Since $f$ is absolutely monotonic, $f^{(m)}$ is increasing, i.e.\ the maximum is attained at $\zeta$. For the product, we note that the closed point to $x$ is at most $K/n$ removed, the next point $2K/n$, and so on. Thus either product has size at most $m!\,(K/n)^m$ as we multiply $m$ factor. The second claim of the Theorem follows immediately by considering $\zeta = 1$.

The third claim follows from the observation that
\begin{align*}
f^{(m)}(\zeta) &= \sum_{k=0}^\infty a_k\left(\prod_{j=0}^{m-1} (k-j)\right)\zeta^{k-m}
    \leq \|a\|_1 \sum_{k=0}^\infty \left(\prod_{j=0}^{m-1} (k-j)\right)\zeta^{k-m}
    = f(1)\frac{d^m}{d\zeta^m}\sum_{k=0}^\infty \zeta^k = \frac{m!\,f(1)}{(1-\zeta)^{m+1}}.
\end{align*}
\end{proof}

In general, the bound on the products $\prod_j (x-x_{k+j})$ is overly pessimistic: If $x$ is roughly in the middle of the interval, we can bound the product by a term scaling only like $(\lceil m/2\rceil !)^2$ rather than $m!$. Also the bound by $f(1)$ is pessimistic, bounding $\|a\|_\infty$ by $\|a\|_1$.

\begin{figure}
    \centering
    \begin{subfigure}[b]{0.32\textwidth}
        \centering
        \includegraphics[width=\linewidth]{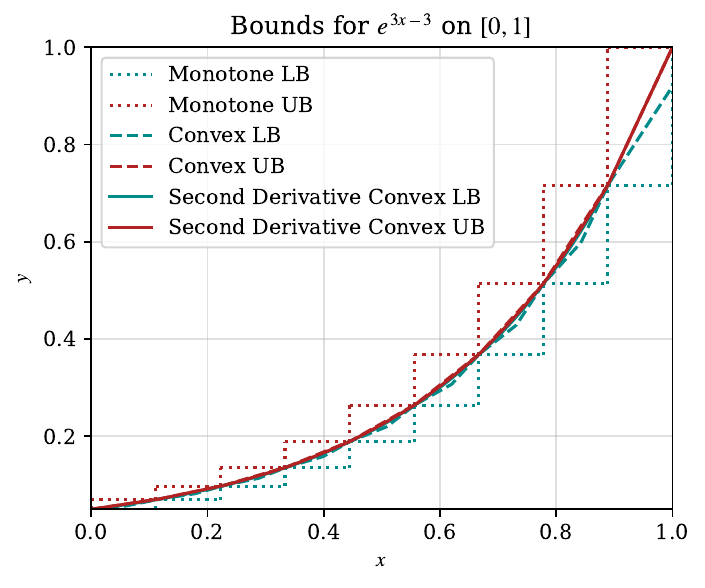}
    \end{subfigure}
    \hfill
    \begin{subfigure}[b]{0.32\textwidth}
        \centering
        \includegraphics[width=\linewidth]{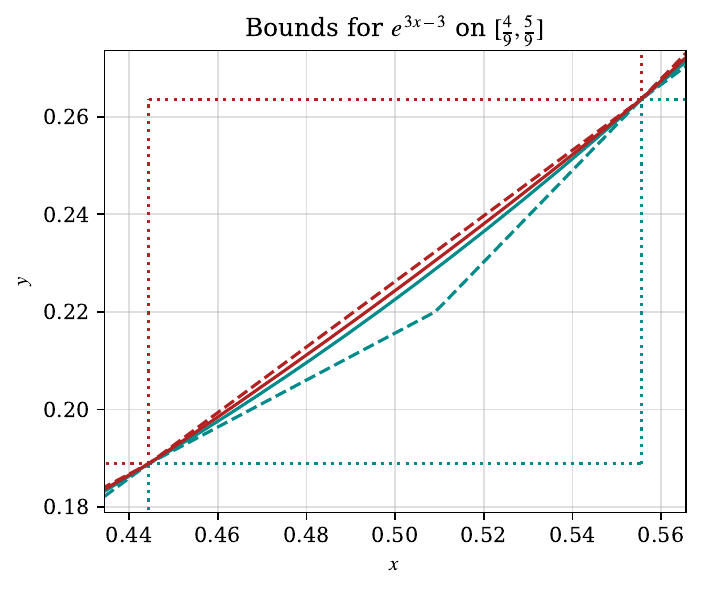}
    \end{subfigure}
    \hfill
    \begin{subfigure}[b]{0.32\textwidth}
        \centering
        \includegraphics[width=\linewidth]{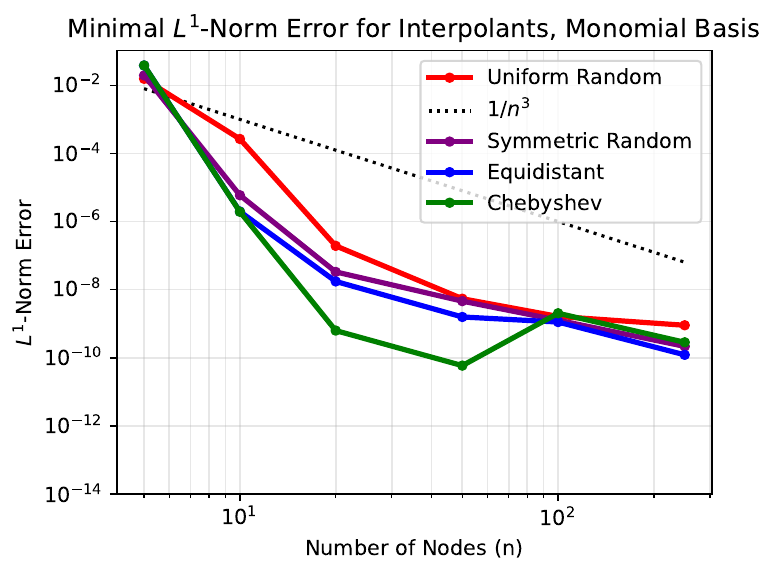}
    \end{subfigure}
    
    \vspace{2ex}

    \begin{subfigure}[b]{0.32\textwidth}
        \centering
        \includegraphics[width=\linewidth]{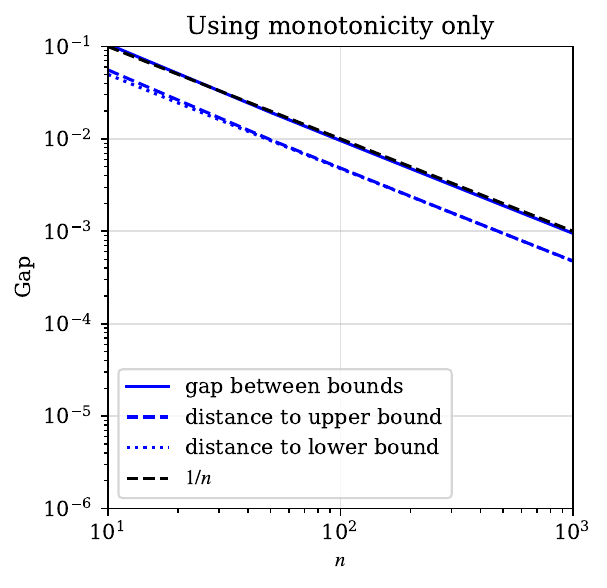}
    \end{subfigure}
    \hfill
    \begin{subfigure}[b]{0.32\textwidth}
        \centering
        \includegraphics[width=\linewidth]{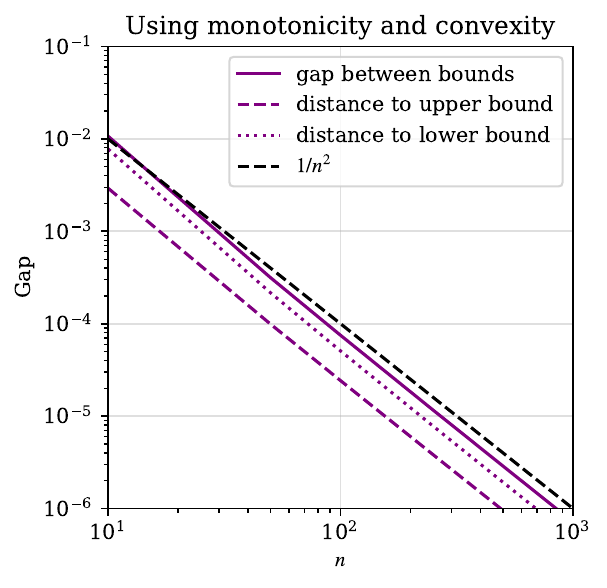}
    \end{subfigure}
    \hfill
    \begin{subfigure}[b]{0.32\textwidth}
        \centering
        \includegraphics[width=\linewidth]{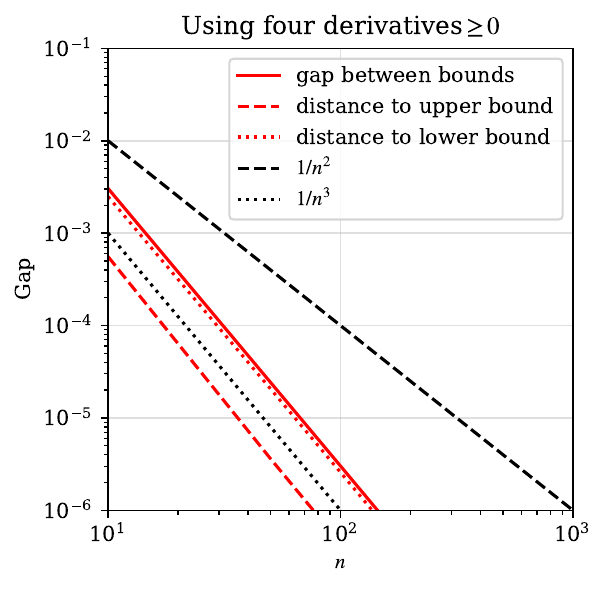}
    \end{subfigure}

    \caption{\label{figure geometrically constrained}
    {\bf First row:} The left $2$ graphs show the bounds in which interpolating functions with the monomial basis may exist with $10$ equidistant points in $[0,1]$ generated by the function $\exp(3x-3)$ using monotonicity, convexity, or positivity of three derivatives. The left graph is of the whole interval $[0,1]$ while the middle graph is zoomed in to $[\frac{4}{9}, \frac{5}9].$ The right graph shows the rates of convergence of a degree $d = 1200$ interpolating polynomial given $n \in \{5, 10, 20, 50, 100, 250\}$ (equidistant, Chebyshev, uniform random) data points  for the function $\exp(3x-3)$ in $[-1,1]$. For sufficiently large problems, the accuracy of the linear program solver appears to be governing approximation accuracy.
    {\bf Second row:} Convergence of bounds given monotonicity, monotonicity and convexity, and three derivatives being nonnegative respectively and the associated rates of convergence for comparison.}
\end{figure}

\subsection{Balanced data}

We can consider absolutely monotonic functions on general intervals $[a,b]$, but to link the results to polynomial interpolation, we were forced to consider $\mathcal X\subset [a,b]\subset [0,\infty)$ in \ref{section ell1 convergence}. In contrast, we only required that $1\in\mathcal X$ and placed no other restrictions on the dataset in Section \ref{section ell^1-norm existence}. In this section, we generalize the results of Section \ref{section ell1 convergence} to datasets which are {\em balanced} about the origin, i.e.\ $x\in \mathcal X \LRa -x\in \mathcal X$.

\begin{lemma} \label{lemma for balanced existence}
Let $f(x) = \sum_{k=0}^\infty a_kx^k$ be a power series such that $a_k\geq 0$ for all $k$ and $\mathcal X\subset \R$ a finite set such that $1\in\mathcal X$ and $f$ converges on $\mathcal X$. Then there exist $d\in \N$, $b\in \R^{d+1}$ such that $b_k\geq 0$ for all $k$ and $P(x) := \sum_{k=0}^d b_k x^k$ satisfies $P(x) = f(x)$ for all $x\in \mathcal X$.
\end{lemma}

\begin{proof}
Building on Lemma \ref{existence lemma}, we present an existence result which avoids generalized Vandermonde matrices for datasets which may contain negative points. Such matrices generally fail to be invertible since we cannot fit an odd function using even powers of $x$ only, for instance.

Decompose $f$ into its even and odd part
\[
    f_e(x)=\frac{f(x)+f(-x)}2=\sum_{k=0}^\infty a_{2k}x^{2k},\qquad f_o(x)=\frac{f(x)-f(-x)}2=\sum_{k=0}^\infty a_{2k+1}x^{2k+1}.
\]
Crucially, both $f_o$ and $f_e$ are analytic functions with non-negative coefficients in their power series.
We apply Lemma \ref{existence lemma} to $f_e(x)$ and $f_o(x)$ on $(\mathcal X \cup -\mathcal X)\cap [0,\infty)$ to attain interpolating polynomials $P_e(x)=\sum_{k=0}^{d_1}b_kx^k$ for $f_e$ and $P_o(x)=\sum_{k=0}^{d_2}c_kx^k$ for $f_o$ respectively. Following the proof of Lemma \ref{existence lemma}, we can choose coefficients $b_k, c_k$ for $P_e, P_o$ which are non-zero only if the corresponding power series coefficient of $f_e$ or $f_o$ respectively is non-zero. Thus, $P_e, P_o$ are even and odd respectively since the odd/even terms in their expansions vanish. 

We find that $P_e(x_i) = P_e(-x_i) = f_e(-x_i) = f_e(x_i)$ if $x_i \in \mathcal X\cap (-\infty,0]$ is a data point, using the fact that $-x_i \in (-\mathcal X) \cap [0,\infty)$ is in the set on which we prescribed interpolation conditions, so $f_e(-x_i) = P_e(-x_i)$. In particular, $P_e, P_o$ interpolate $f_e, f_o$ over the entire dataset $\mathcal X$, not just $\mathcal X\cap [0,\infty)$. Consequently $P:= P_e+P_o$ interpolates $f$ on $\mathcal X$. 
\end{proof}

A power series with non-negative coefficients may be absolutely monotonic also on $(-\infty,0]$ such as $\exp$, have derivatives of alternating sign such as $\exp$ or $\cosh$, or behave in a variety of complicated ways. We therefore continue treating the even and odd part of $f$ separately, since we can study their approximation properties easily on the positive half-axis. To identify them accurately, we require symmetry of our dataset. 
Specifically, we deduce closeness of the interpolants to the target function under the stronger condition that $\mathcal X$ is {\em balanced}, i.e.\ that $-\mathcal X = \mathcal X$.

\begin{proof}[Proof of Theorem \ref{main theorem} for balanced data]
Consider the even and odd parts of functions
\[
f_e(x) = \frac{f(x) + f(-x)}2, \quad f_o(x) = \frac{f(x) - f(-x)}2, \quad P_e(x) = \frac{P(x) + P(-x)}2, \quad P_o(x) = \frac{P(x) - P(-x)}2.
\]
As in Lemma \ref{lemma for balanced existence}, we find that the monomial coefficients of $f_e, f_o, P_e$ and $P_o$ are non-negative. Furthermore, $P_e$ interpolates $f_e$ on the dataset $\mathcal X$ since $\mathcal X$ is balanced, so
\[
P_e(x) = \frac{P(x) + P(-x)}2 = \frac{f(x) + f(-x)}2 = f_e(x) \qquad \forall\ x\in\mathcal X.
\]
Thus 
\begin{align*}
\|f-P\|_{L^\infty(-\rho, \rho)} &= \|f_e + f_o - (P_e+P_o)\|_{L^\infty(-\rho,\rho)}\\
    &\leq \|f_e - P_e\|_{L^\infty(-\rho,\rho)} + \|f_o - P_o\|_{L^\infty(-\rho,\rho)}\\
    &=  \|f_e - P_e\|_{L^\infty(0,\rho)} + \|f_o - P_o\|_{L^\infty(0,\rho)}\\
    &\leq f_e^{(m)}(1) \,\left(\frac Kn\right)^m + f_o^{(m)}(1)\,\left(\frac Kn\right)^m\\
    &= f^{(m)}(1)\,\left(\frac Kn\right)^m
\end{align*}
if $\rho = 1$ by applying the bounds of Theorem \ref{main theorem} on $[0,1]$. The other bounds generalize similarly.
\end{proof}

With the guarantees extended to balanced data, we compare interpolation on equidistant nodes and Chebyshev nodes in $[-1,1]$ in Figure \ref{figure norm decrease with degree}.

\begin{figure}
    \centering
    \begin{subfigure}[b]{0.48\textwidth}
        \centering
        \includegraphics[width=\linewidth]{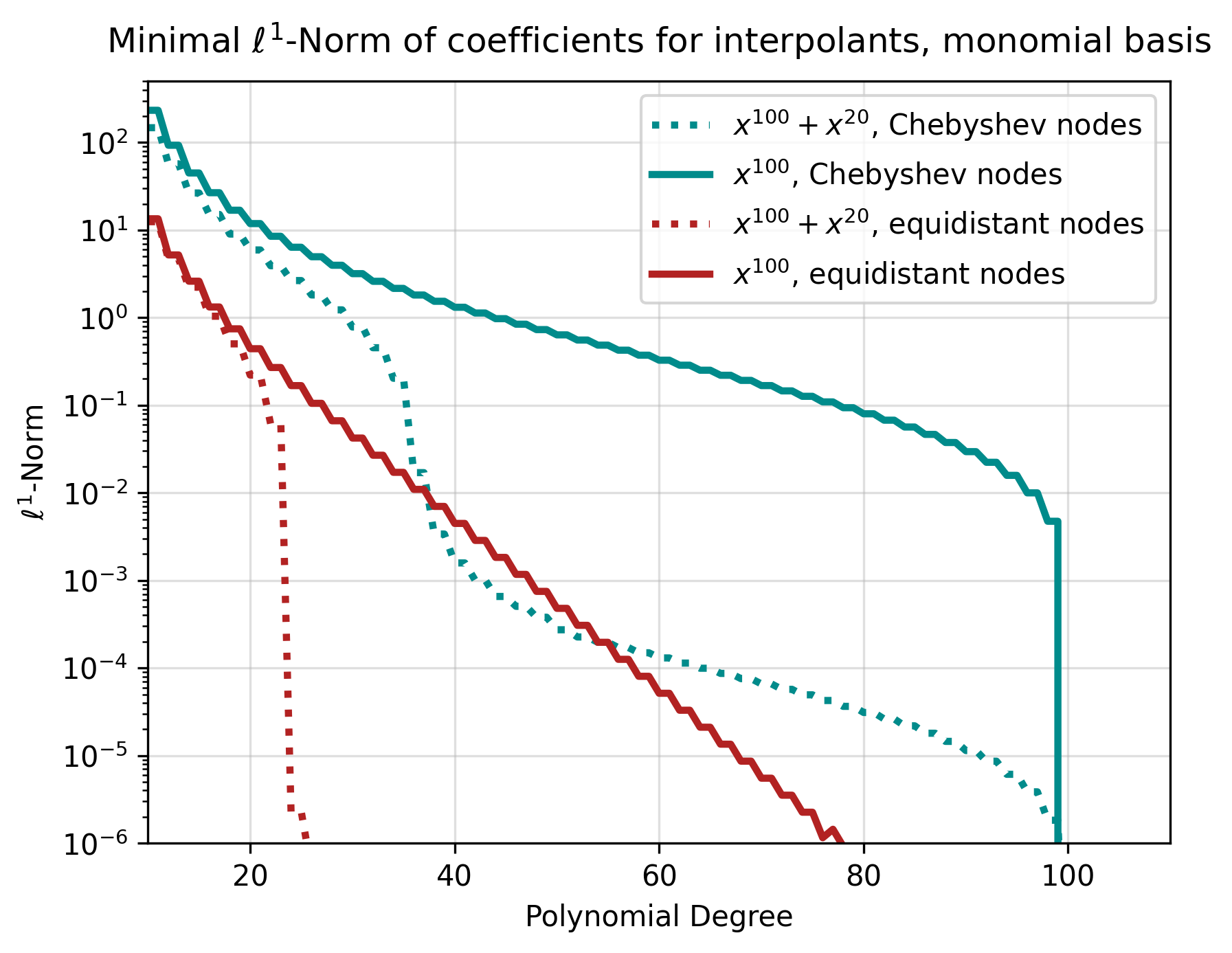}
    \end{subfigure}
    \hfill
    \begin{subfigure}[b]{0.48\textwidth}
        \centering
        \includegraphics[width=\linewidth]{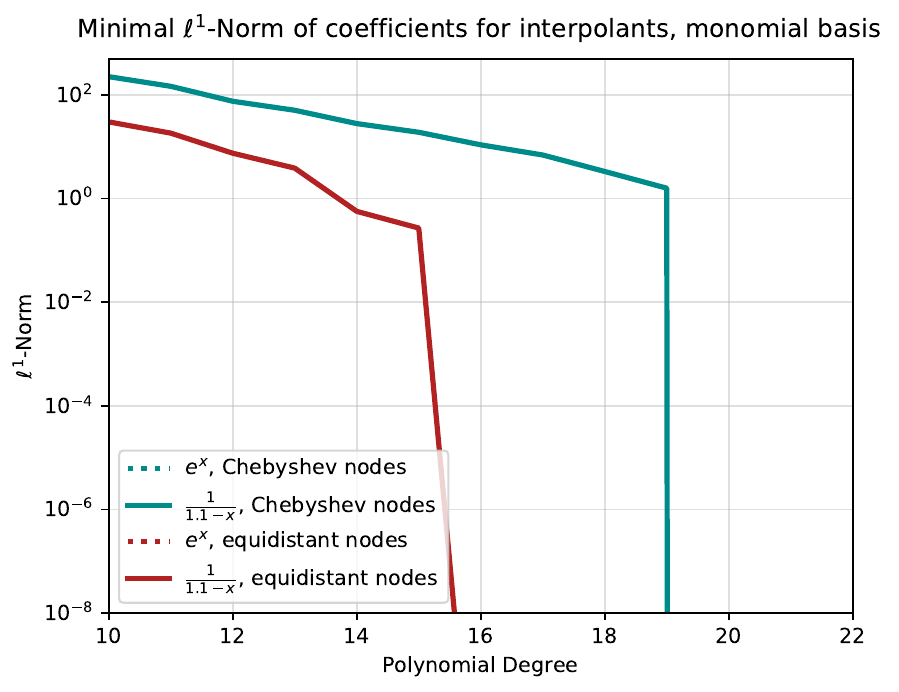}
    \end{subfigure}

    \caption{\label{figure norm decrease with degree}
    We plot $\min\{\|a\|_1 : a\in\R^{d+1} \text{ s.t. }\sum_{k=0}^da_kx^k_i = f(x_i)\} - f(1)$ as a function of $d$ on a semilog scale for functions $f_1(x) = x^{100}+ x^{10},\ f_2(x) = x^{100},\ f_3(x) = e^x$ and $f_4(x) = 1/(1.1-x)$. We explore interpolation constraints at both Chebyshev nodes and equidistant nodes, always with $11$ nodes in the interval $[-1,1]$. For $f_3(x)$, in the right plot, the values were too small to register on the scale of the plot, which was chosen above the threshold where optimizer errors for solving the linear program dominate.
    }
\end{figure}

\subsection{On relaxing the sign condition}

Naturally, Theorem \ref{main theorem} holds in spirit also if all coefficients are non-positive (replacing $f$ by $-f$) or if the even coefficients and odd coefficients have opposing signs and the dataset is balanced (replacing $f(x)$ by $\pm f(-x)$). In the spirit of the introduction it is natural to wonder whether a sign condition is in fact required, or if the coefficients indeed could be complex. We note that analytic difficulties are anticipated and point to the example of the sinc function in Figure \ref{figure monomial l1 and l2} to suggest that these may reflect true complications.

We note the following minimal asymptotic guarantee in the setting of taking $d\to+\infty$ separately before doing the same for $n$.

\begin{theorem}\label{theorem no sign condition}
Let $f:\mathbb D^\circ \to \mathbb C$ and $\mathcal Z_n = \{z_{n,0}, \dots, z_{n,n}\}$ a family of subsets of $\mathbb D^\circ$ and for $d\geq n$ denote
\[
a(d,n) \in \argmin\{ \|a\|_1 : a\in \mathbb C^{d+1} : P_a(z_i) = f(z_i)\:\text{ for } i \in[0:n]\}, \qquad P_a(z) = \sum_{k=0}^da_kz^k.
\]
Then, there exists $f_n(z) = \sum_{k=0}^\infty a^{(n)}_k z^k$ with $\|a^{(n)}\|\leq \liminf_{d\to\infty} \|a(d,n)\|_1 \leq \|a(n,n)\|_1$ such that a subsequence of $P_{a(d,n)}$ (not relabeled) converges to $f_n$ as $d\to\infty$
\begin{itemize}
    \item uniformly on $\overline{\mathbb D_r}$ for all $r<1$,
    \item pointwise at $\pm 1$, and
    \item weakly in $L^2(\partial \mathbb D)$.
\end{itemize}
Furthermore:
\begin{enumerate}
    \item If $f(z) = \sum_{k=0}^\infty c_kz^k$ with $\|c\|_1 <+\infty$, then $\|a^{(n)}\|_1\leq \|c\|_1$ for all $n\in\mathbb N$ and there exists $\tilde f(z) = \sum_{k=0}^\infty \tilde c_kz^k$ such that $\|\tilde c\|_1\leq \|c\|_1$ and $f^{(n)}\to g$ in the same notion of convergence.
    \item If $\mathcal Z_n$ converges in Hausdorff distance to a compact set $K\subseteq \overline{\mathbb D}$ which has an accumulation point in $\mathbb D^\circ$, then $\tilde f = f$.
\end{enumerate}
\end{theorem}

The existence of a subsequential Hausdorff limit $K$ is  guaranteed by compactness.

An analogous statement for power series with $\ell^2$-bounded coefficients is stated in Theorems \ref{theorem hardy} and \ref{theorem hardy 2}. The current proof is a slightly more technical version of that proof since $\ell^1(\mathbb C)$ is not reflexive while $\ell^2(\mathbb C)$ is a reflexive and separable Hilbert space.

The notion of convergence is natural: For instance, the function $f_d(z) = z^d$ interpolates the data pairs $(0,0), (1,1)$ for all $d\in \mathbb N_0$ with minimal coefficient $\ell^1$-norm 1. Naturally $f_d\to 0$ locally uniformly inside the unit disk but merely weakly on the boundary where $z^d = (e^{i\phi})^d = e^{id\phi}$ and $|z^d| \equiv 1$. The choice of boundary points $\pm 1$ was arbitrary and any countable set could have been selected, utilizing a diagonal sequence argument.
We observe in the third line of Figure \ref{figure monomial l1 and l2} that the focus on merely locally uniform convergence may well be required and that boundary layers may develop without a sign condition on the power series coefficients. In particular, non-asymptotic error bounds as in Theorem \ref{main theorem} appear out of reach.

\begin{proof}
{\bf Step 1. Limit $d\to\infty$.} We embed $\mathbb C^{d+1}$ into $\ell^1(\mathbb C)$ as by $Ia = (a_0, \dots, a_d, 0, 0, \dots)$ and note that $\|Ia(d,n)\|_1 = \|a(d,n)\|_1 \leq \|a(n,n)\|_1$ for all $d\geq n$ since the class over which we take the infimum increases with $d$. Since $\ell^1$ is the dual space of the separable Banach space $c_0$ \cite[Exercise 7.33]{einsiedler2017functional}, the weak* topology of $\ell^1$ is locally metrizable \cite[Theorem 3.28]{brezis2011functional}. By the Banach-Alaoglu theorem \cite[Theorem 3.16]{brezis2011functional} there exists a subsequence of $Ia(d,n)$ which converges in the weak* topology to a limit $a^{(n)}$ as $d\to\infty$ with $\|a^{(n)}\|_1 \leq \liminf_{d\to\infty}\|a(d,n)\|_1$.

The pointwise convergence of $P_{a(d,n)}(z) \to f^{(n)}(z) = \sum_{k=0}^\infty a^{(n)}_kz^k$ for $|z|<1$ follows immediately from the definition of weak* convergence since
\[
P_{a(d,n)}(z) = \sum_{k=0}^d a(d,n)_k z^k = \sum_{k=0}^\infty \big(Ia(d,n)\big)_k z^k = \left\langle Ia(d,n), z^{\mathbb N}\right\rangle_{\ell^1;c_0}
\]
where $z^{\mathbb N} := (1,z,z^2,\dots) \in c_0$ since $|z|<1$. We note that the polynomials $P_{a(d,n)}$ have Lipschitz constants with uniform bounds on $\mathbb D_r$ since
\[
\max_{z\in\mathbb D_r}\big|P_{a(d,n)}'(z)\big| \leq \sum_{k=0}^d \big|ka(d,n)_k\,r^{d-1}\big| \leq \max_{k\in\mathbb N} k r^{k-1} \sum_{k=0}^d |a(d,n)_k| \leq C\, \|a(d,n)\|_1 \leq C\, \|a(n,n)\|_1.
\]
By the Arzela-Ascoli theorem, we can upgrade the pointwise convergence to uniform convergence on disks $\overline{\mathbb D_r}= \{z\in\C: |z|\leq r\}$ for any $r<1$ since the uniform Lipschitz bound implies equicontinuity of the sequence and thus uniform convergence of a subsequence. As this holds for any subsequence as well, the entire sequence converges uniformly.

On the boundary of the disk, we note that 
\[
\big|P_{a(d,n)}(z)\big| =\left|\sum_{k=0}^d a(d,n)_k z^k\right| \leq \sum_{k=0}^d \big|a(d,n)_k\big| = \|a(d,n)\|_1 \leq \|a(n,n)\|_1,
\]
so $\|P_{a(d,n)}\|_{L^\infty(\partial \mathbb D)} \leq \|a(n,n)\|_1$ for all $d\geq n$. Since $\partial\mathbb D$ has finite measure, we conclude that for any $q<\infty$, we can find a further subsequence of $P_{a(d,n)}$ which converges weakly in $L^q(\partial \mathbb D)$. Similarly, we can take a further subsequence which converges pointwise at $-1$ and $1$ on the boundary (or indeed, on any given finite subset of $\partial \mathbb D$).

{\bf Step 2. Limit $n\to\infty$.} The functions 
\[
f_n(z) = \sum_{k=0}^\infty a^{(n)}_k z^k, \qquad f(z) = \sum_{k=0}^\infty c_k z^k
\]
coincide on the set $\mathcal Z_n$ and $f_n$ is approximated by a sequence of polynomials which satisfy a minimality condition on the $\ell^1$-norm. Arguing like in Lemma \ref{existence lemma}, if the degree of the polynomial is high enough, $\|a(d,n)\|_1$ is not much larger than $\|c\|_1$, so we conclude that $\|a^{(n)}\|_1 \leq \|c\|_1$. As in the previous step, $f_n$ has a subsequence which converges suitably to a limit $\tilde f$. 

Assume that $z_n\in \mathcal Z_n$ converges to a limit point $z\in \mathbb D^\circ$. Then $z\in \mathbb D_r$ for some $r<1$ and all but finitely many points $z_n$ lie in the slightly larger disk $\mathbb D_{(1+r)/2}$ on which $f_n \to \tilde f$ uniformly. Hence 
\[
\widetilde f(z) = \lim_{n\to \infty} \tilde f(z_n) = \lim_{n\to \infty} \big(\big\{\tilde f(z_n) - f_n(z_n)\big\} + f_n(z_n)\big) = \lim_{n\to\infty} f_n(z_n) = \lim_{n\to \infty} f(z_n) = f(z).
\]
The Hausdorff limit $K$ is the set of all limit points of convergent sequences $z_n\in\mathbb D$ such that $z_n\in \mathcal Z_n$, so we conclude that $\tilde f \equiv f$ on $K \cap \mathbb D^\circ$. Since both $f$ and $\tilde f$ are analytic, they coincide if $K\cap \mathbb D^\circ$ has an accumulation point inside $\mathbb D^\circ$ \cite[Theorem III.3.2]{busam2009complex}.
\end{proof}

\subsection{Numerical Illustration}

For all experiments, Moore-Penrose solutions (i.e.\ least squares best solutions of minimal $\ell^2$-norm) were found by the {\tt numpy.linalg.lstsq} method \cite{harris2020array}. In the monomial basis, $\ell^1$-minimal solutions were found using the {\tt scipy.optimize.linprog} method \cite{2020SciPy-NMeth}. In particular in cases where positive coefficient solutions do not exist, the problem was cast as the linear program
\[
\min\sum_{k=0}^d a_k + b_k \qquad\text{subject to}\quad \begin{pde}
a_k,\ b_k &\geq 0 &\forall\ k = 0, \dots, d\\
\sum_{k=0}^d (a_k-b_k)x_i^k &= y_i & \forall\ i = 0, \dots, n.
\end{pde}
\]
If positive coefficient solutions exist, they automatically correspond to minimum norm solutions and the $b_k$ variables can be eliminated as $b\equiv 0$.

In Figure \ref{figure monomial l1 and l2}, we visualize the interpolating polynomials with $\ell^1$- and $\ell^2$-minimal coefficients interpolating a set of 15 data points generated by a function $f$ for varying degrees. In the first line, we consider the dilogarithm function 
\[
Li_2(x) = - \int_0^x \frac{\log(1-s)}s\ds = \sum_{k=1}^\infty \frac{x^k}{k^2}
\]
which has summable positive coefficients, but whose derivative at the boundary point $x=1$ becomes infinite. The conditions of Theorem \ref{main theorem} are therefore met without satisfying strong additional bounds. The second function is a monomial which fails to be absolutely monotonic for $x<0$. The third function is a scaled sinc function $\sin(3x)/(3x)$, which has a globally convergent power series, but whose coefficients change sign. For this function, the assumptions of Theorem \ref{main theorem} are {\em not} met and indeed, we observe empirically polynomial interpolants of minimal norm develop boundary layers as the degree increases. We take this as an indication that the sign condition in Theorem \ref{main theorem} is not just an artifact of the proof, but indeed required for non-asymptotic error bounds, at least at the edge of the interval.

Out of general curiosity, we also include interpolating polynomials whose coefficients are $\ell^2$-minimal rather than $\ell^1$-minimal in Figures \ref{figure monomial l1 and l2} and \ref{figure overparametrized runge} -- see Section \ref{section monomial ell2} for further background.

\begin{figure}
    \centering
    \includegraphics[width=0.24\linewidth]{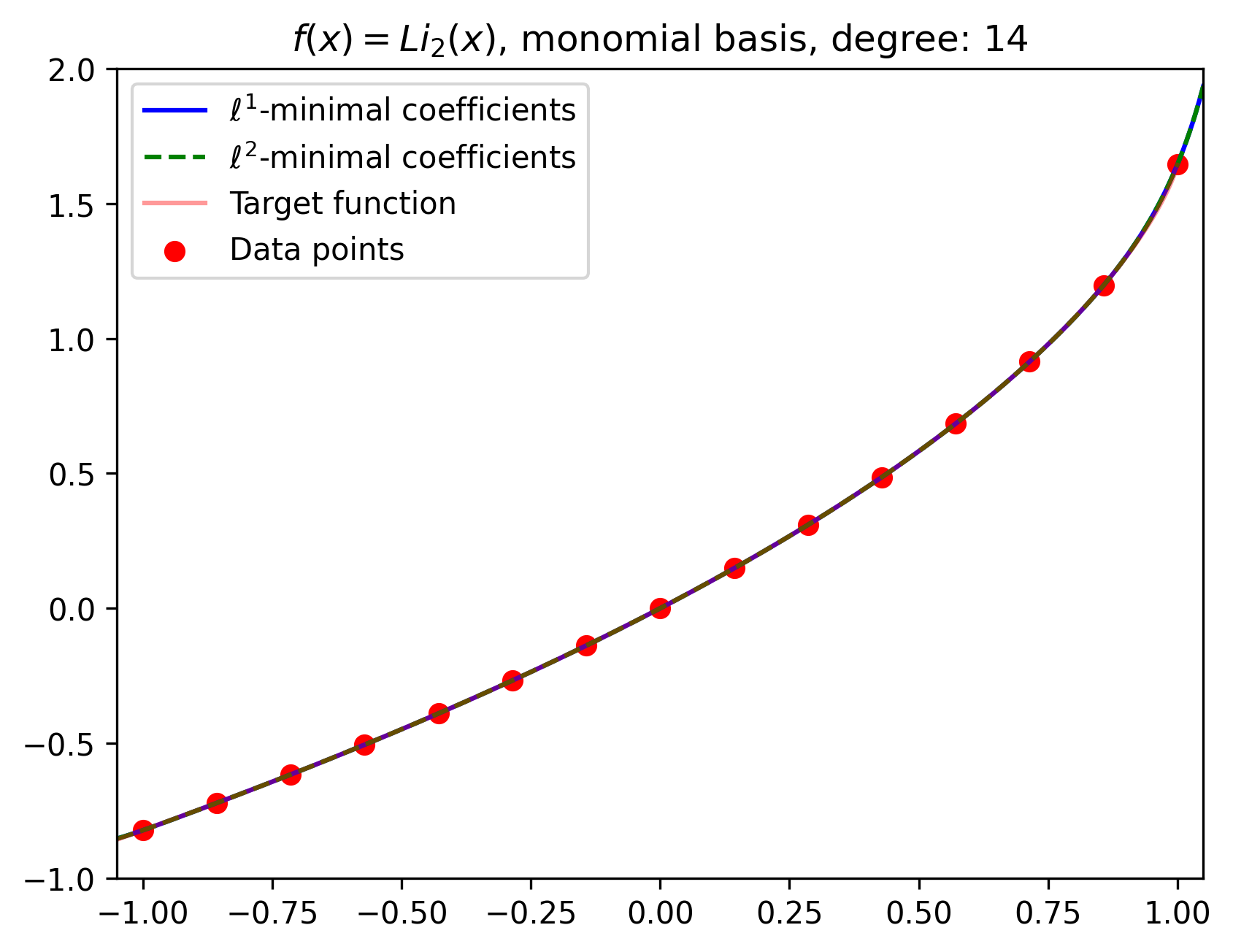}
    \includegraphics[width=0.24\linewidth]{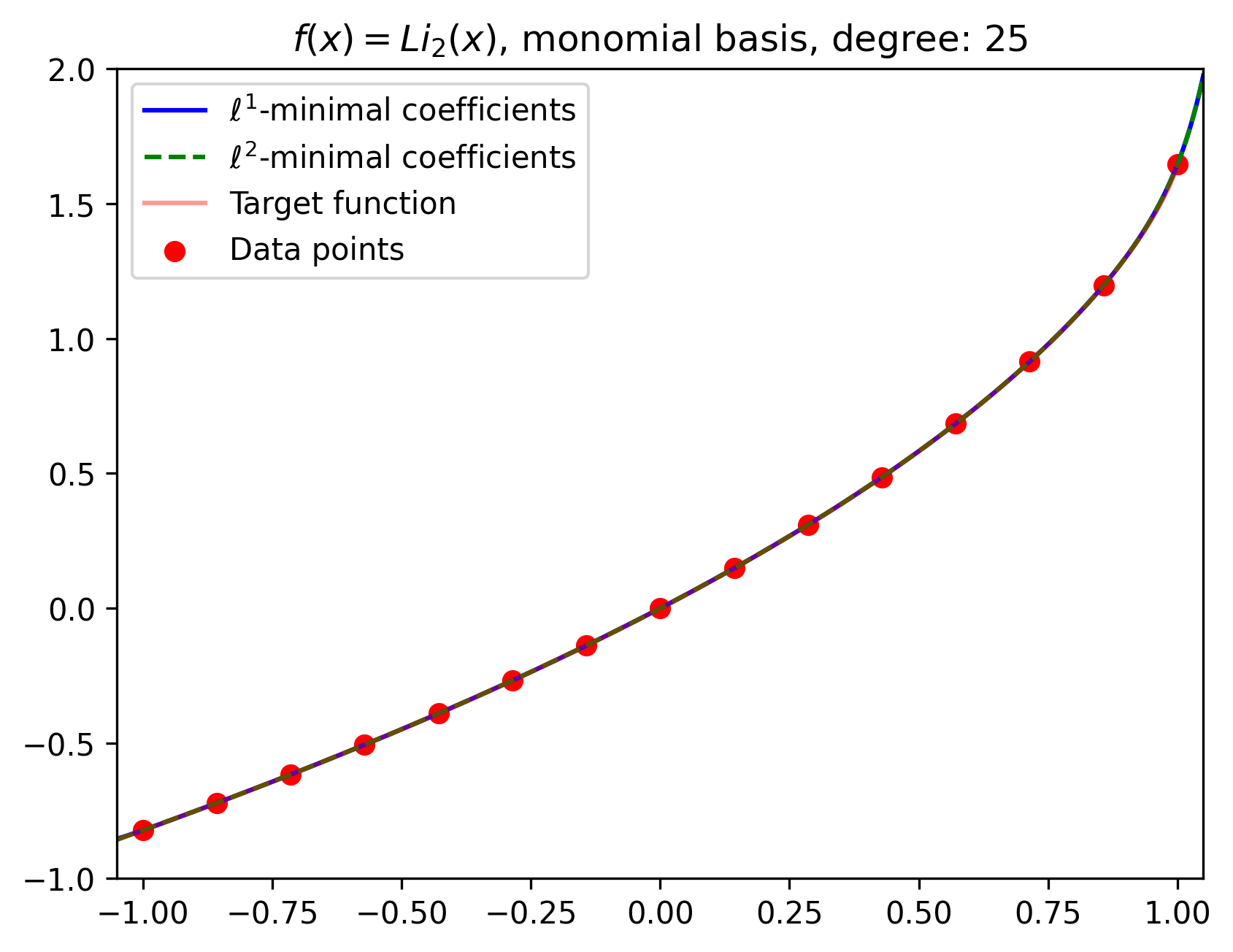}
    \includegraphics[width=0.24\linewidth]{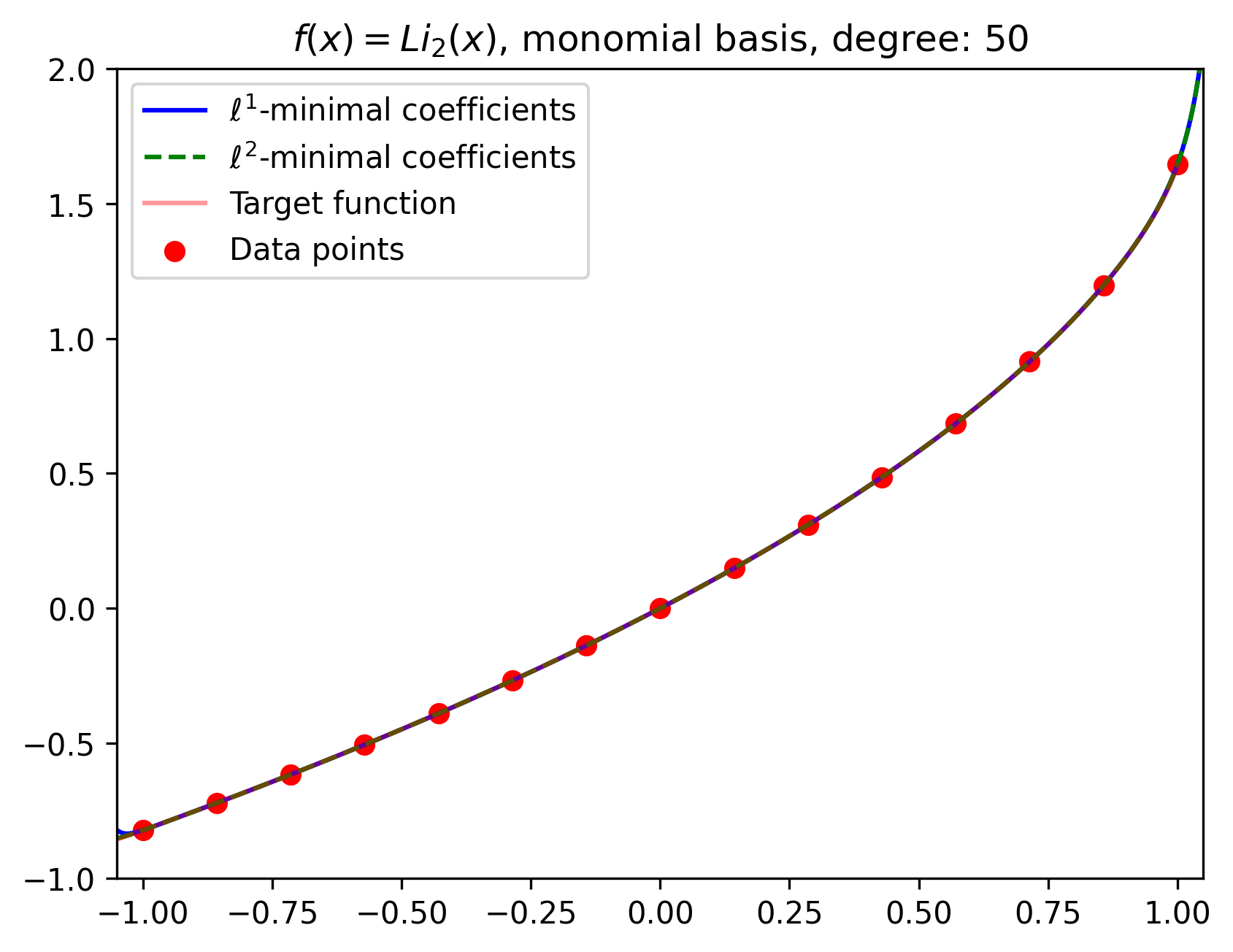}
    \includegraphics[width=0.24\linewidth]{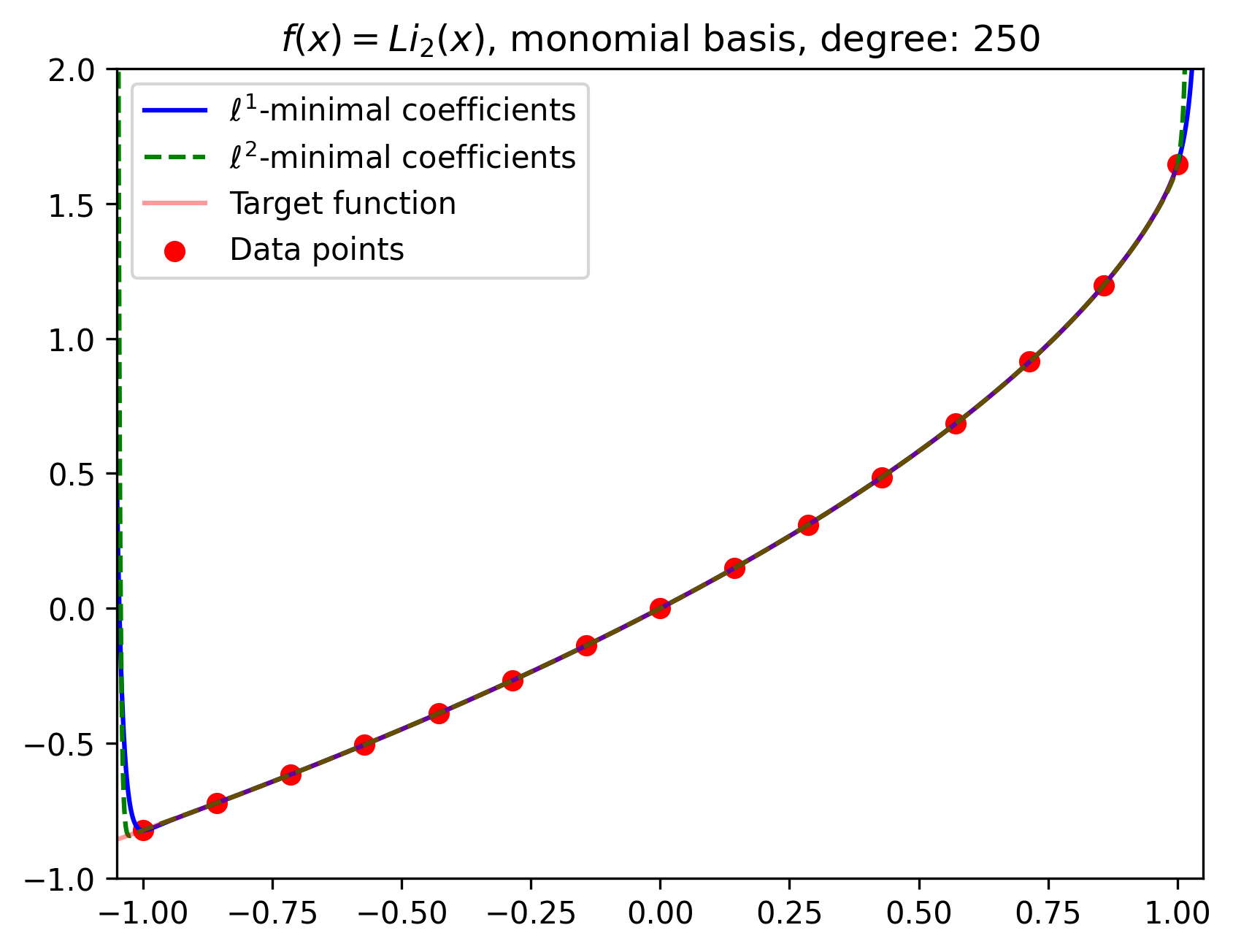}

    \includegraphics[width=0.24\linewidth]{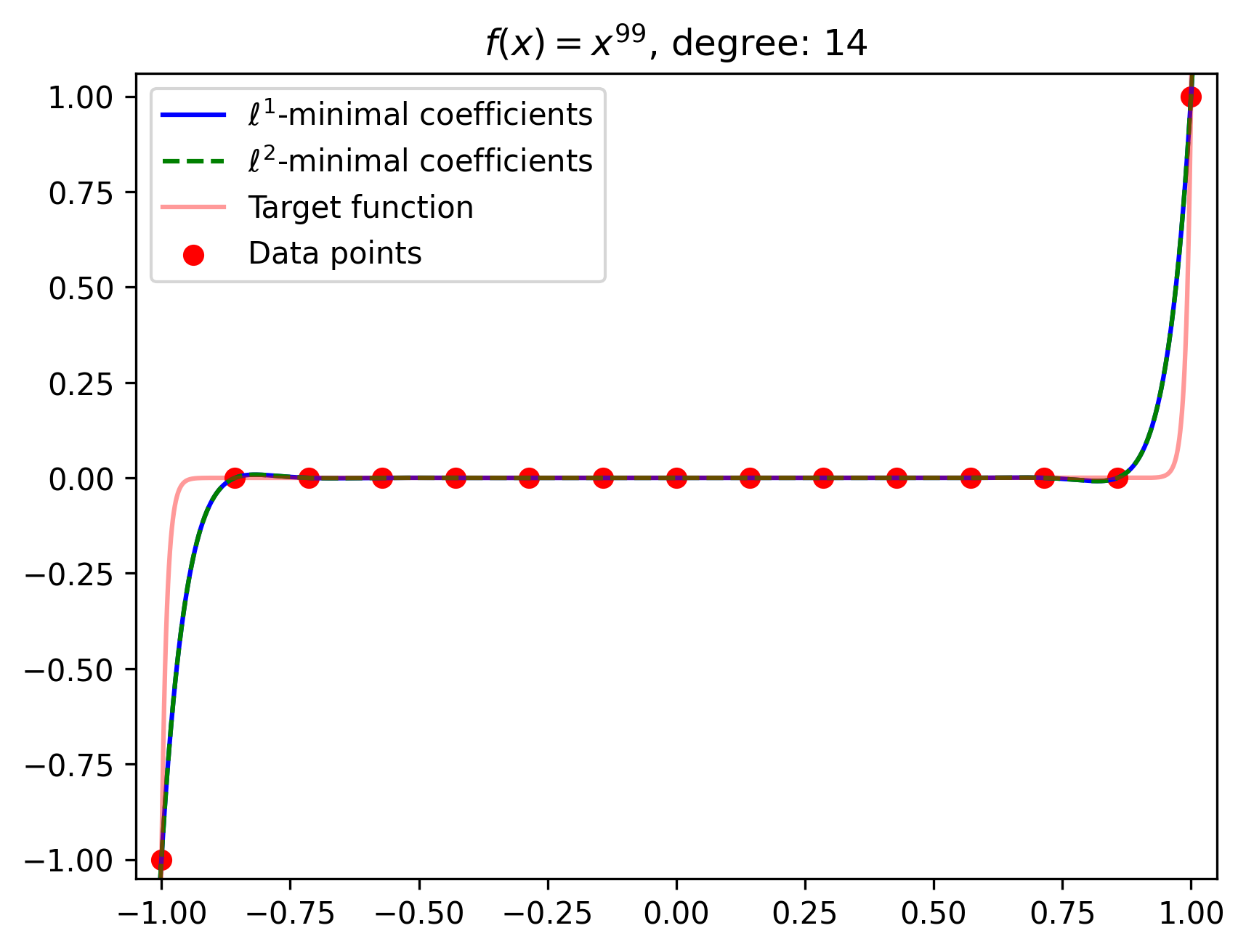}
    \includegraphics[width=0.24\linewidth]{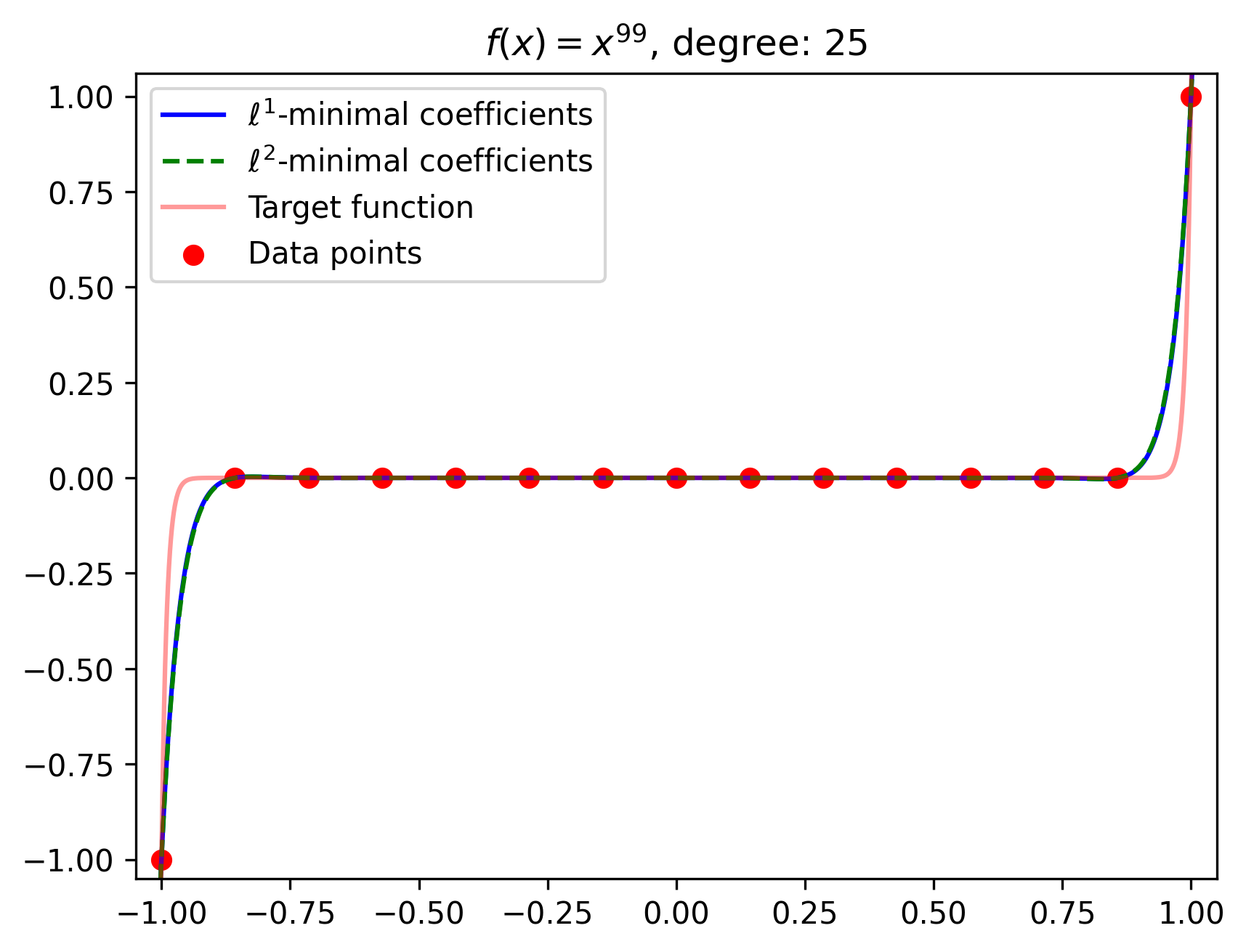}
    \includegraphics[width=0.24\linewidth]{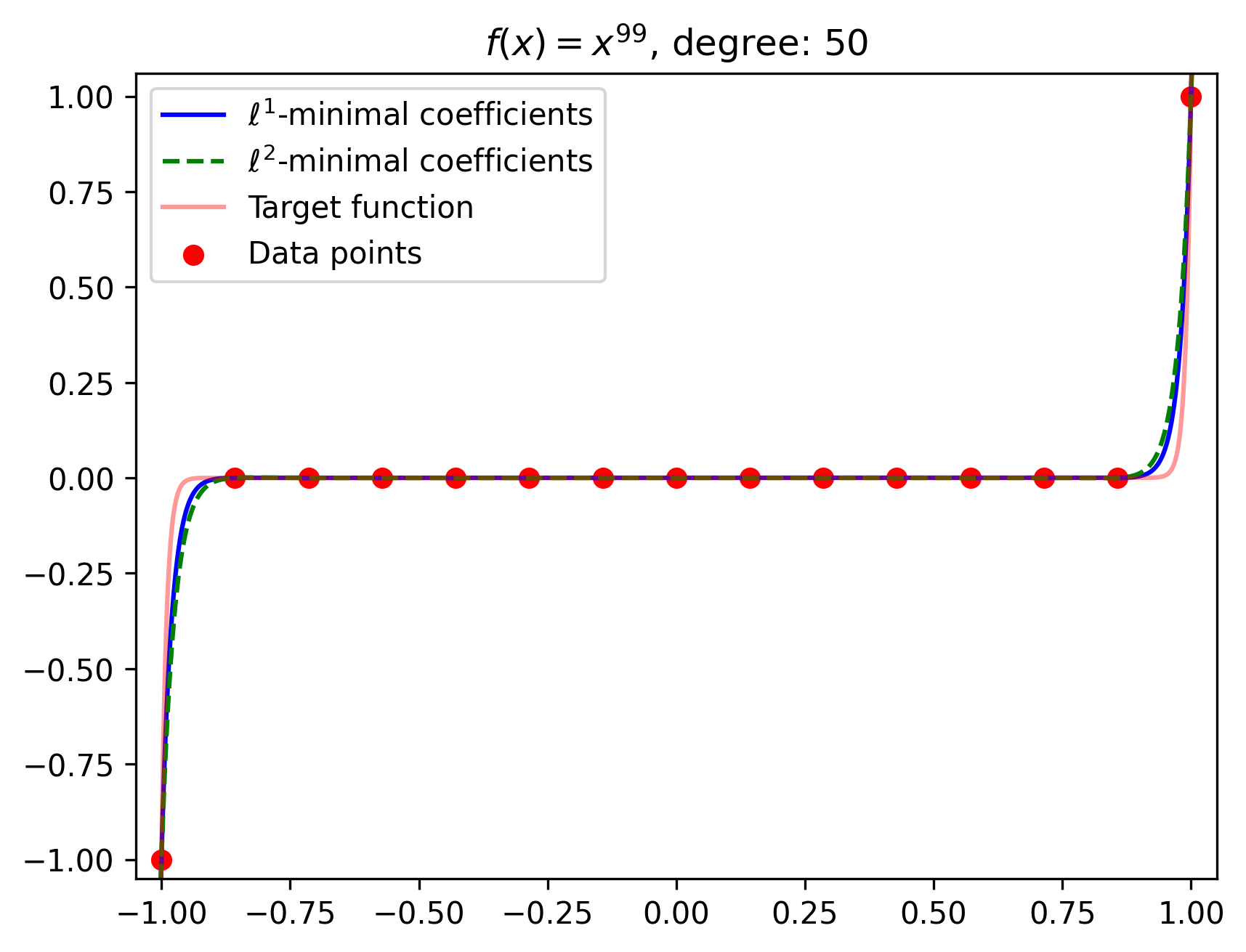}
    \includegraphics[width=0.24\linewidth]{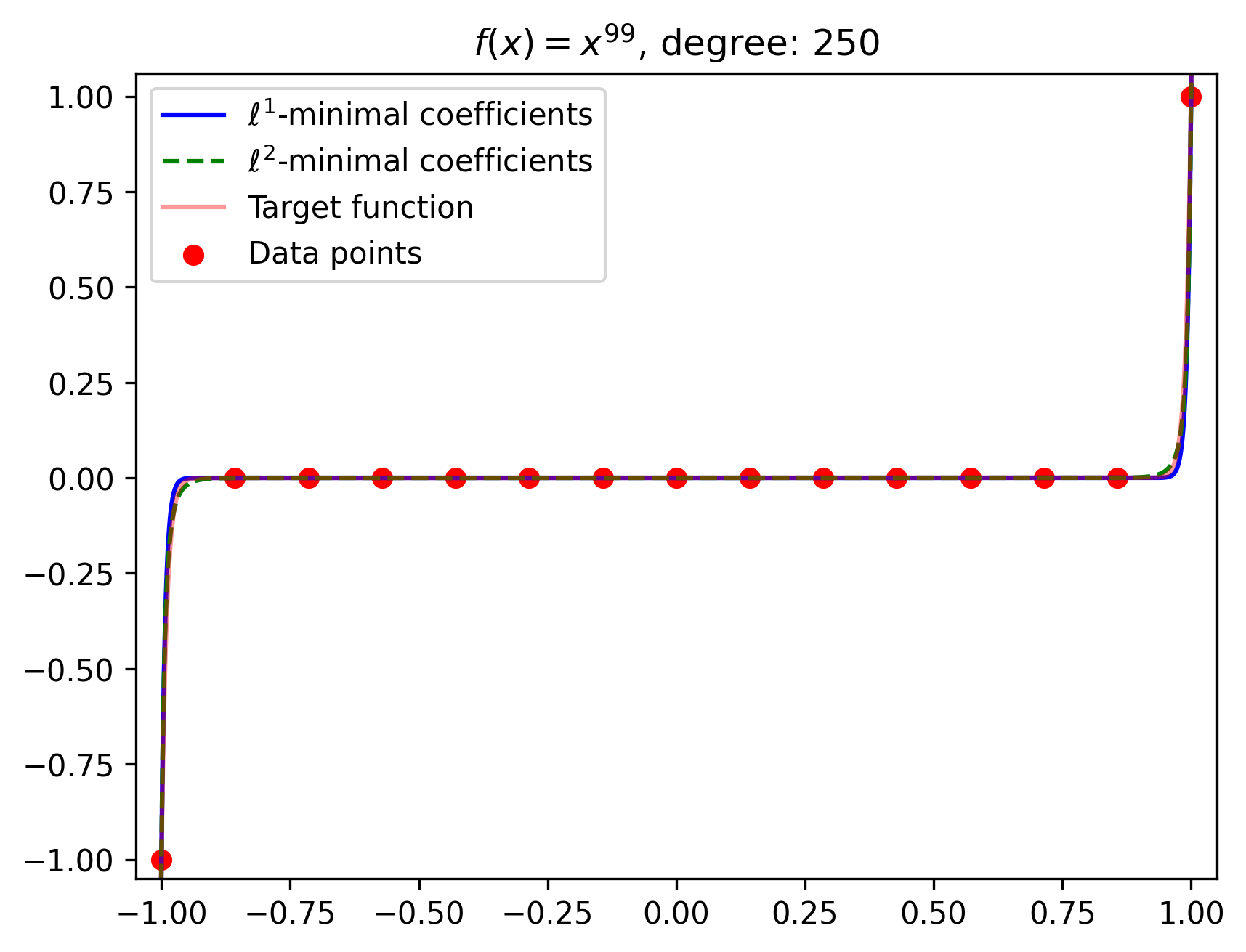}

    \includegraphics[width=0.24\linewidth]{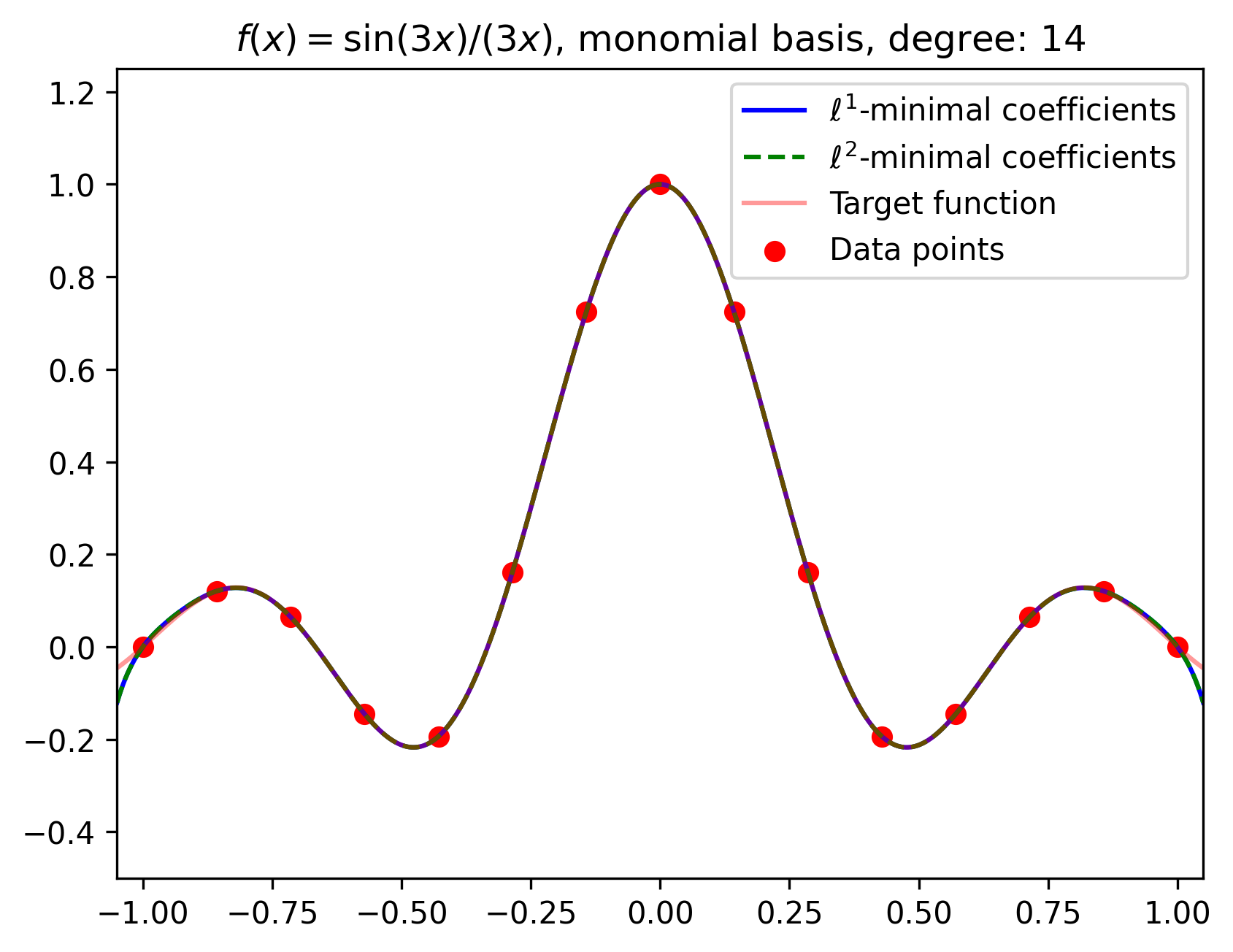}
    \includegraphics[width=0.24\linewidth]{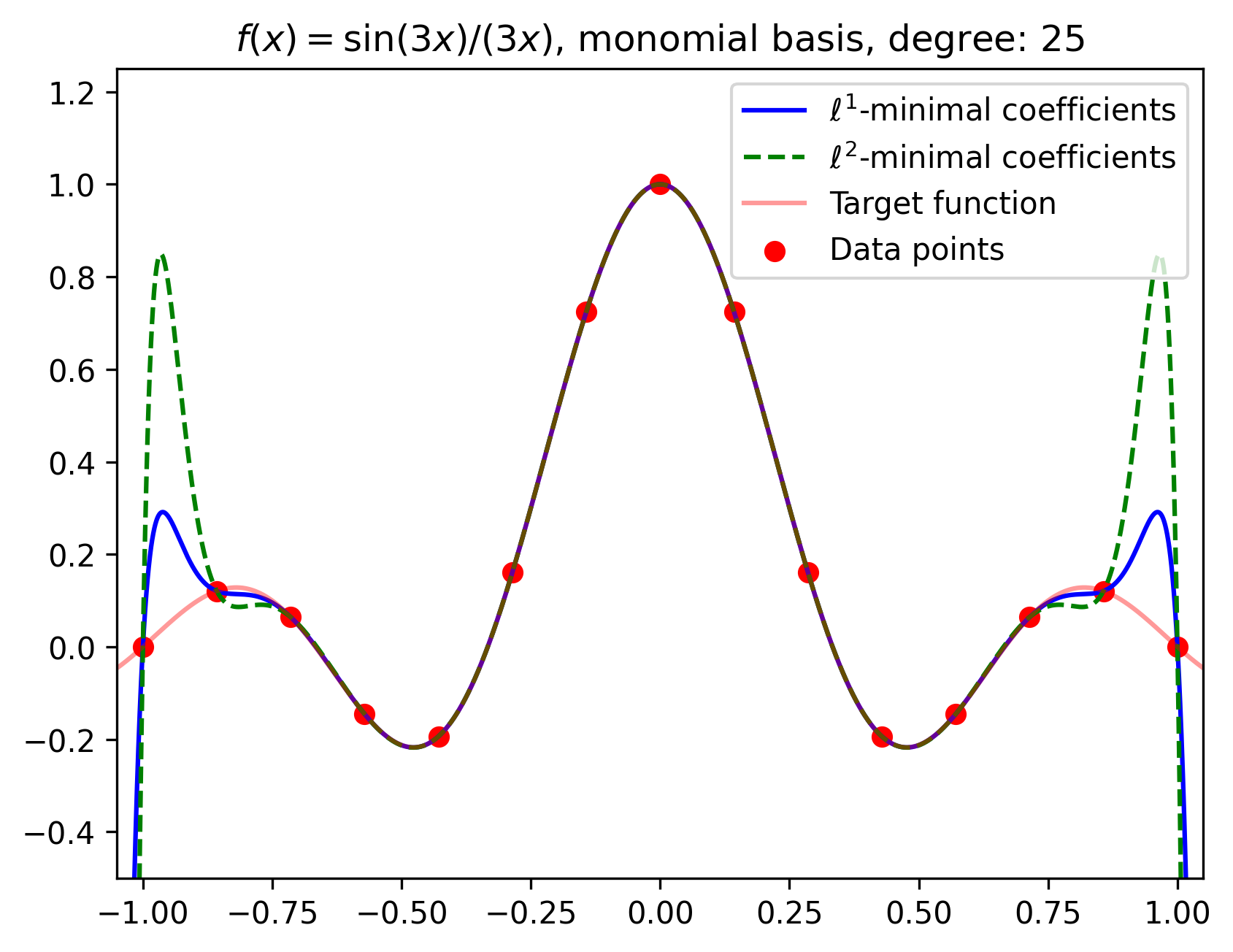}
    \includegraphics[width=0.24\linewidth]{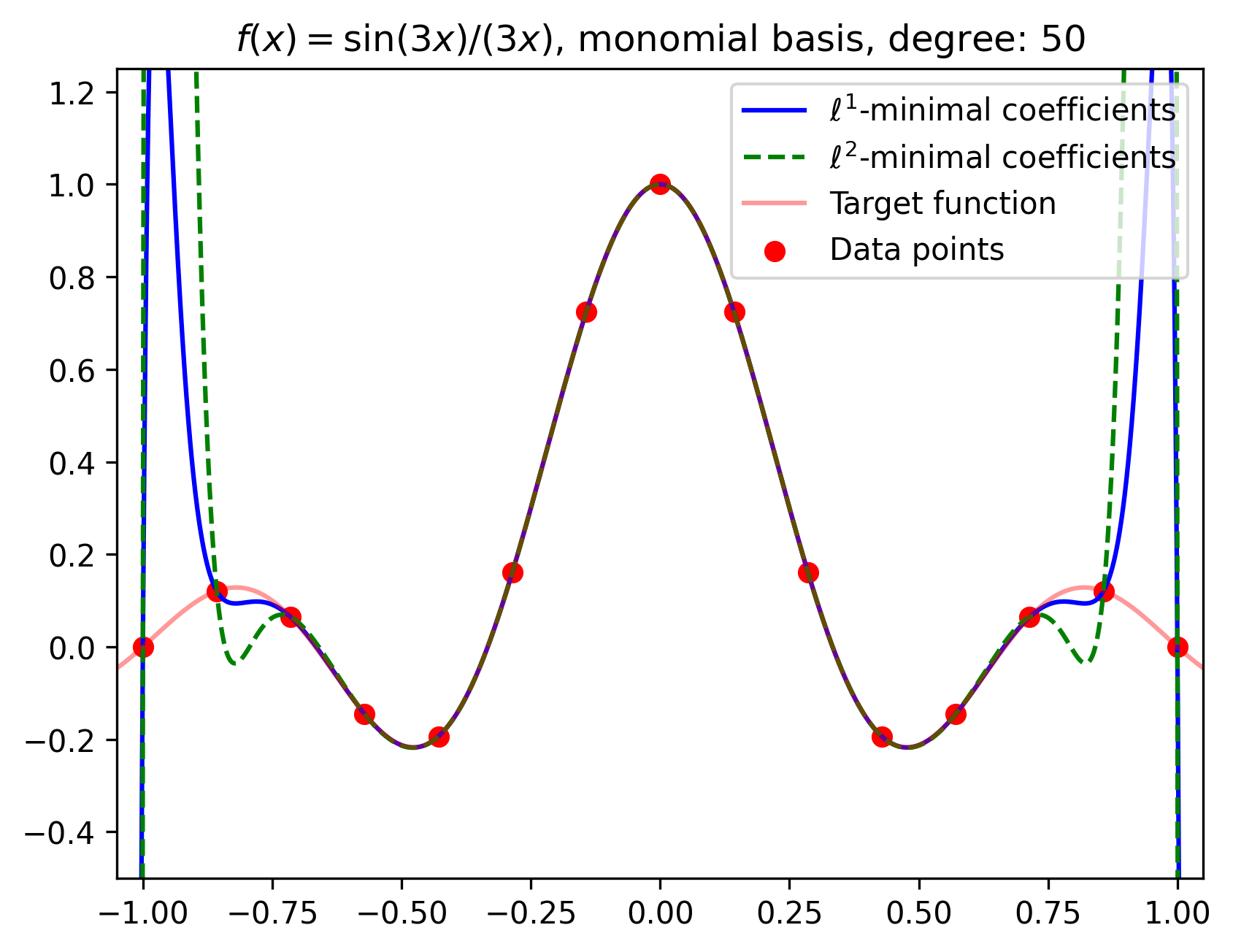}
    \includegraphics[width=0.24\linewidth]{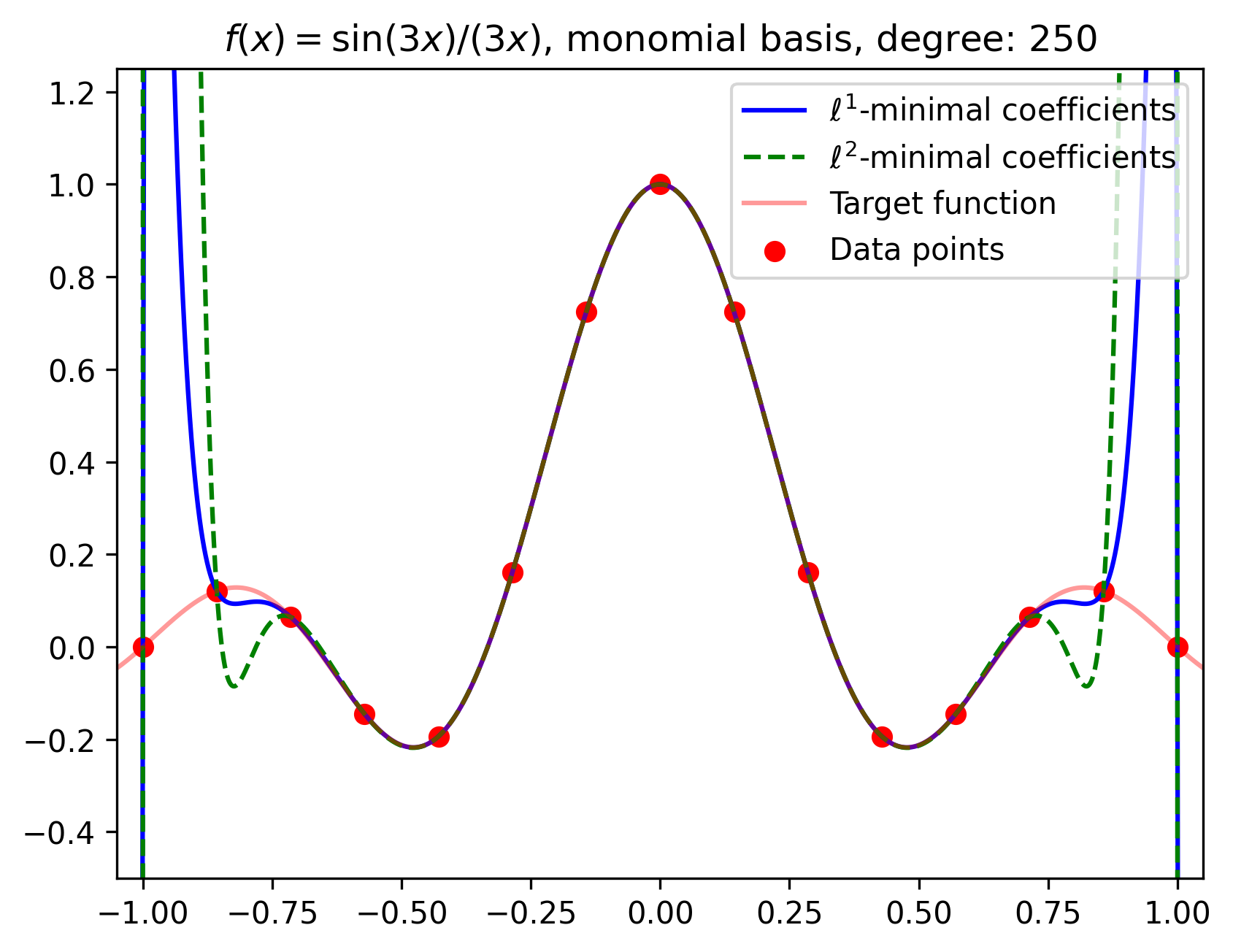}
    
    \caption{\label{figure monomial l1 and l2}
    We compare interpolating polynomials of degree $d$ with minimal monomial basis coefficients in the $\ell^1$- and $\ell^2$-sense for $n=15$ data points $(x_i, f(x_i))$ with equidistant $x_i$ in the interval $[-1,1]$. Left to right, the degree of interpolants increases from 14 (unique interpolating polynomial) to 25, 50 and 250. Top to bottom, we consider the target functions $f_1(x) = Li_2(x)$, $f_2(x) =x^{99}$ and $f_3(x) = \mathrm{sinc}(3x) = \sin(3x)/(3x)$.\\ In the top line, high degree interpolants are visually indistinguishable from the target function between the initial and final data points, but change rapidly outside the given interval, just below $-1$. In the second line, we see that the overparametrized interpolants outperform Lagrange polynomials. In the third line, both interpolants develop large boundary oscillation at the interval edge. 
    }
\end{figure}

In Figure \ref{figure overparametrized runge}, we present empirical evidence that overparametrized polynomials with minimal basis coefficients (in the $\ell^1$- or the $\ell^2$-sense) may fail to converge to target functions even in settings where the Lagrange polynomial (the unique minimal degree interpolating polynomial) is guaranteed to converge. Such a convergence guarantee holds for instance when $y_i = f(x_i)$ for a $C^2$-function $f:[a,b]\to\R$ if the points $x_i = (a+b)/2 + (b-a)/2 \,\cos(\pi i/n)$ are Chebyshev points in $[a,b]$ \cite[Chapter 7]{trefethen2019approximation}. For both the absolute value function $f(x) = |x|$ and the Runge function $f(x) = 1/ (1+25x^2)$ on $[-1,1]$, we observe that the interpolating solutions develop large boundary oscillations if the admissible degree is sufficiently large compared to the number of data points. We do not prove rigorously that a Runge type phenomenon occurs.

\begin{figure}
    \centering
    \includegraphics[width=0.24\linewidth]{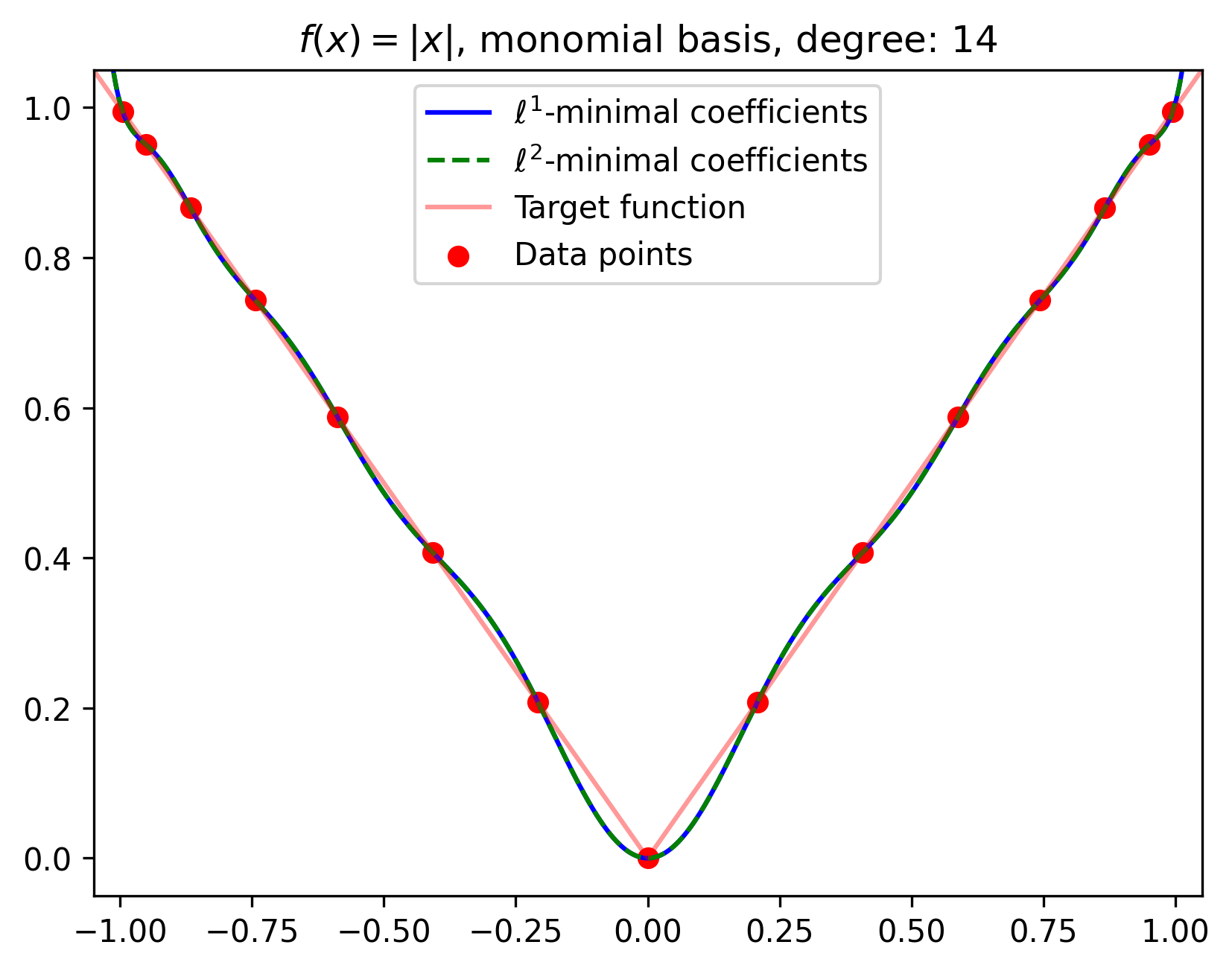}
    \includegraphics[width=0.24\linewidth]{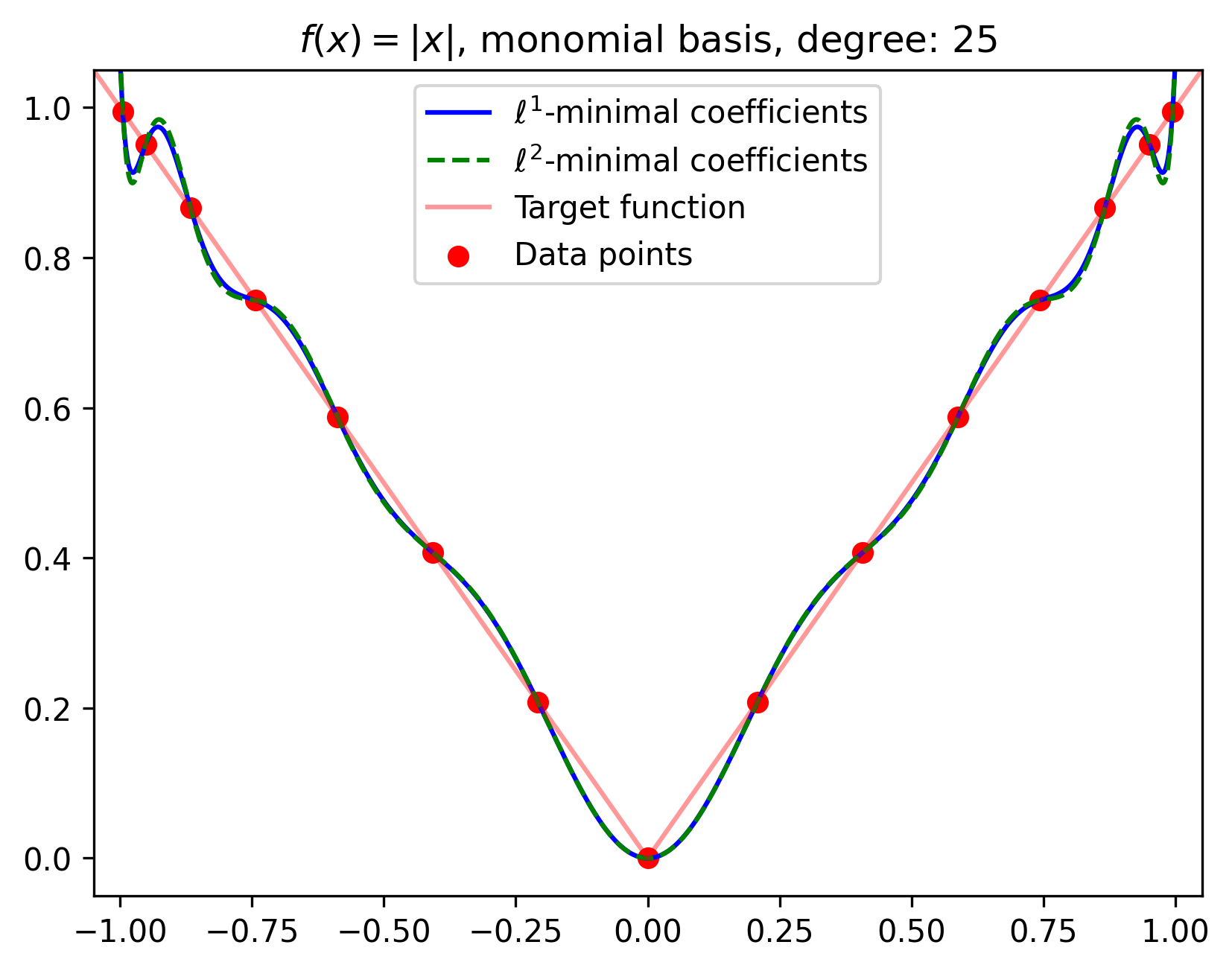}
    \includegraphics[width=0.24\linewidth]{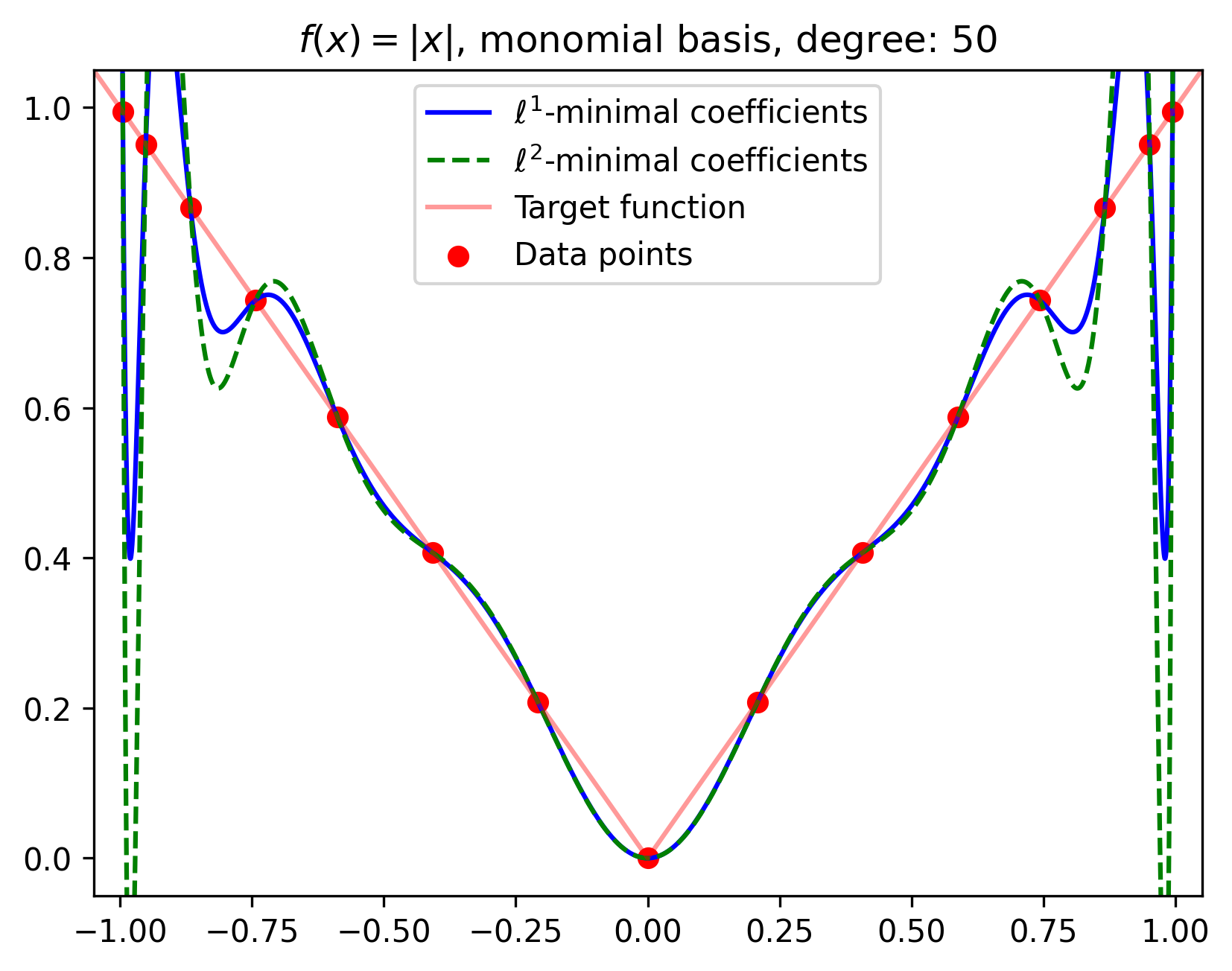}
    \includegraphics[width=0.24\linewidth]{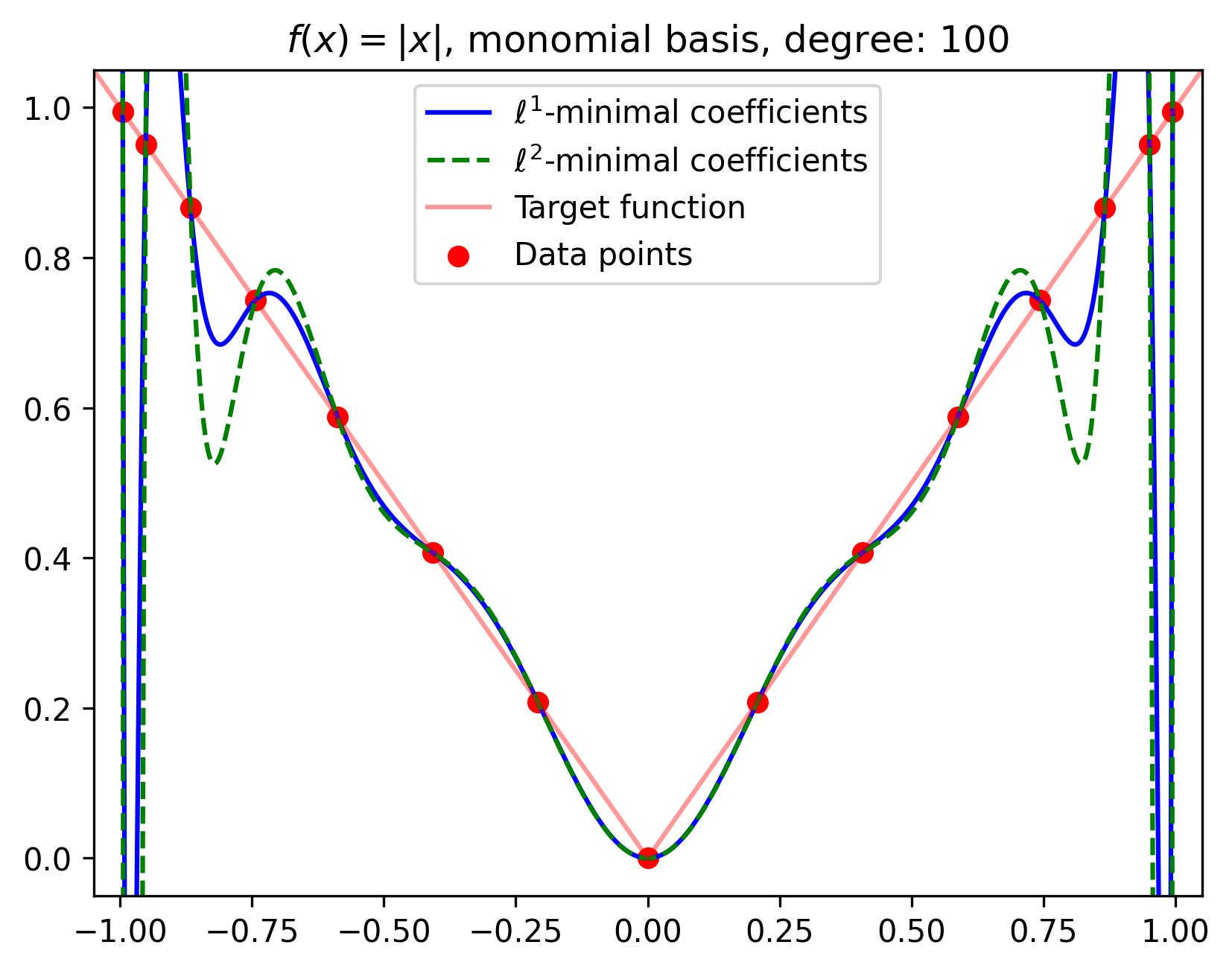}

    \includegraphics[width=0.24\linewidth]{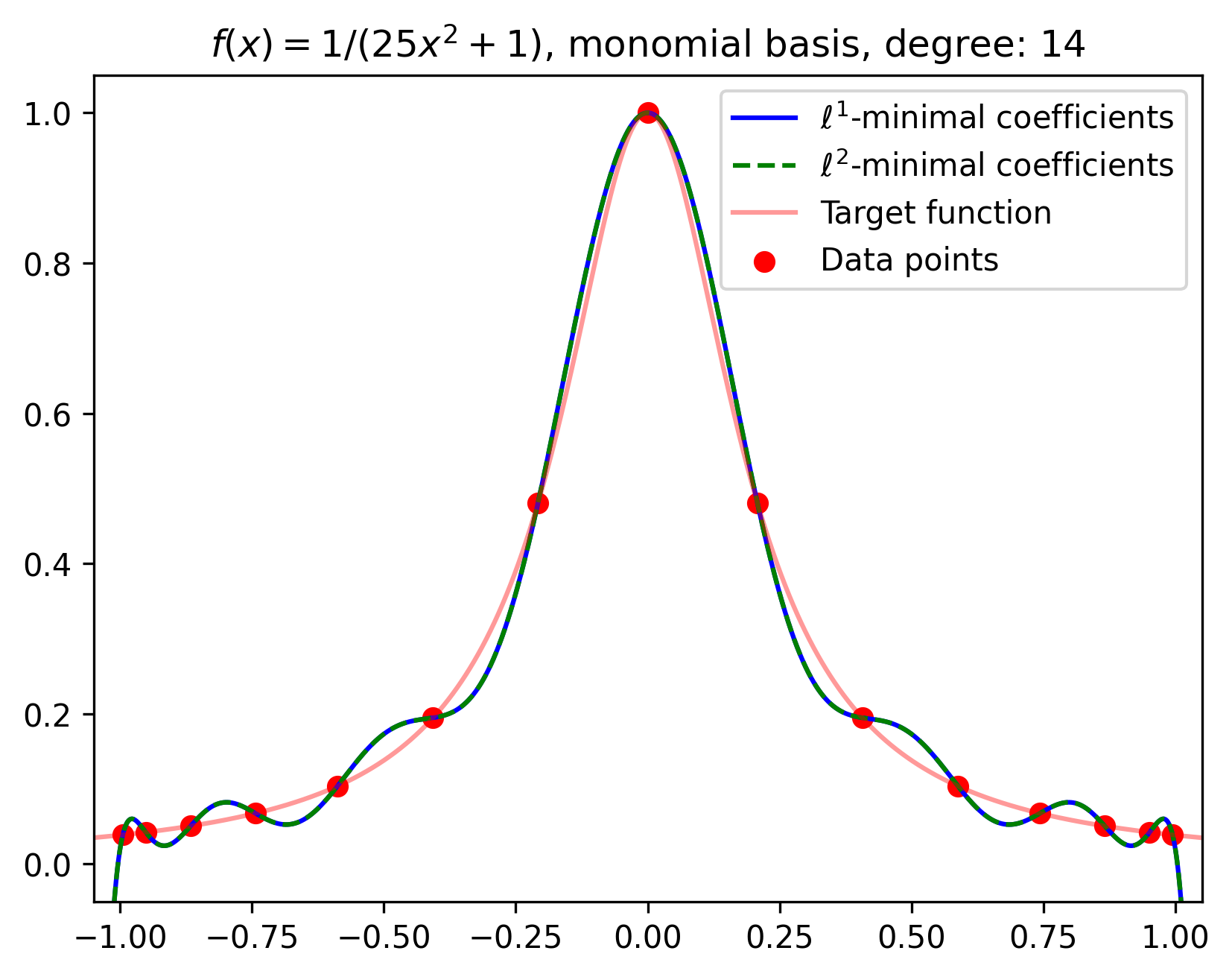}
    \includegraphics[width=0.24\linewidth]{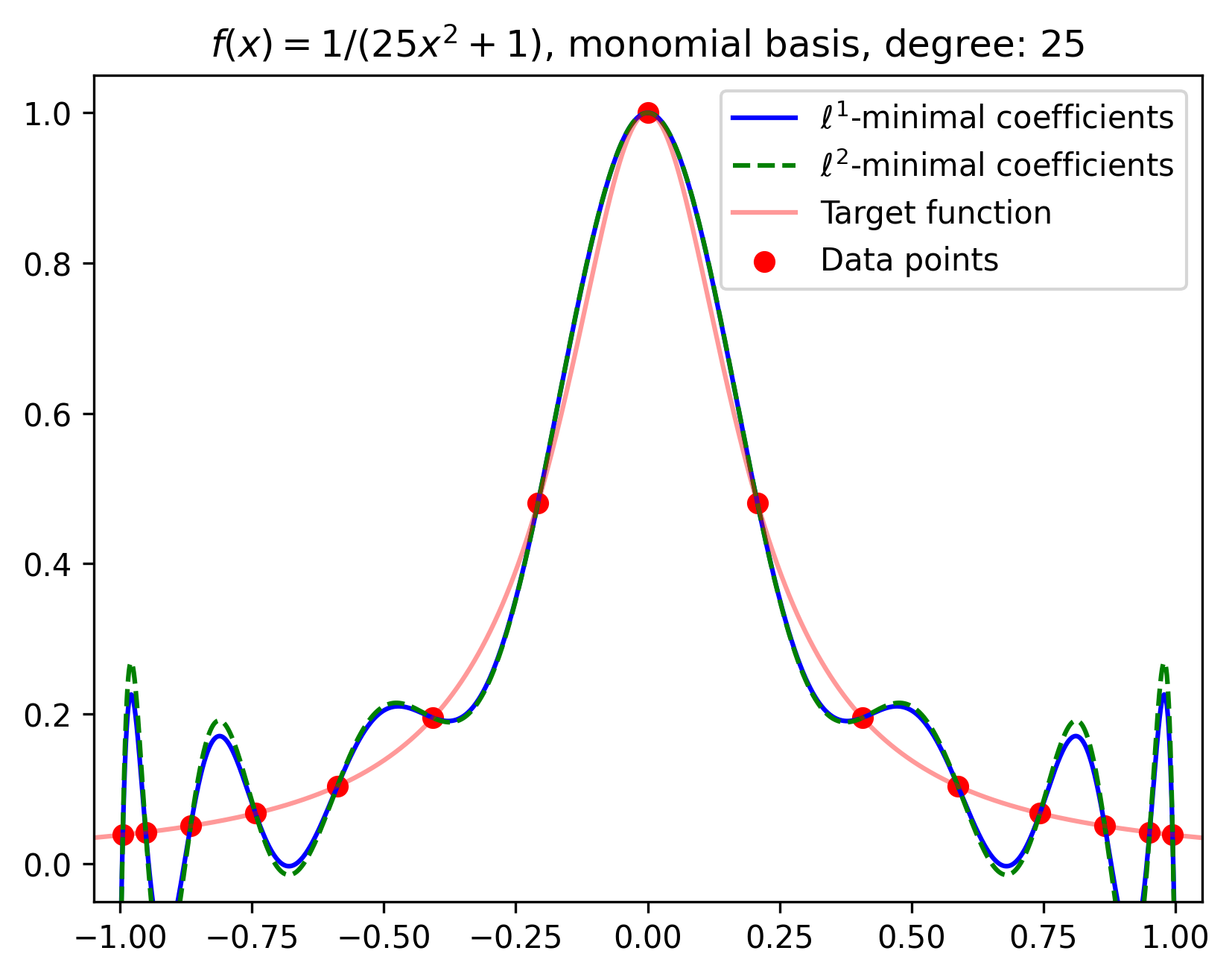}
    \includegraphics[width=0.24\linewidth]{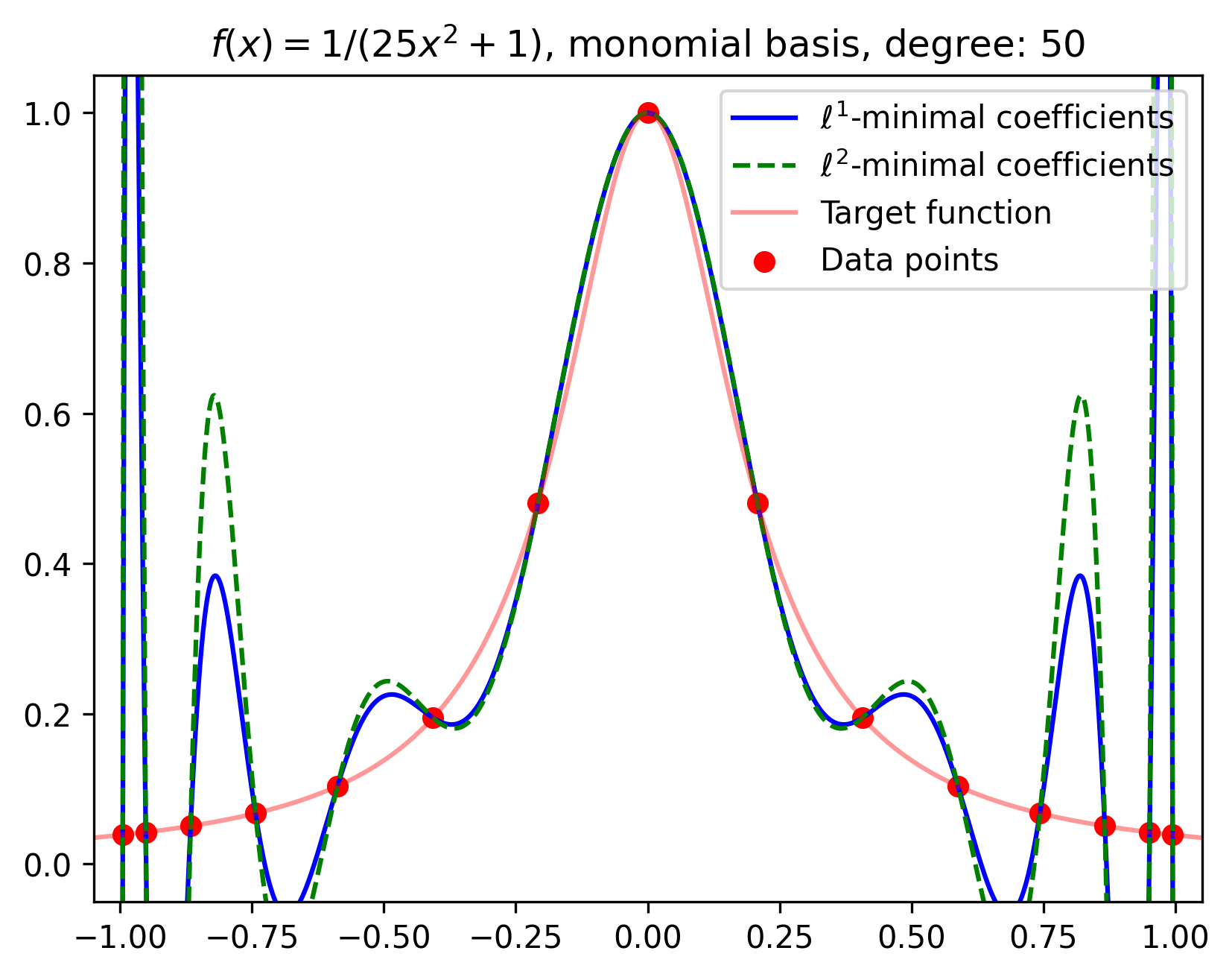}
    \includegraphics[width=0.24\linewidth]{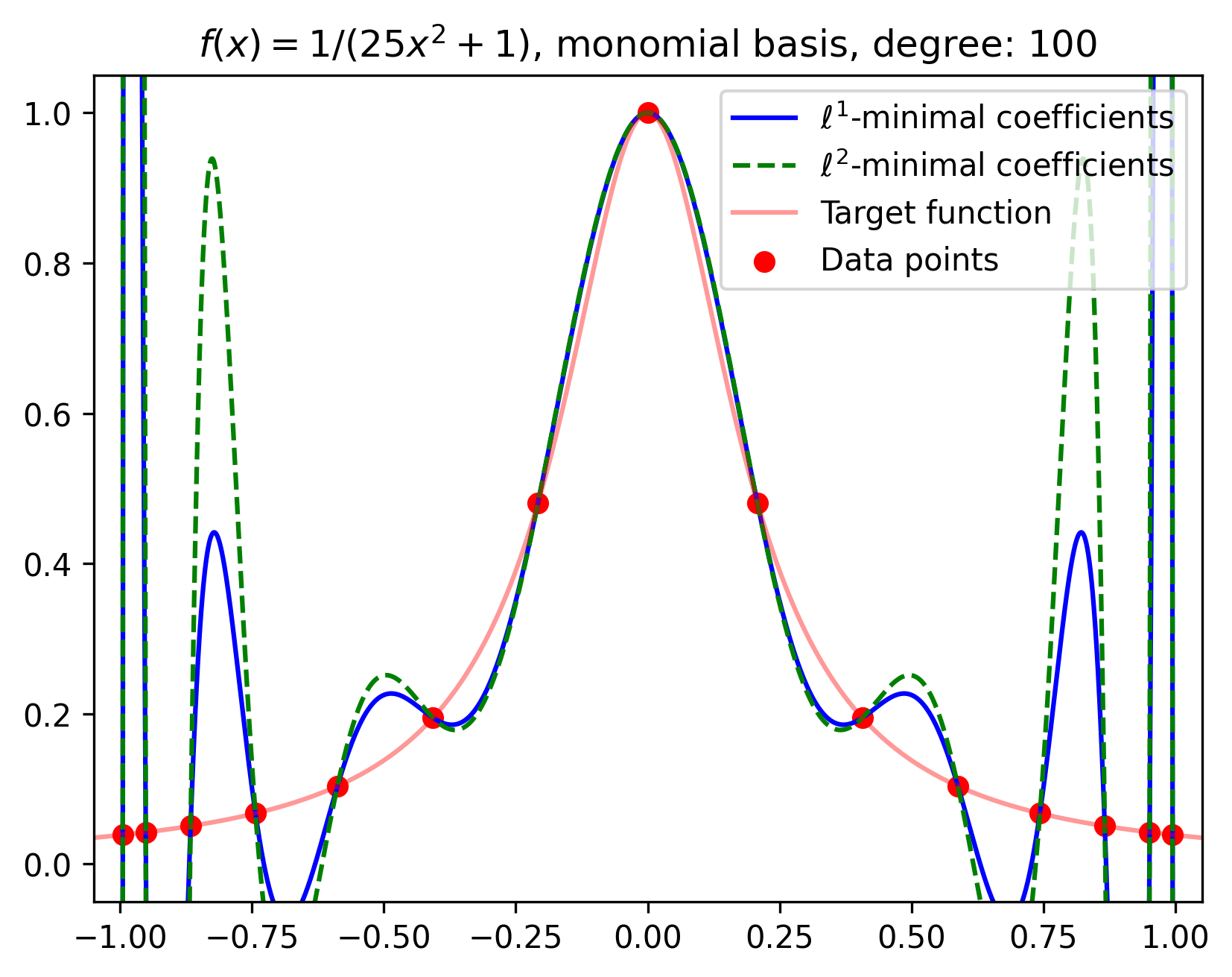}

    \includegraphics[width=0.24\linewidth]{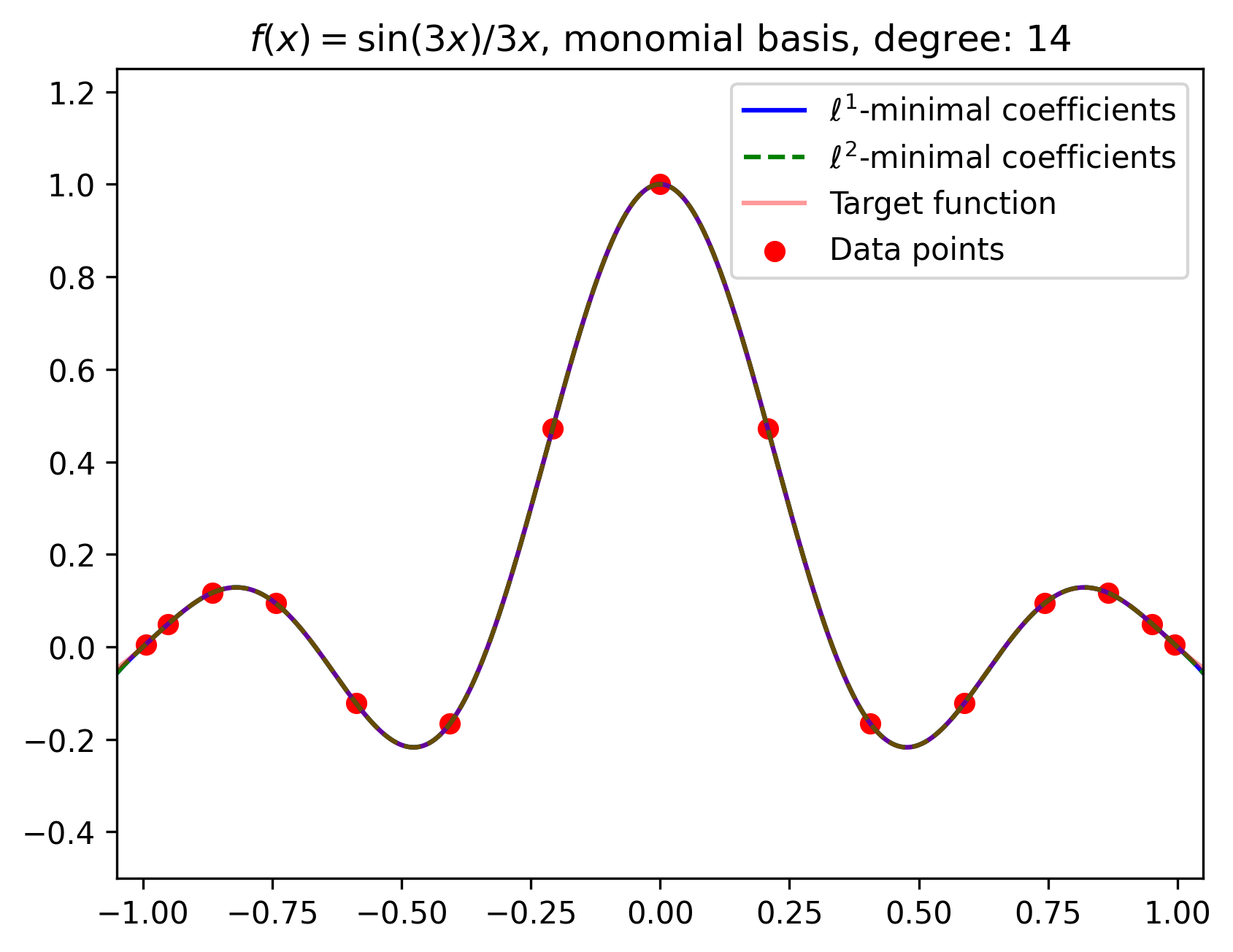}
    \includegraphics[width=0.24\linewidth]{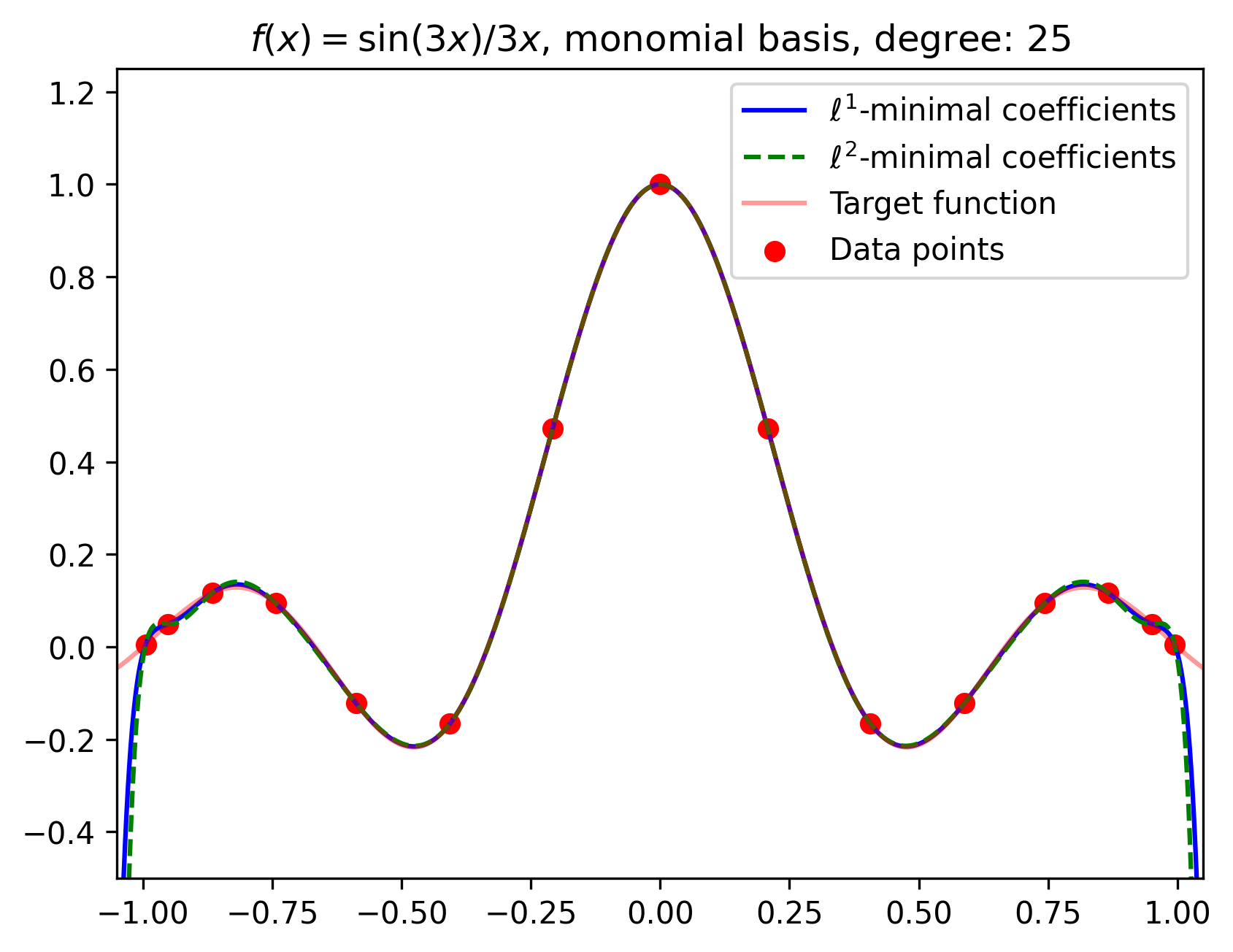}
    \includegraphics[width=0.24\linewidth]{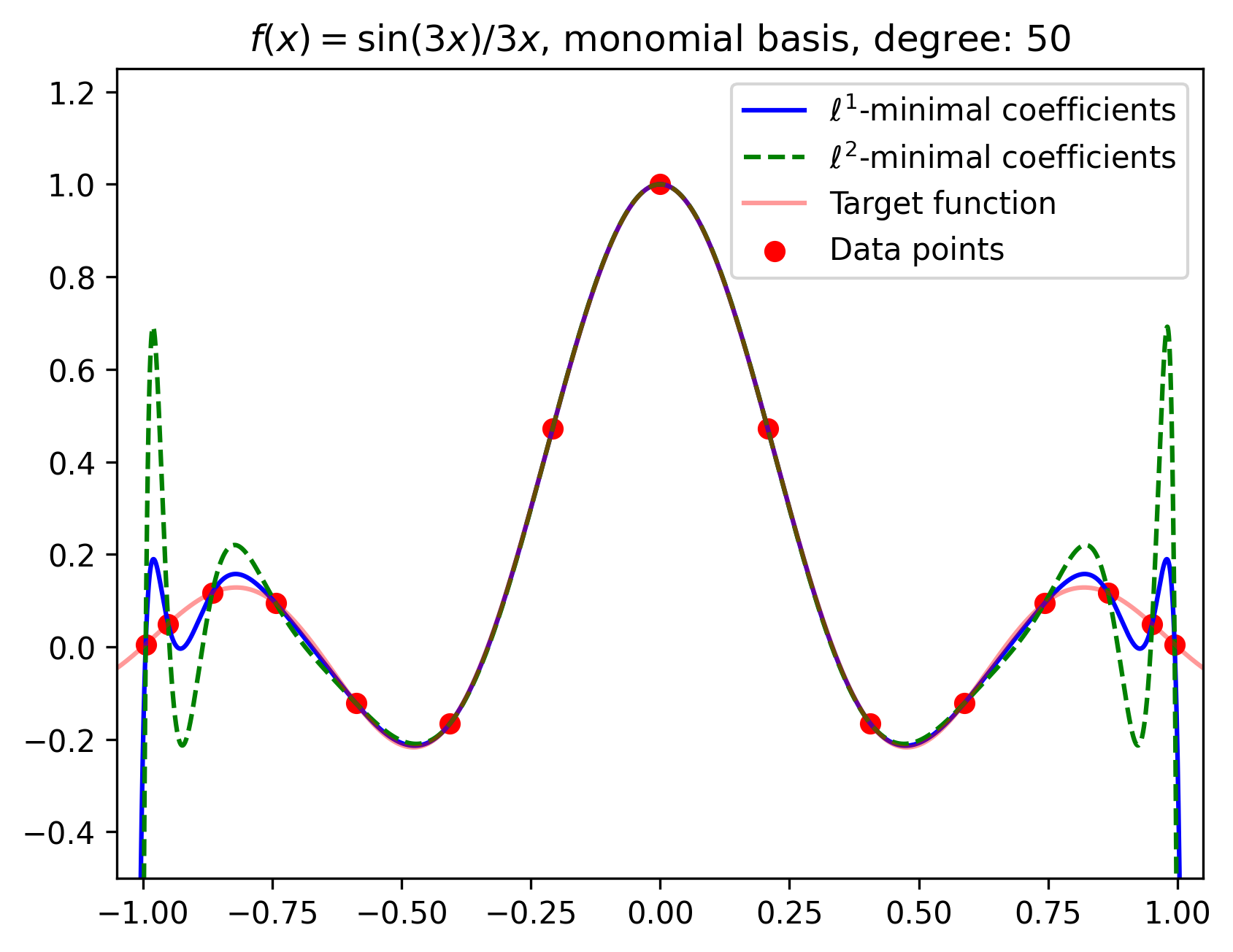}
    \includegraphics[width=0.24\linewidth]{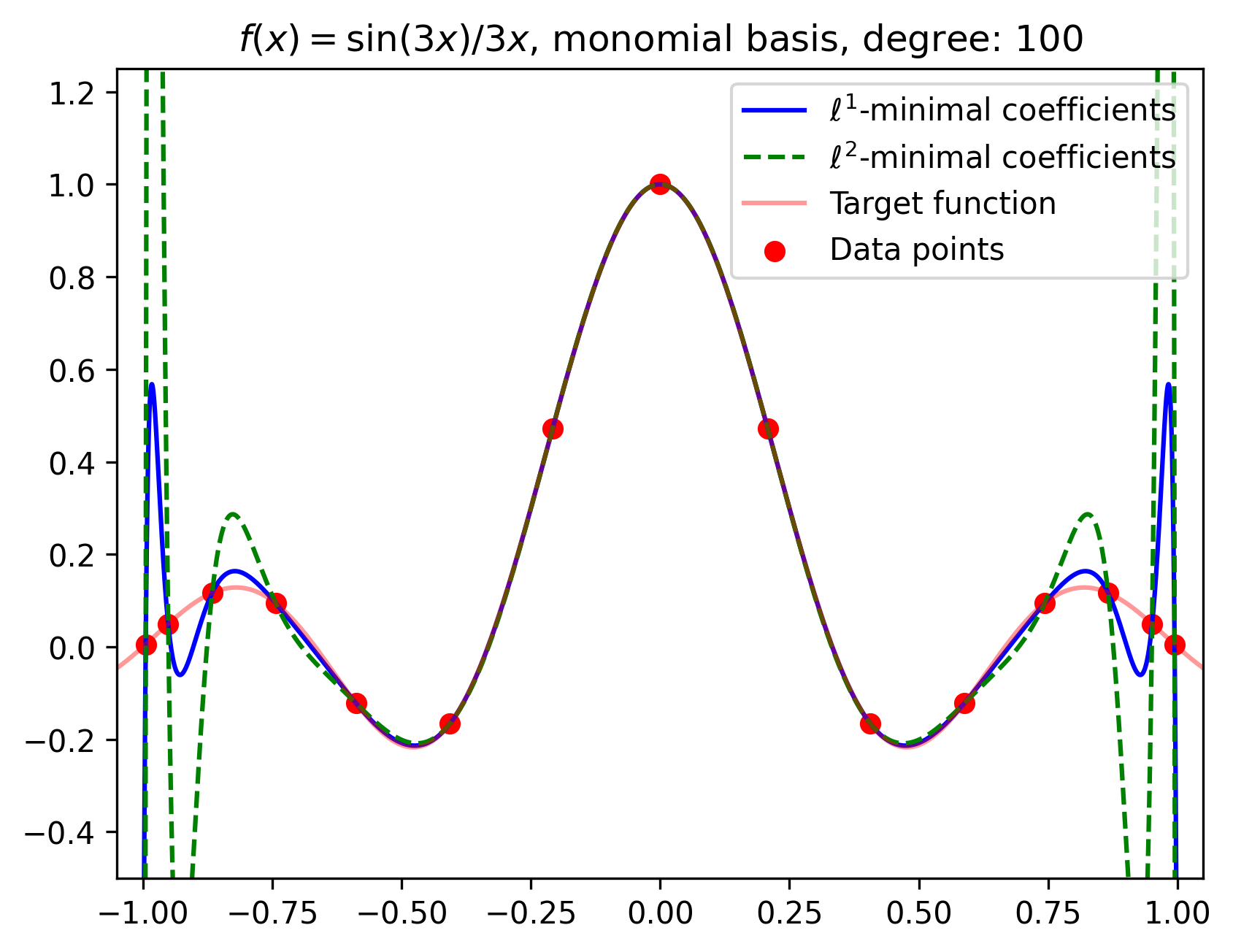}

    \caption{\label{figure overparametrized runge}
    We compare interpolating polynomials of degree $d$ with minimal monomial basis coefficients in the $\ell^1$- and $\ell^2$-sense for $n=15$ data pairs $(x_i, f(x_i))$ where $x_i$ are Chebyshev points in the interval $[-1,1]$. Left to right, the degree of interpolants increases from 14 (unique interpolating Lagrange polynomial) to 25, 50, and 100. Top to bottom, we consider the target functions $f_1(x) = |x|$, $f_2(x) =1/(25x^2+1)$ and $f_3(x) = \mathrm{sinc}(3x)$.\\ While the Lagrange solutions are guaranteed to converge as $n\to\infty$, we observe large boundary oscillations in the overparametrized solutions with minimal norm monomial basis coefficients of either type.
    }
\end{figure}

\section{Chebyshev Basis Regularization}\label{section chebyshev}

The Chebyshev basis $\{q_0, \dots, q_d\}$ is a collection of polynomials such that the first $d+1$ functions span the same space as the first $d+1$ monomials, i.e.\ the space $\mathcal P_d$ of polynomials of degree $d$. The Chebyshev polynomials (of the first kind) satisfy the orthogonality relation
\[
\langle q_d, q_r\rangle = \int_{-1}^1\frac{q_d(x)\,q_r(x)}{\sqrt{1-x^2}} \dx = \begin{cases} 0 &\text{if }d\neq r\\ \pi &\text{if }d=r=0\\ \pi/2 &\text{if }d=r\neq 0.\end{cases}
\]
with respect to the $L^2(\nu)$-inner product where $\d\nu = (1-x^2)^{-1/2}\dx$. Naturally, $\nu$ is a Radon measure and $\nu \geq \L^1$, so in particular $\|\cdot \|_{L^2(\nu)}\geq \|\cdot \|_{L^2(-1,1)}$ and convergence in $L^2(\nu)$ implies regular $L^2$-convergence.

\begin{theorem} \label{main chebyshev theorem}
Let $(x_i, y_i)_{i=0}^n$ be a collection of data pairs with $x_i \in [-1,1]$ for all $i$. Then for every $d\geq n$, there exists a unique coefficient vector $a = a^d\in \R^{d+1}$ such that
\[
a = \argmin \left\{ \|a\|_{\ell^2}^2 : \sum_{k=0}^d a_k \,q_k(x_i) = y_i \:\:\forall\ i=0, \dots, n\right\}.
\]
Furthermore, denoting $P_d(x) = \sum_{k=0}a^d_k\,q_k(x)$ we have
\[
\lim_{d\to \infty}\|a^d\|_{\ell^2}^2 = 0, \qquad \lim_{d\to\infty} \|P_d\|_{L^2(-1,1)} = \lim_{d\to\infty} \|P_d\|_{L^2(\nu)} = 0.
\]
\end{theorem}

\begin{figure}
    \centering
    \includegraphics[width=0.24\linewidth]{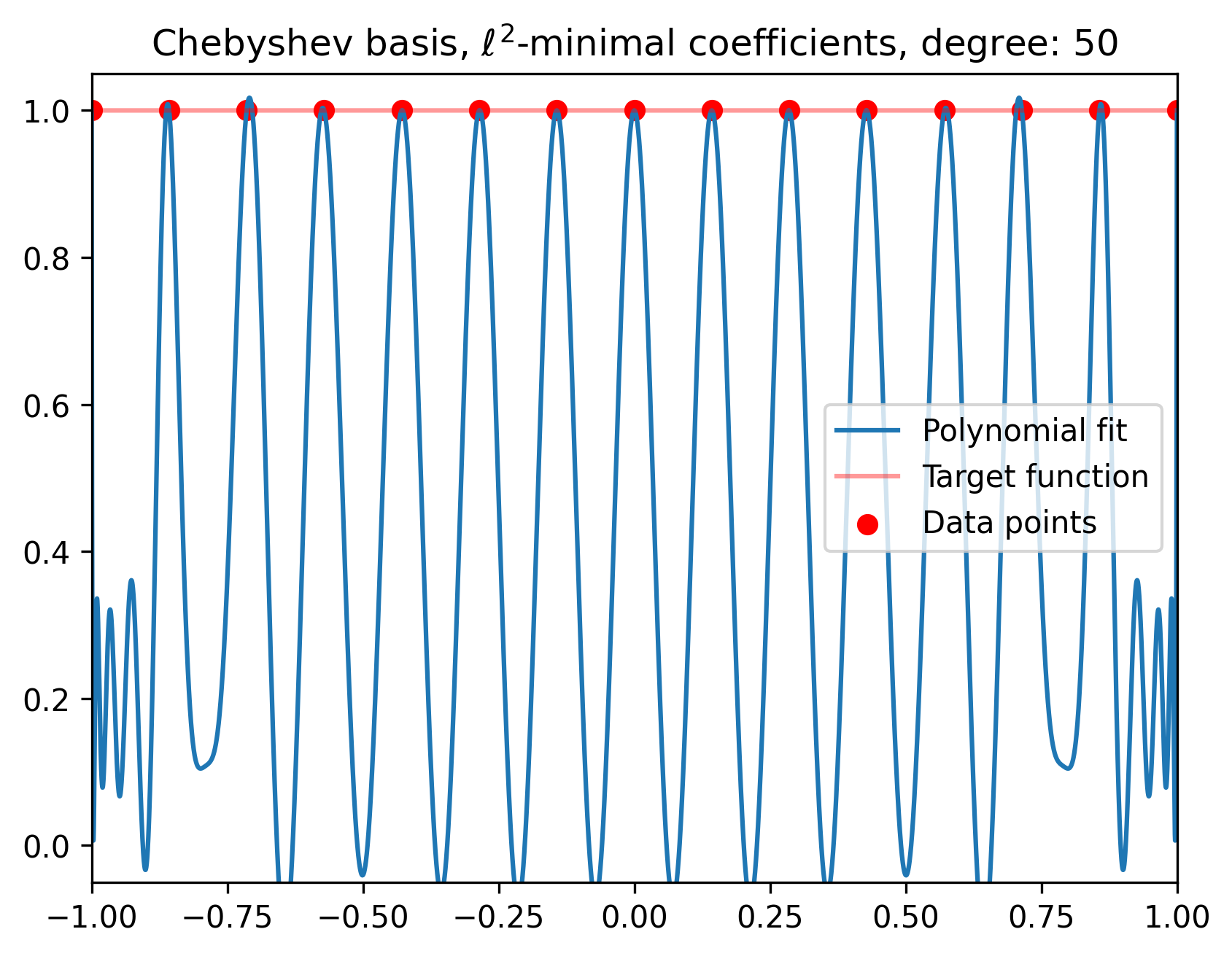}
    \includegraphics[width=0.24\linewidth]{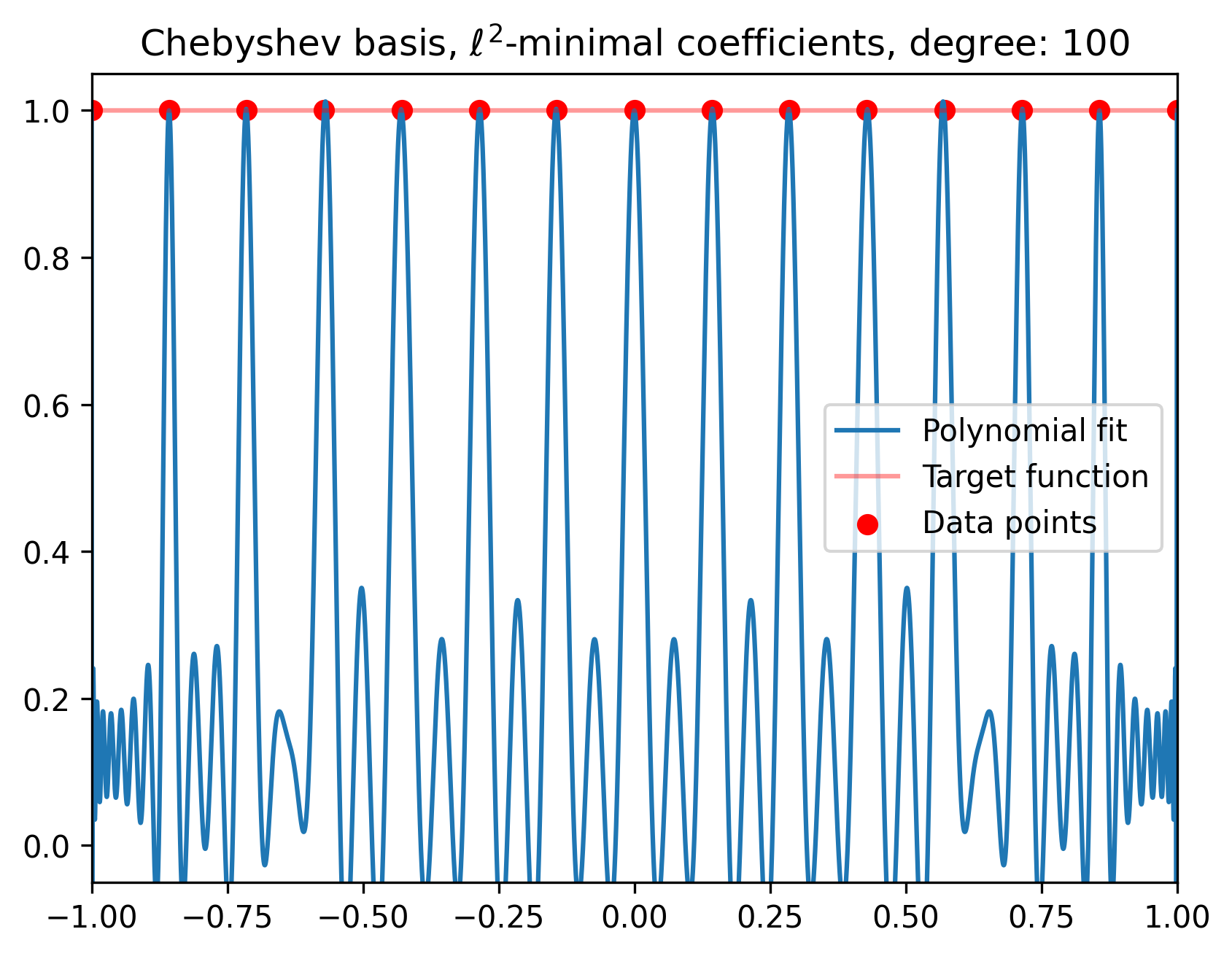}
    \includegraphics[width=0.24\linewidth]{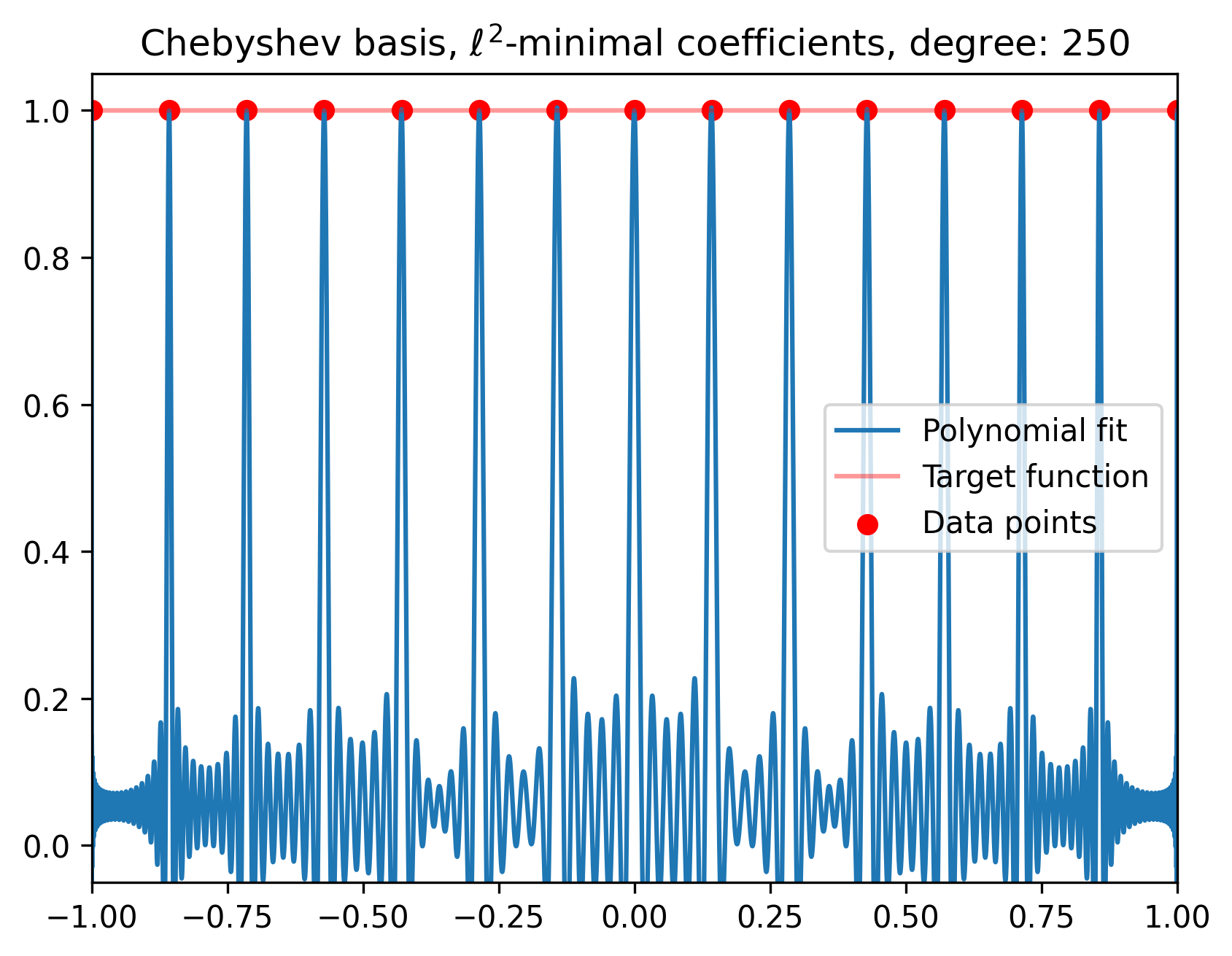}
    \includegraphics[width=0.24\linewidth]{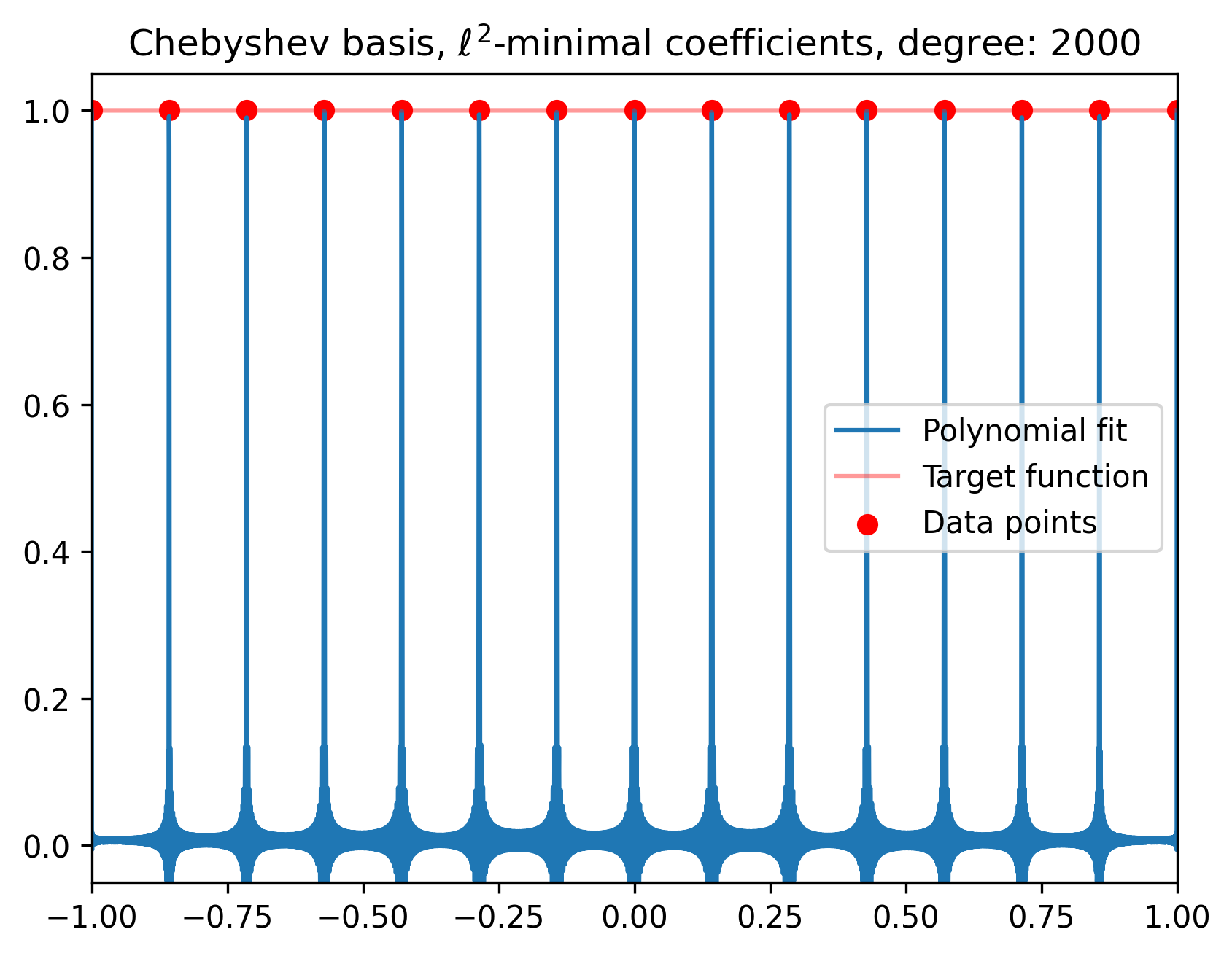}

\vspace{3mm}

    \includegraphics[width=0.24\linewidth]{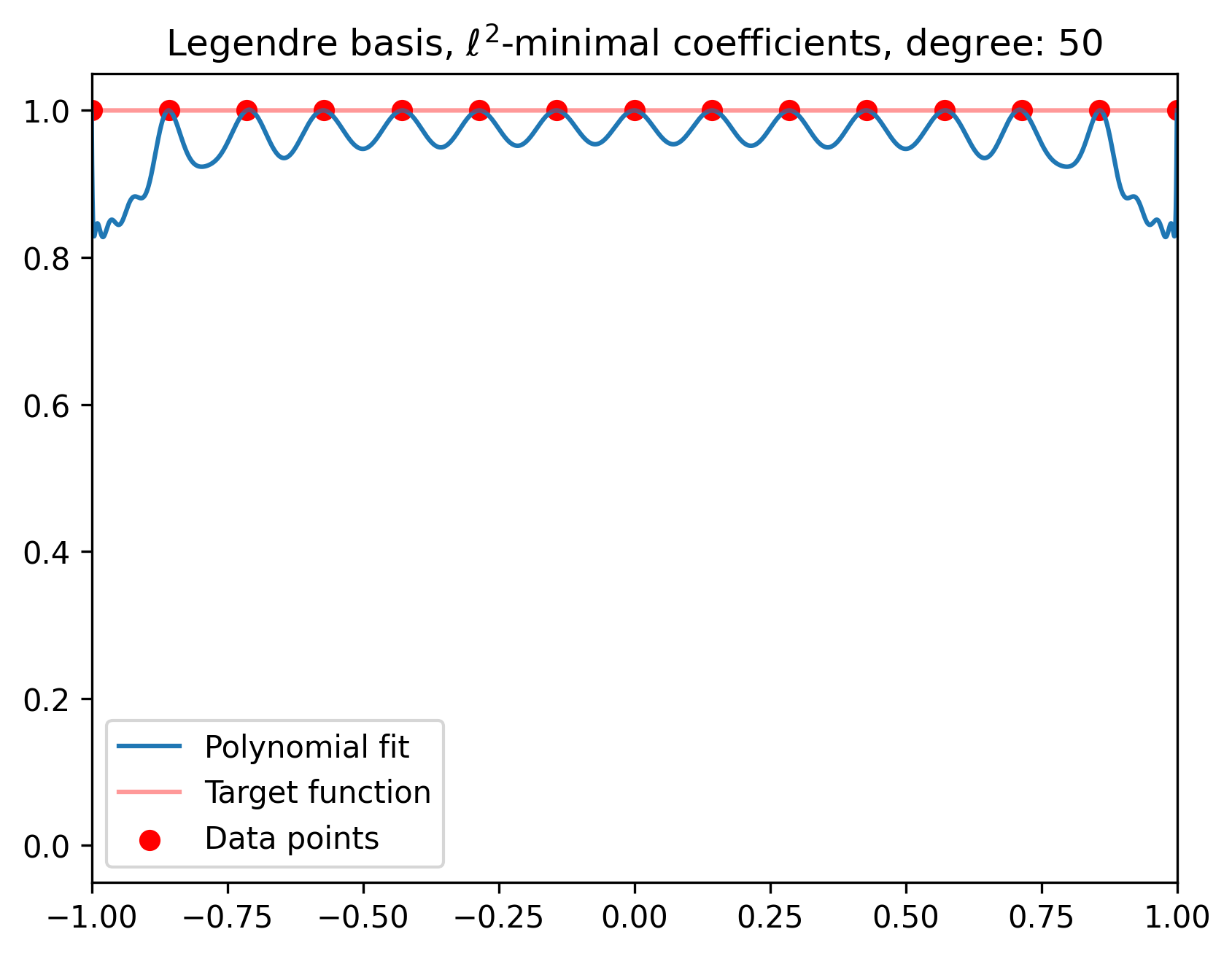}
    \includegraphics[width=0.24\linewidth]{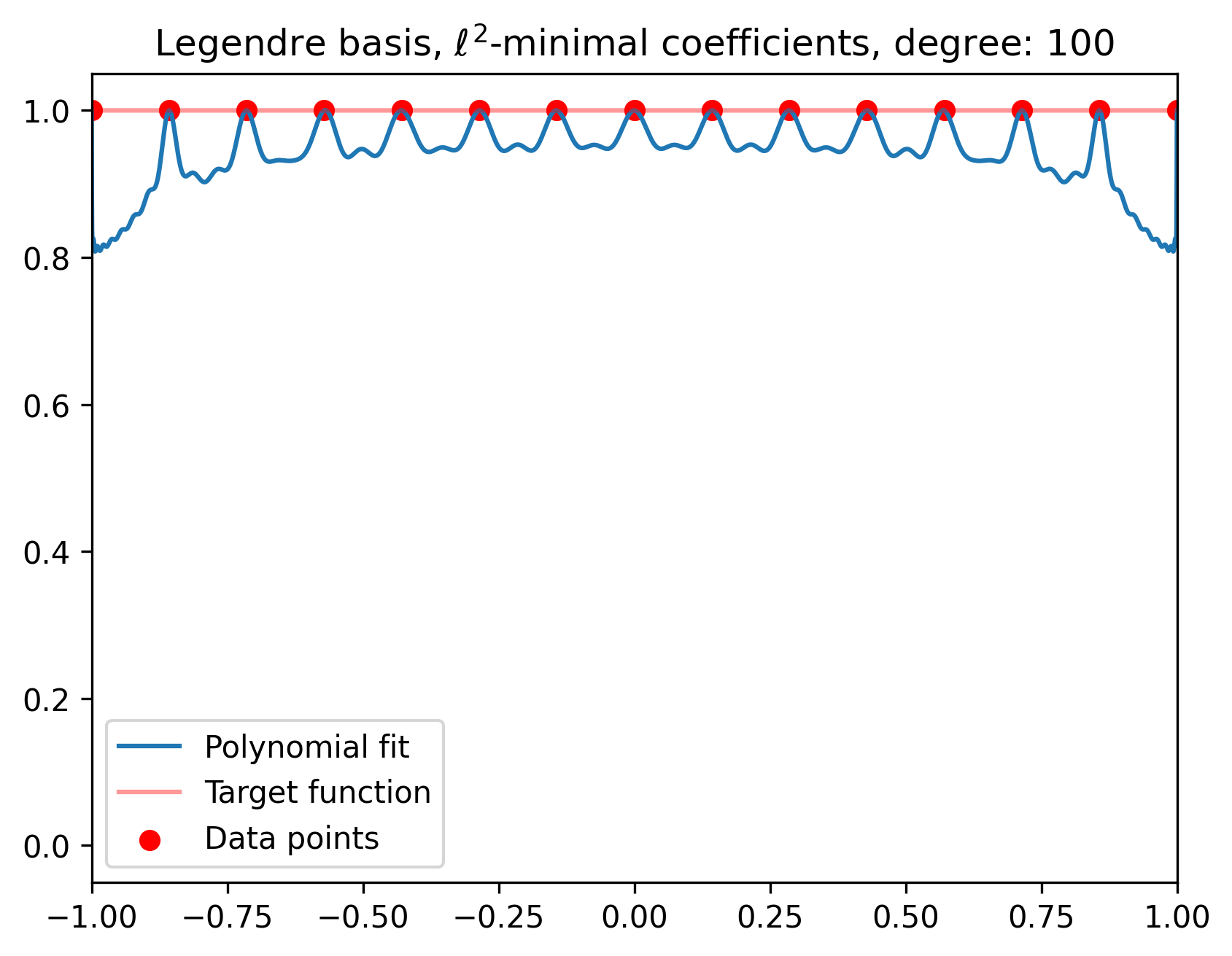}
    \includegraphics[width=0.24\linewidth]{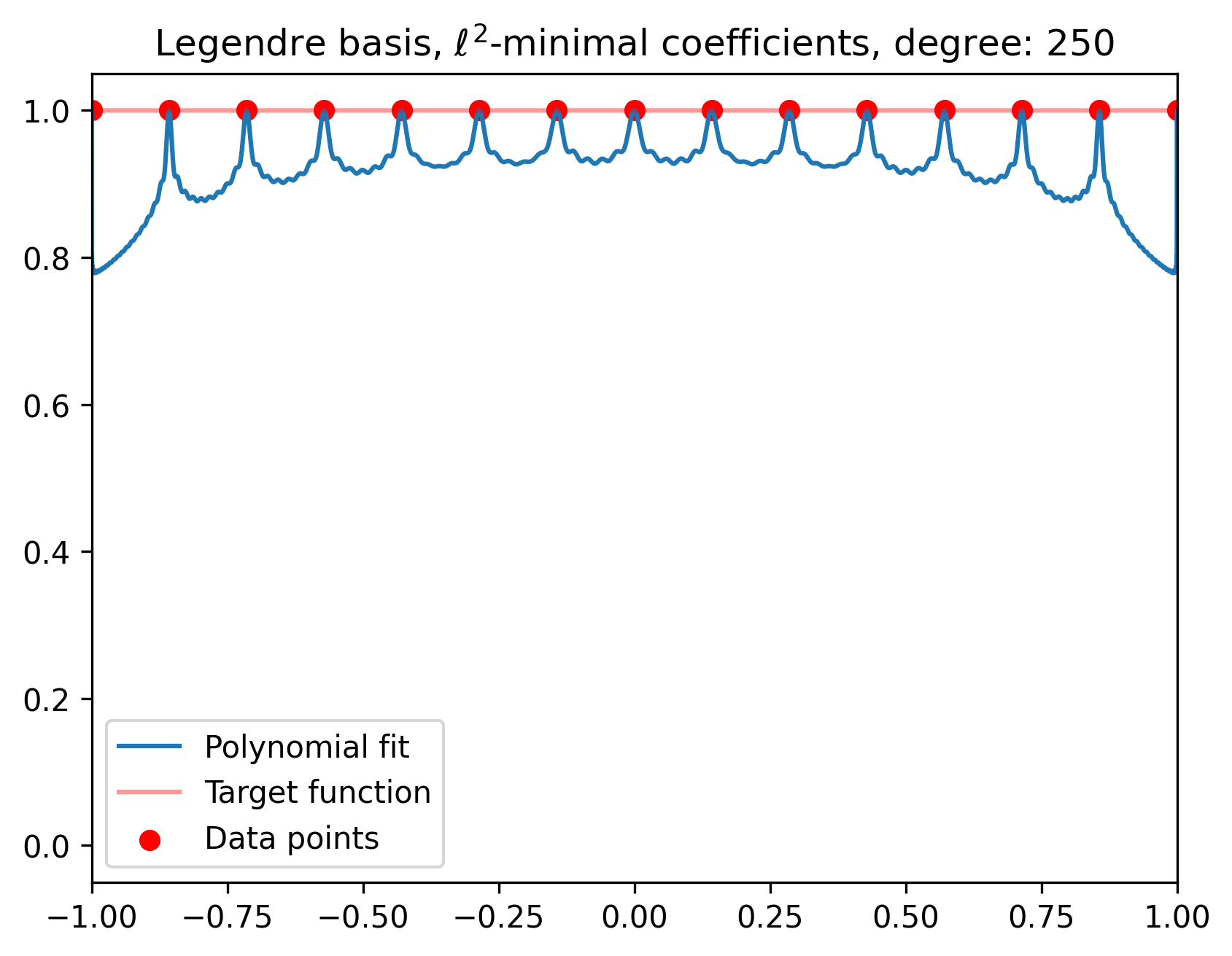}
    \includegraphics[width=0.24\linewidth]{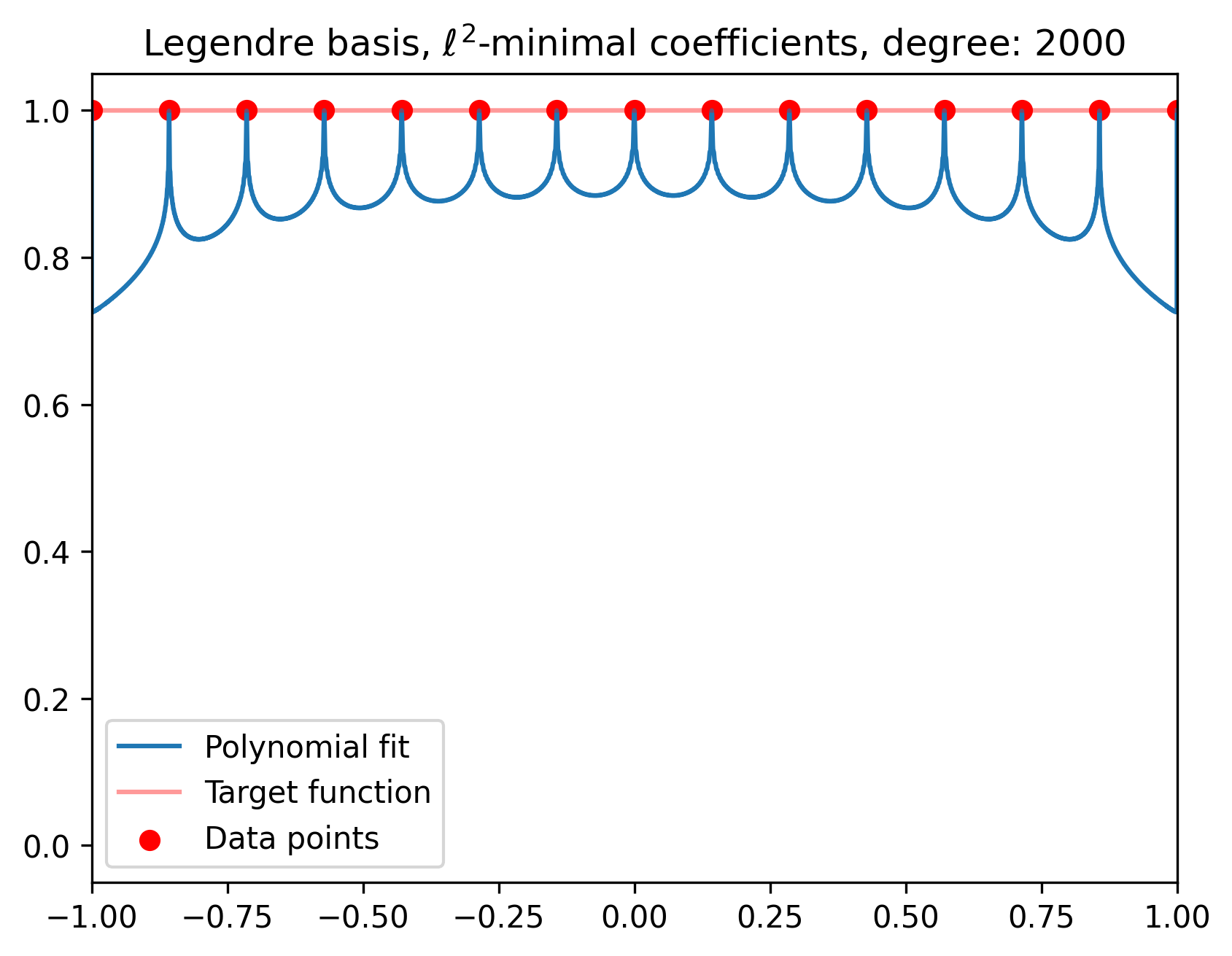}

\caption{\label{figure orthogonal bases}
Polynomial interpolants of the function $f(x) = 1$ with $\ell^2$-minimal coefficients with respect to the Chebyshev basis (top line) and Legendre basis (bottom line), based on 15 equidistant data points. By Sections \ref{section chebyshev} and \ref{section legendre}, both interpolants are expected to tend to zero in $L^2(-1,1)$ as the degree increases (left to right: $d=50, 100, 250, 2000$), but the Chebyshev basis interpolants converge  provably and significantly faster.
}
\end{figure}

The proof does not use specific properties of the Chebyshev polynomials beyond the fact that the first $d+1$ functions span $\mathcal P_d$ and that there exists a weight function $w$ and constant $c$ such that
\[
\inf_{x\in(-1,1)}w(x)>0, \qquad \int_{-1}^1 q_k(x)q_l(x)\,w(x) \dx = 0 \quad\forall\ k\neq l\qquad\text{and}\quad 
\int_{-1}^1 |q_k(x)|^2w(x)\dx \geq c.
\]
It thus holds also e.g.\ for a renormalized Legendre basis $\{c_0p_0, \dots, c_dp_d\}$ where $p_k$ denotes the $k$-th Legendre polynomial and $c_k \sim \sqrt{k}$. For a heuristic consideration of the Legendre basis in standard normalization, see Section \ref{section others} and compare Figure \ref{figure orthogonal bases}.

\begin{proof}
Since $\|a\|^2_{\ell^2}$ is strongly convex on $\R^{d+1}$, there exists a unique coefficient vector $a = a^d$ which minimizes $\|a\|_2$ in the affine subspace specified by the interpolation conditions. For $P = \sum_{k=0}^d a_kp_k$, we compute
\begin{align*}
\|P\|_{L^2(\nu)}^2
    &= \sum_{k, l=0}^d a_ka_l \,\langle p_k, p_l\rangle_{L^2(\nu)} = \pi \,a_0^2 + \frac\pi 2 \sum_{k=1}^d a_k^2  =\frac{\pi}2\|a^d\|^2_{\ell^2} +\frac{\pi}{2}a_0^2.
\end{align*}
We hence aim to show that it is possible to interpolate the data exactly by polynomials with small $L^2(\nu)$-norm if the degree $d$ is high enough. As an auxiliary construct, we define
\[
\eta(x)= \max \{1-|x|, 0\}, \qquad
h_\varepsilon(x)=\sum_{i=0}^n y_i \eta \left ( \frac{x-x_i}{\varepsilon}\right ).
\]
If $\eps < \min_{i\neq j}|x_i-x_j|/2$, the support of the summands in $h_\eps$ does not overlap and
\begin{align*}
\|h_\varepsilon\|_{L^2(\nu)}^2 &= \int_{-1}^1 \frac{|h_\varepsilon(x)|^2}{\sqrt{1-x^2}}dx 
    = \sum_{i=0}^n |y_i|^2 \int_{-1}^1 \frac{\chi_{(x_i-\eps, x_i+\eps)}}{\sqrt{1-x^2}}\dx \leq \sum_{i=0}^n |y_i|^2 \int_{-1}^{-1+2\eps}\frac{1}{\sqrt{(1-x)(1+x)}}\dx\\
    &\leq \left(\sum_{i=0}^n |y_i|^2 \right) \int_0^{2\eps}\frac1{\sqrt{t}}\dt = \left(\sum_{i=0}^n |y_i|^2 \right) 2\,\sqrt{2\eps} \leq 4\left(\sum_{i=0}^n |y_i|^2 \right)\sqrt\eps.
\end{align*}
By the Stone-Weierstrass Theorem, there exists polynomial $Q$ such that $\|h_\varepsilon - Q\|_{L^\infty(-1,1)} <\eps$
and in particular
\begin{enumerate}
\item $|h_\varepsilon(x_i)-Q(x_i)| < \eps$ for $i=0, \ldots, n.$
\item $\|h_\eps - Q\|_{L^2(\nu)} \leq \|h_\eps -Q\|_{L^\infty(-1,1)}\|\nu\|_{TV} \leq \pi\eps$.
\end{enumerate}
However, $Q$ may not interpolate the data pairs $(x_i, y_i)$ exactly. Define the Lagrange polynomials
\[
L_i(x)=\prod_{j\neq i}\frac{x-x_j}{x_i-x_j},\qquad L_i(x_k)=\delta_{ik}.
\]
Then the modified polynomial $P$ of degree $\min\{\deg Q, n\}$ given by
\[
P(x) := Q(x) + \sum_{i=0}^n \big(y_i - Q(x_i)\big)\,L_i(x)
\]
fits the data exactly and
\begin{align*}
\|Q\|_{L^2(\nu)} &\leq \|Q-P\|_{L^2(\nu)} + \|P-h_\eps\|_{L^2(\nu)} + \|h_\eps\|_{L^2(\nu)}\\
    &\leq \sum_{i=0}^n \big|Q(x_i) - y_i\big|\,\|L_i\|_{L^2(\nu)} + \pi \eps + \sqrt{4\sum_{i=0}^n |y_i|^2}\,\eps^{1/4}\\
    &\leq n\eps \,\max_{i=0,\dots, n} \|L_i\|_{L^2(\nu)} + \pi \eps + \sqrt{4\sum_{i=0}^n |y_i|^2}\,\eps^{1/4}.
\end{align*}
The first term is finite since all Lagrange polynomials are bounded on the compact set $[-1,1]$ and $C^0([-1,1])\embeds L^2(\nu)$. The bound depends on $n$ and the dataset, but not $d$ or $\eps$. Thus, for any given $\delta>0$ we may choose $\eps$ small enough to ensure that $\|Q\|_{L^2(\nu)} < \delta$ if $d$ is large enough.
\end{proof}

Thus, for large $d$, $P_d$ approximates $f$ about as well as the constant zero function. If $\int_{-1}^1 f(x)\dx \neq 0$, other constant functions are better approximators of $f$ and can be found in the {\em underparametrized} setting. For comparison, if $d\leq n$ we consider the unique least squares polynomial $P_d \in \mathcal P_d$ which minimizes the mean squared error
\[
P\mapsto \frac1{n+1}\sum_{i=0}^n \big|P(x_i) - y_i\big|^2 
\]
over the dataset.

\begin{corollary}[Suboptimal Asymptotic Descent]\label{corollary suboptimal descent}
Let $-1 \leq x_i \leq  1$ for $i= 0, \dots, n$ and $y_i = f(x_i)$ for a target function $f:[-1,1]\to\R$. If
\[
\left|\frac1{n+1}\sum_{i=0}^n f(x_i) - \frac 12\int_{-1}^1f(x)\dx\right| < \left|\frac12\int_{-1}^1 f(x)\dx\right|,
\]
then $\|f-P_0\|_{L^2(-1,1)} < \lim_{d\to\infty} \|f-P_d\|_{L^2(-1,1)} = \|f\|_{L^2(-1,1)}$.
\end{corollary}

In other words, if $\frac1{n+1} \sum_{i=0}^n f(x_i)$ has the same sign as the average of $f$ over the interval $(-1,1)$ and is not larger in magnitude than the true average by a factor of more than 2, then $P_0$ is a better approximation of $f$ than $P_d$ for high degree $d$. Similar results hold for weighted $L^2$-norms.

If $f$ is for instance a Runge type function, this indicates a true double descent phenomenon: $\|f-P_0\|_{L^2}$ and $\lim_{d\to\infty} \|f-P_d\|_{L^2(-1,1)}$ are both small, but $\|f-P_n\|_{L^2(-1,1)}$ is large for $d=n$ at the interpolation threshold. On the other hand, the descent at infinity is to a suboptimal value and {\em worse} than the approximation performance even of constant functions.

\begin{proof}
We have that $P_0=\frac{1}{n+1}\sum_{i=0}^n f(x_i)$ is the minimizer of the empirical mean squared error
\[
\widehat{MSE}_n(c) = \frac1{n+1}\sum_{i=0}^n \big(f(x_i) - c\big)^2.
\]
The true mean squared error 
\[
MSE(c) = \frac12 \int_{-1}^1 \big(f(x) - c\big)^2\dx = c^2 - c\int_{-1}^1 f(x)\dx + \frac12\int_{-1}^1 f^2(x)\dx
\]
is a parabola in $c$ and $MSE(c) < MSE(0)$ if and only if
\[
c\left(c- \int_{-1}^1f(x)\dx\right) = c^2 - c\int_{-1}^1f(x)\dx = MSE(c) - MSE(0) < 0.
\]
If the average of $f$ over the interval is positive, this reduces to $0<c < \frac22 \int_{-1}^1 f(x)\dx$ and similarly if the average is negative. In either case, the criterion can be written as
\[
\left|c - \frac12 \int_{-1}^1f(x)\dx\right| < \left|\frac12 \int_{-1}^1f(x)\dx\right|.
\]
Thus $P_0$ strictly outperforms $0 = \lim_{d\to\infty} P_d$ under the conditions of Corollary \ref{corollary suboptimal descent}.
\end{proof}

\section{Other Regularizations}\label{section others}

In this section, we explore two alternative regularizations more briefly and somewhat heuristically. Specifically, we focus on $\ell^2$-regularization for the basis coefficients in monomial basis (Section \ref{section monomial ell2} and Legendre basis (Sections \ref{section legendre} and \ref{section legendre sobolev}). For completeness, we also present numerical examples of interpolants with $\ell^1$-minimal coefficients with respect to all bases considered in Figure \ref{figure basis comparison ell1}.

All regularizations explored in this note are naturally adapted to data in the interval $[-1,1]$. We leave polynomial classes like Laguerre polynomials (for exponentially distributed data) or Hermite polynomials (for normally distributed data) for future work.

\begin{figure}
    \centering
\includegraphics[width=.32\textwidth]{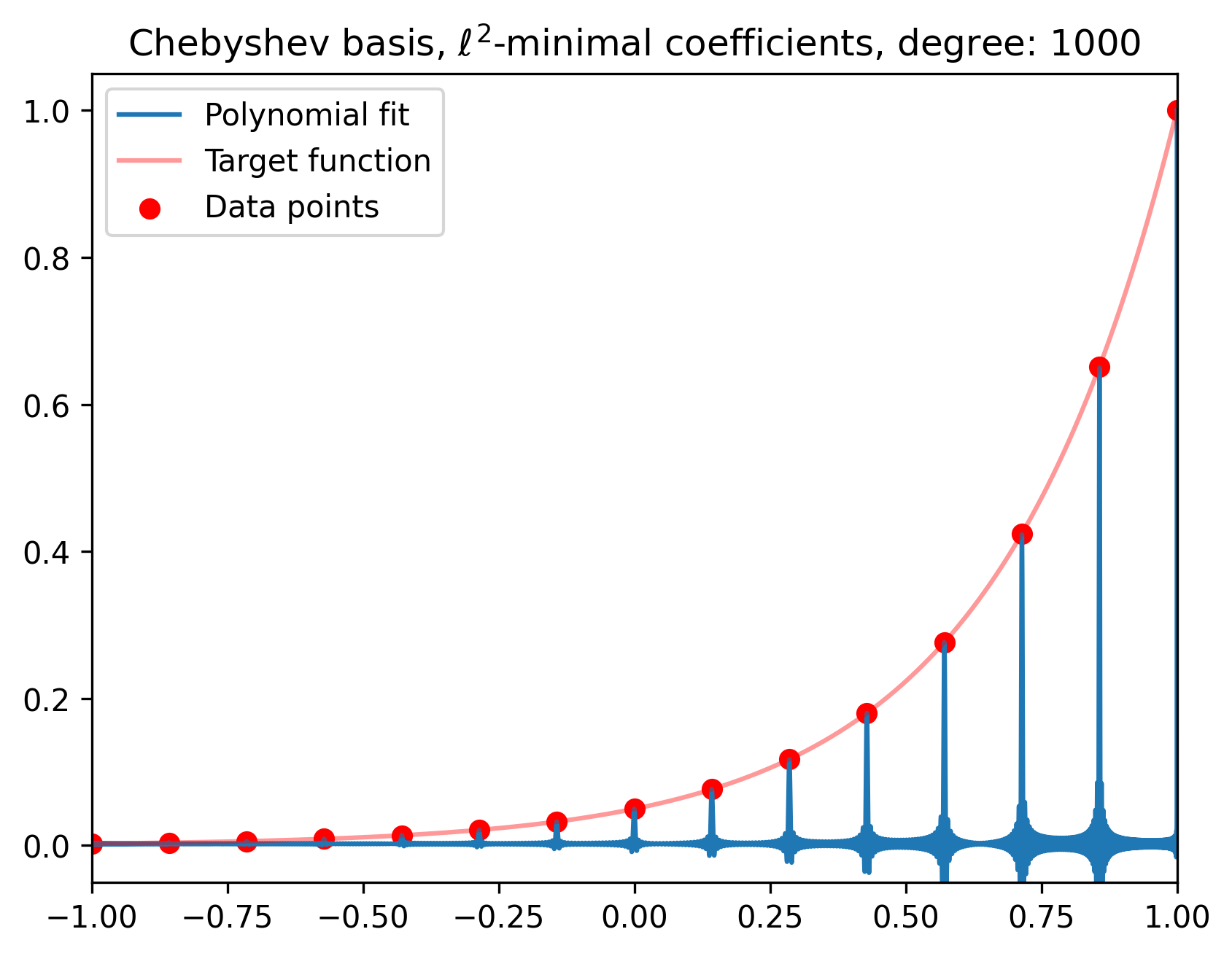}
\includegraphics[width=.32\textwidth]{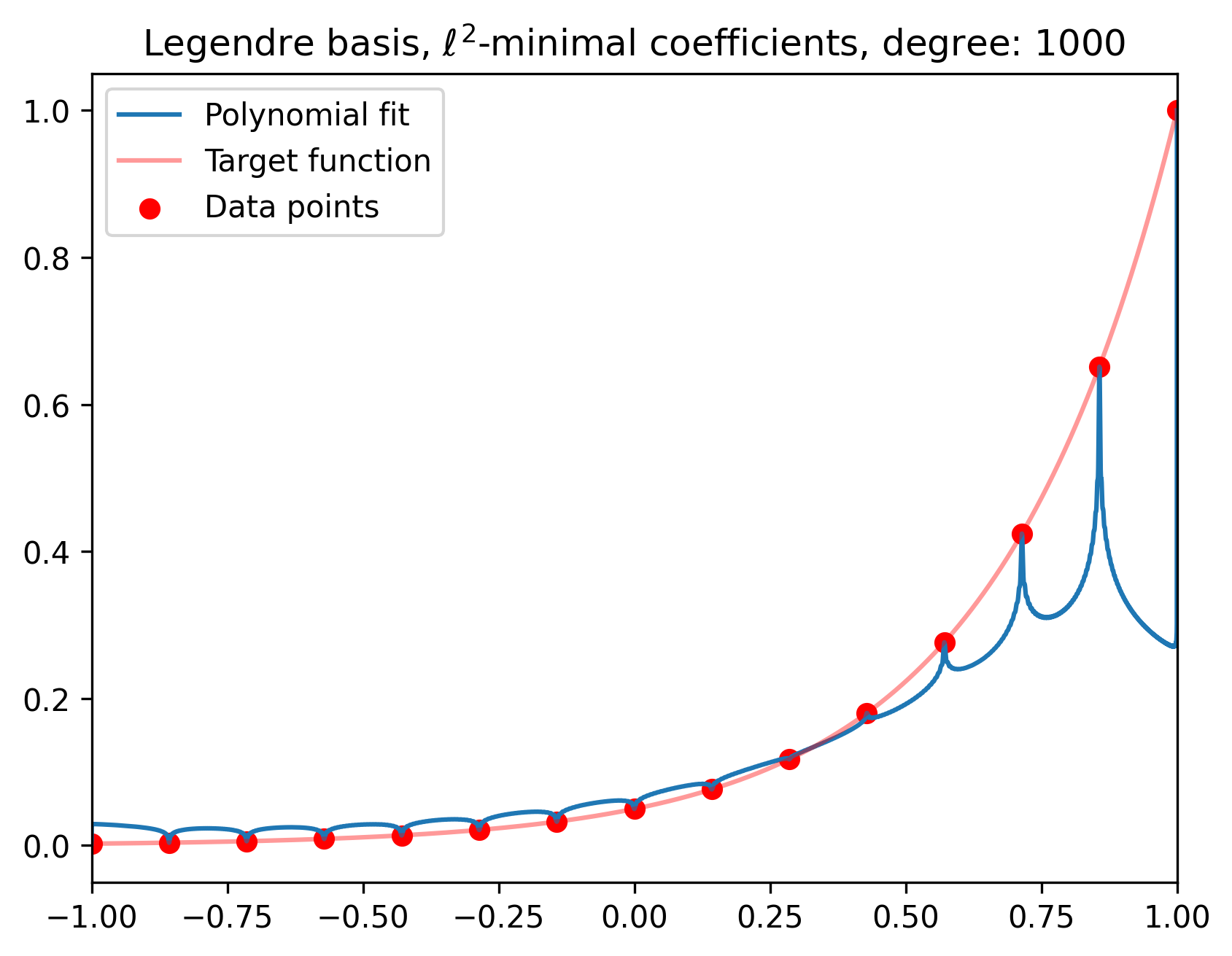}
\includegraphics[width=.32\textwidth]{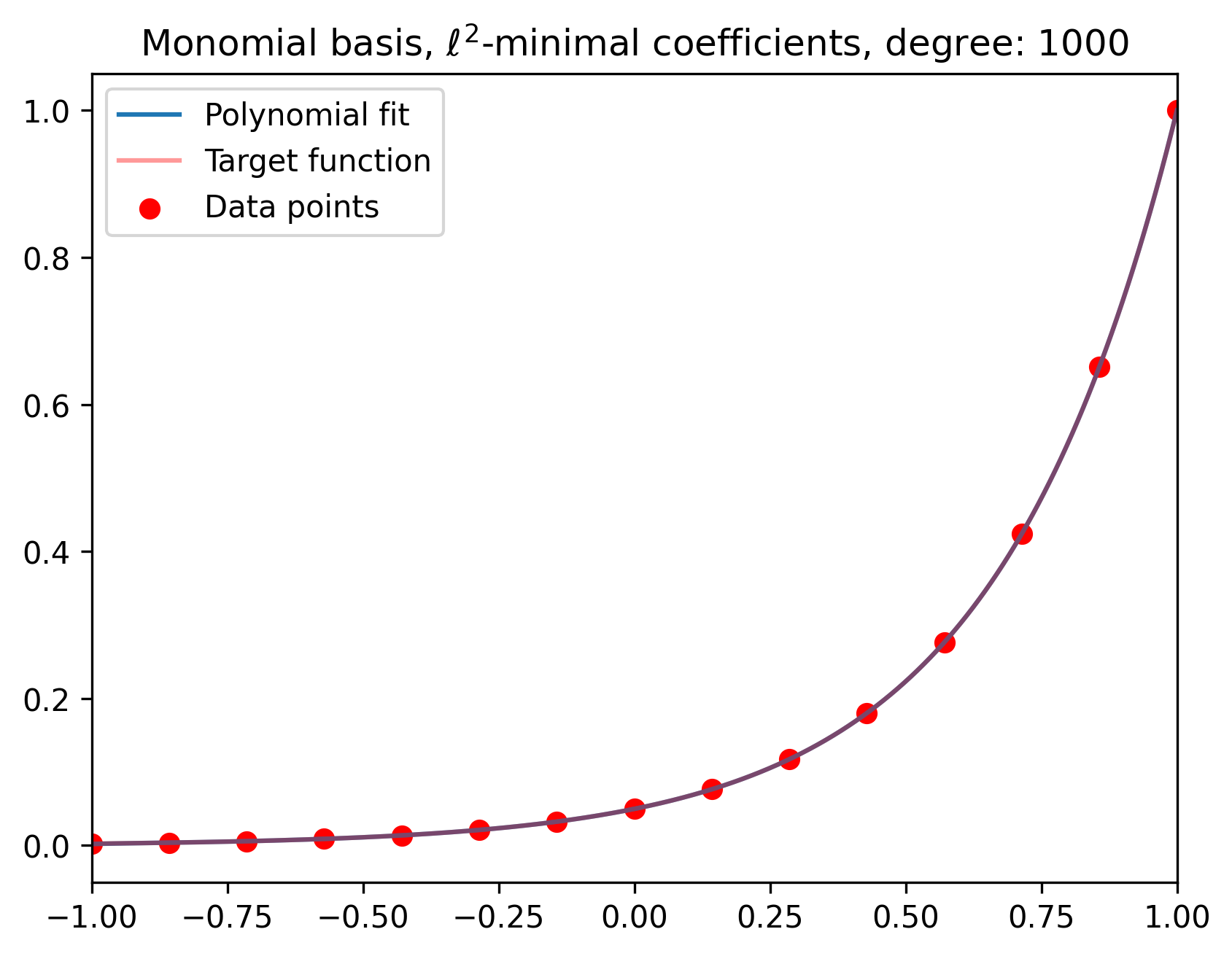}
\vspace{3mm}

\includegraphics[width=.32\textwidth]{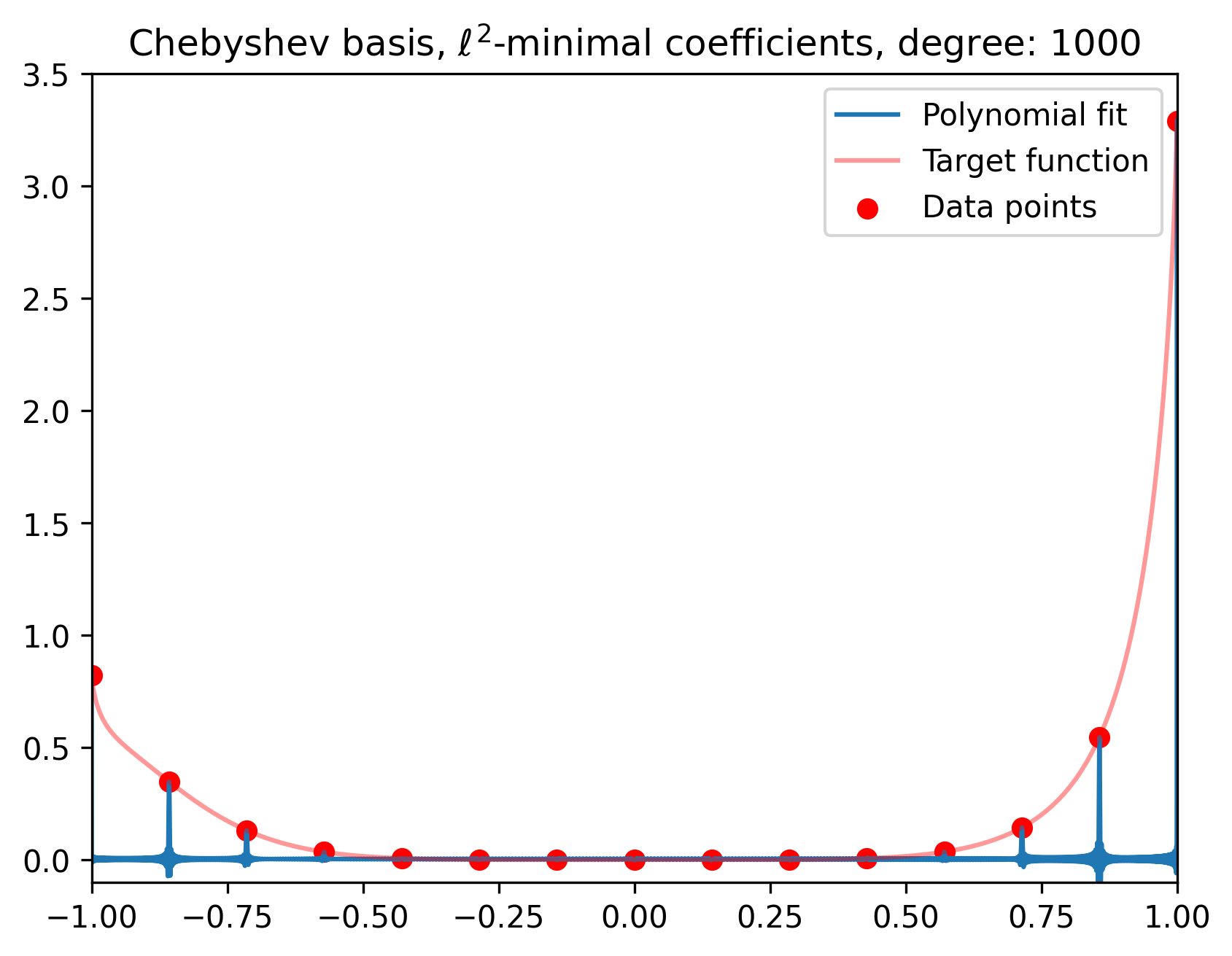}
\includegraphics[width=.32\textwidth]{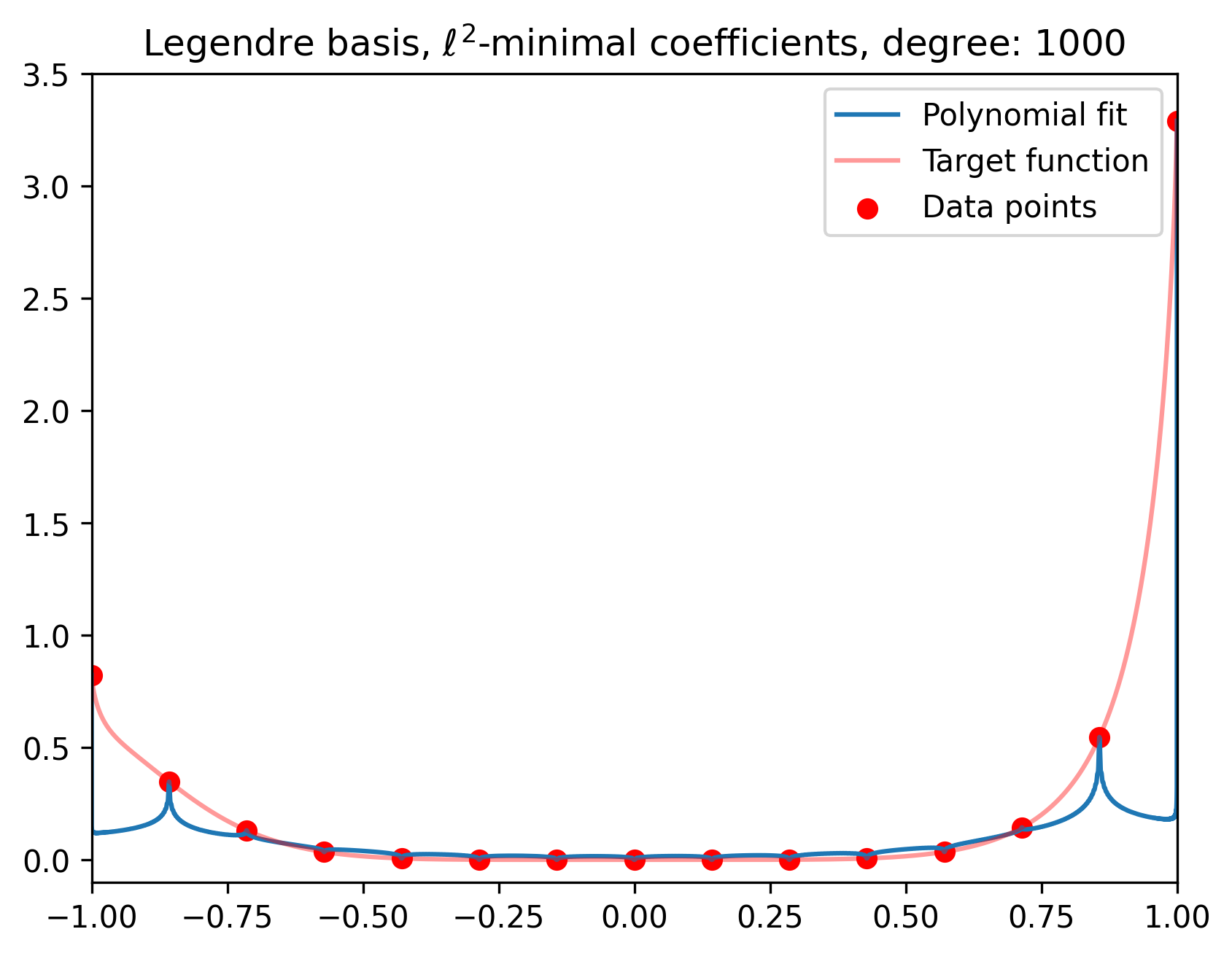}
\includegraphics[width=.32\textwidth]{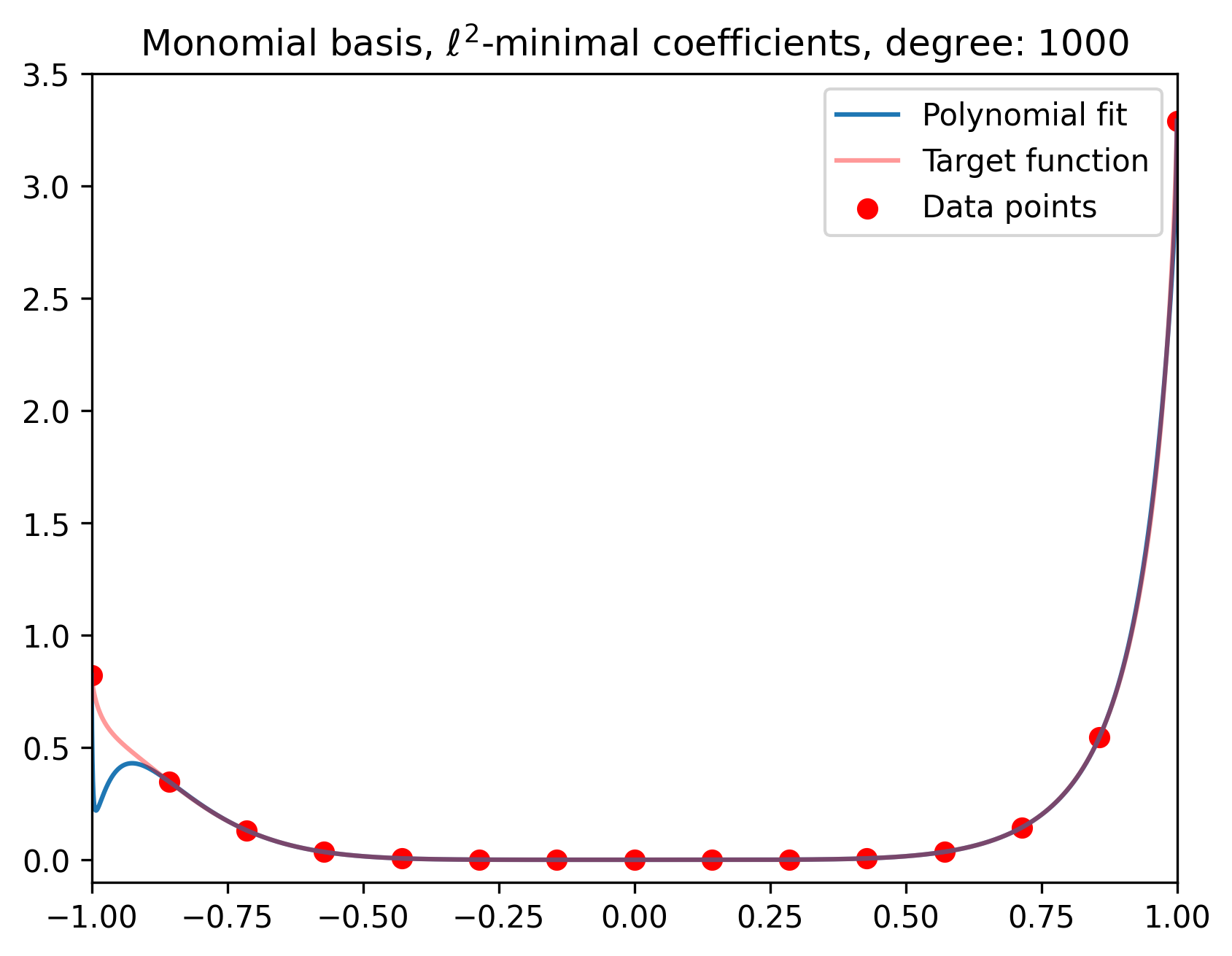}

\vspace{3mm}

\includegraphics[width=.32\textwidth]{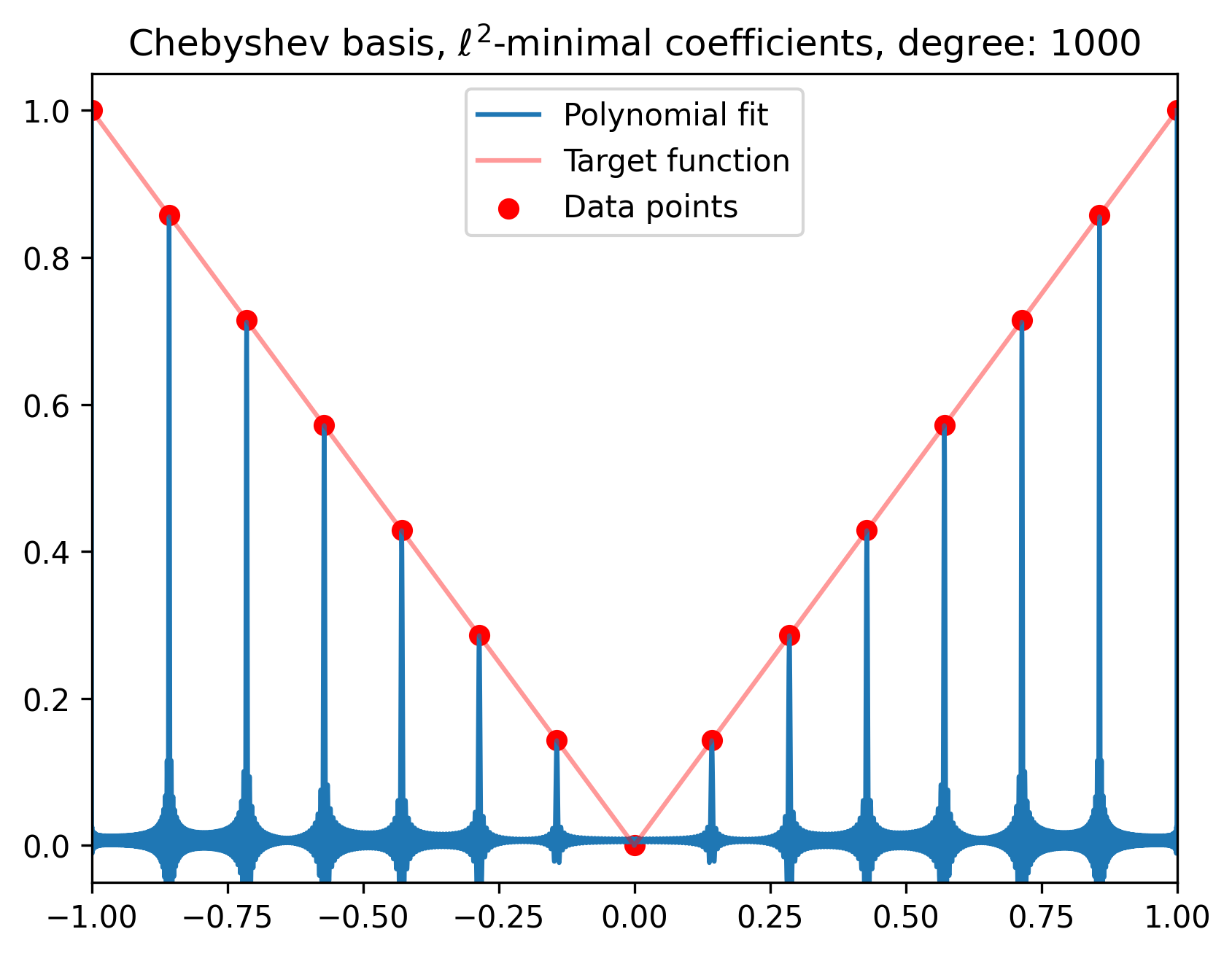}
\includegraphics[width=.32\textwidth]{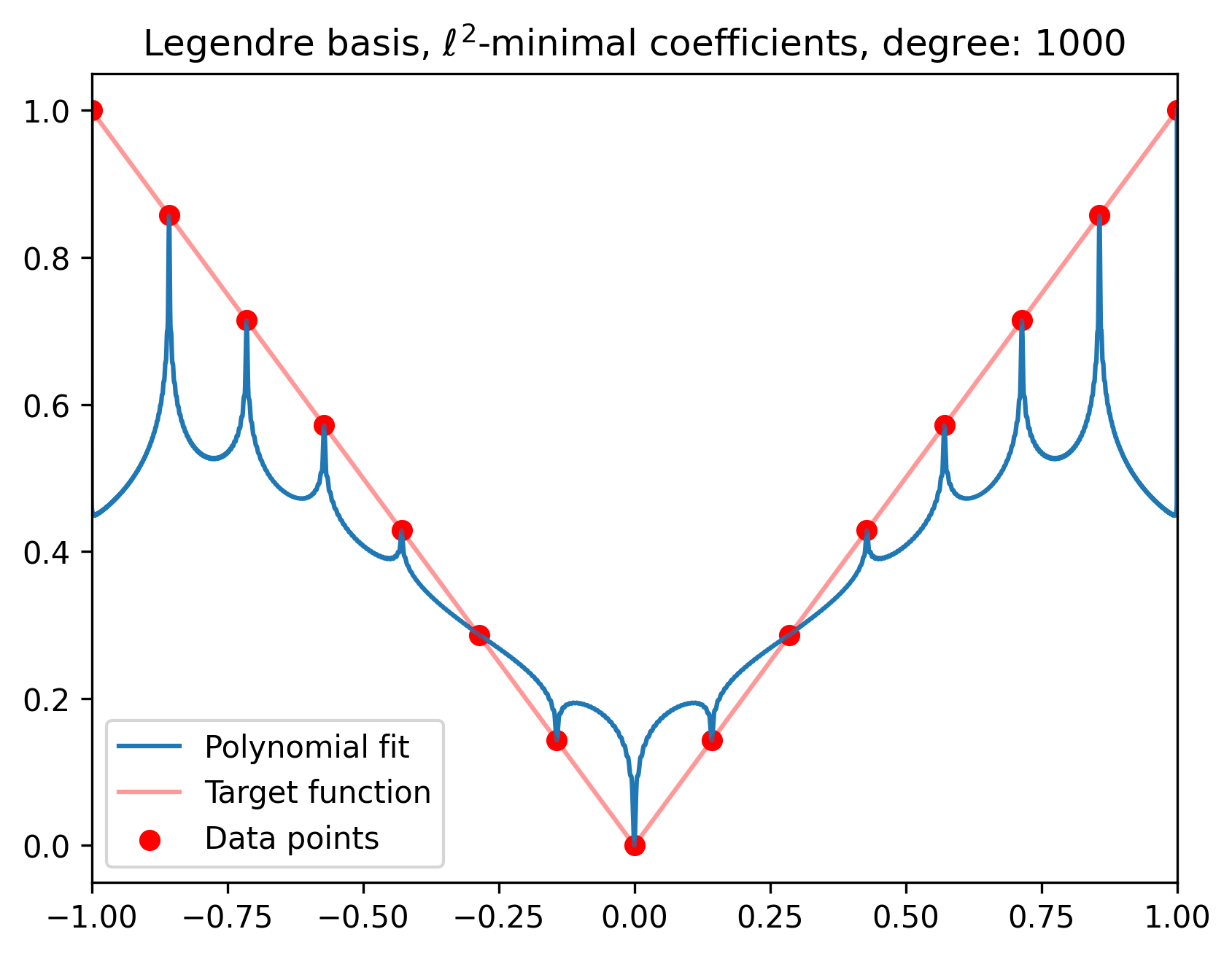}
\includegraphics[width=.32\textwidth]{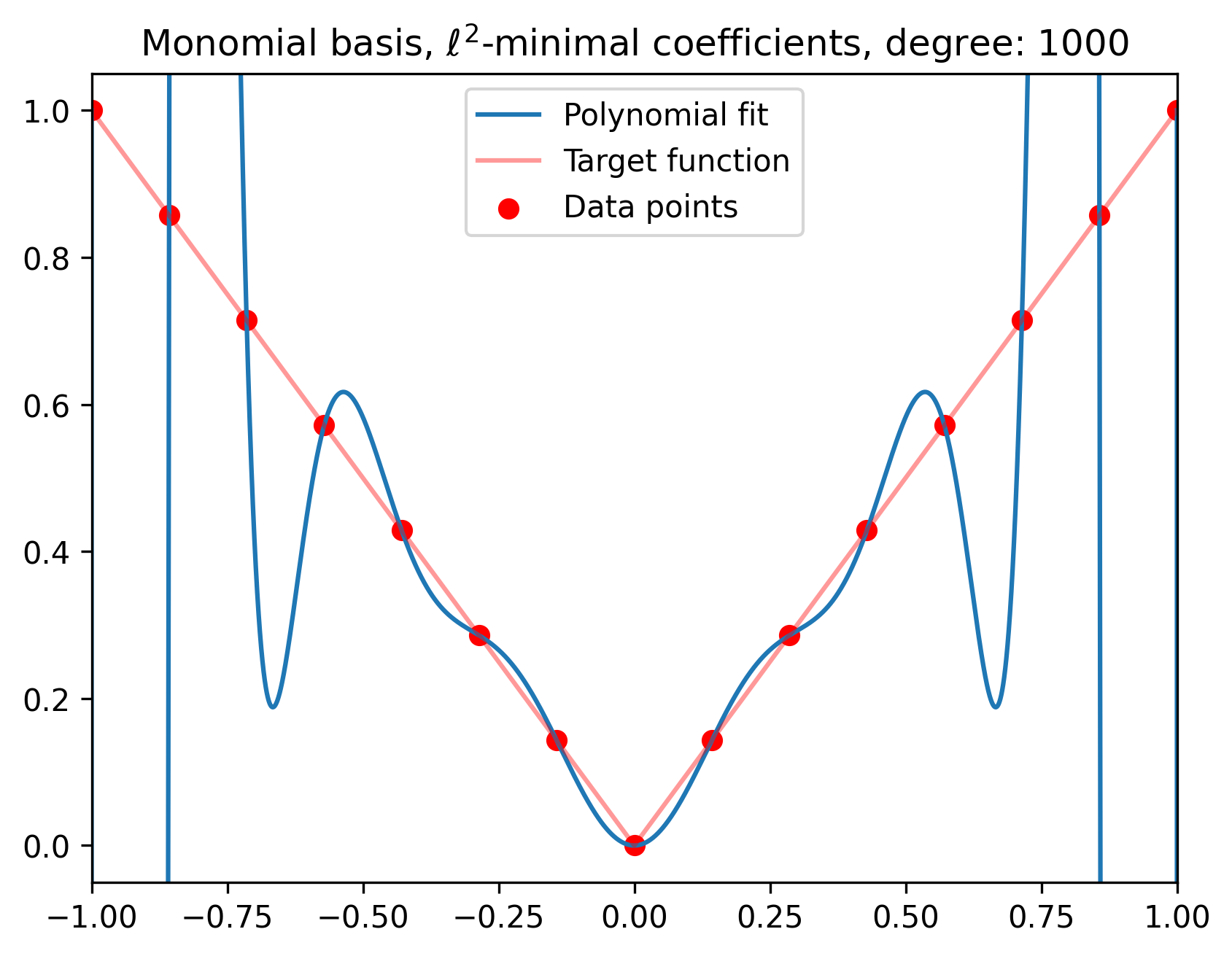}
\vspace{3mm}

\includegraphics[width=.32\textwidth]{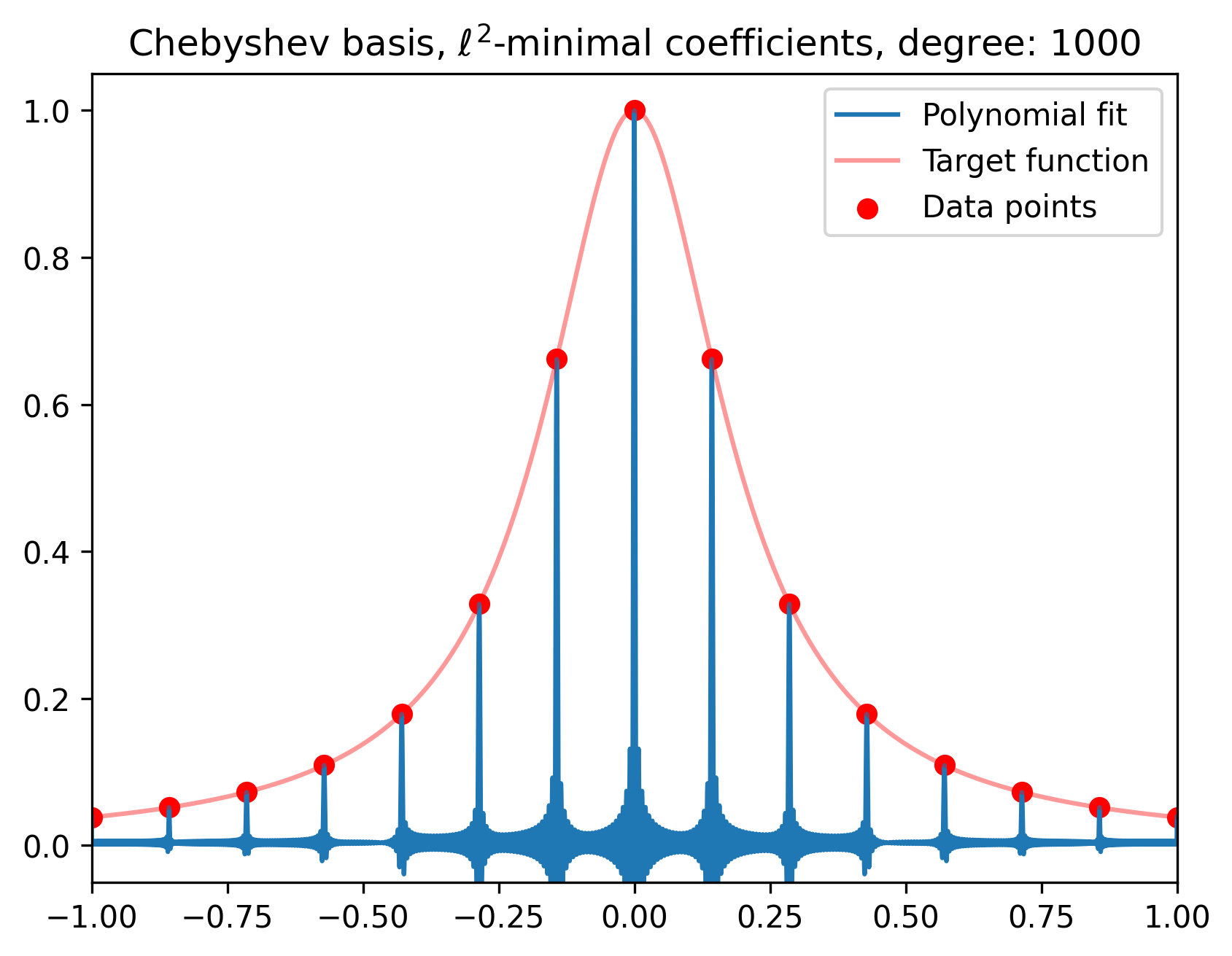}
\includegraphics[width=.32\textwidth]{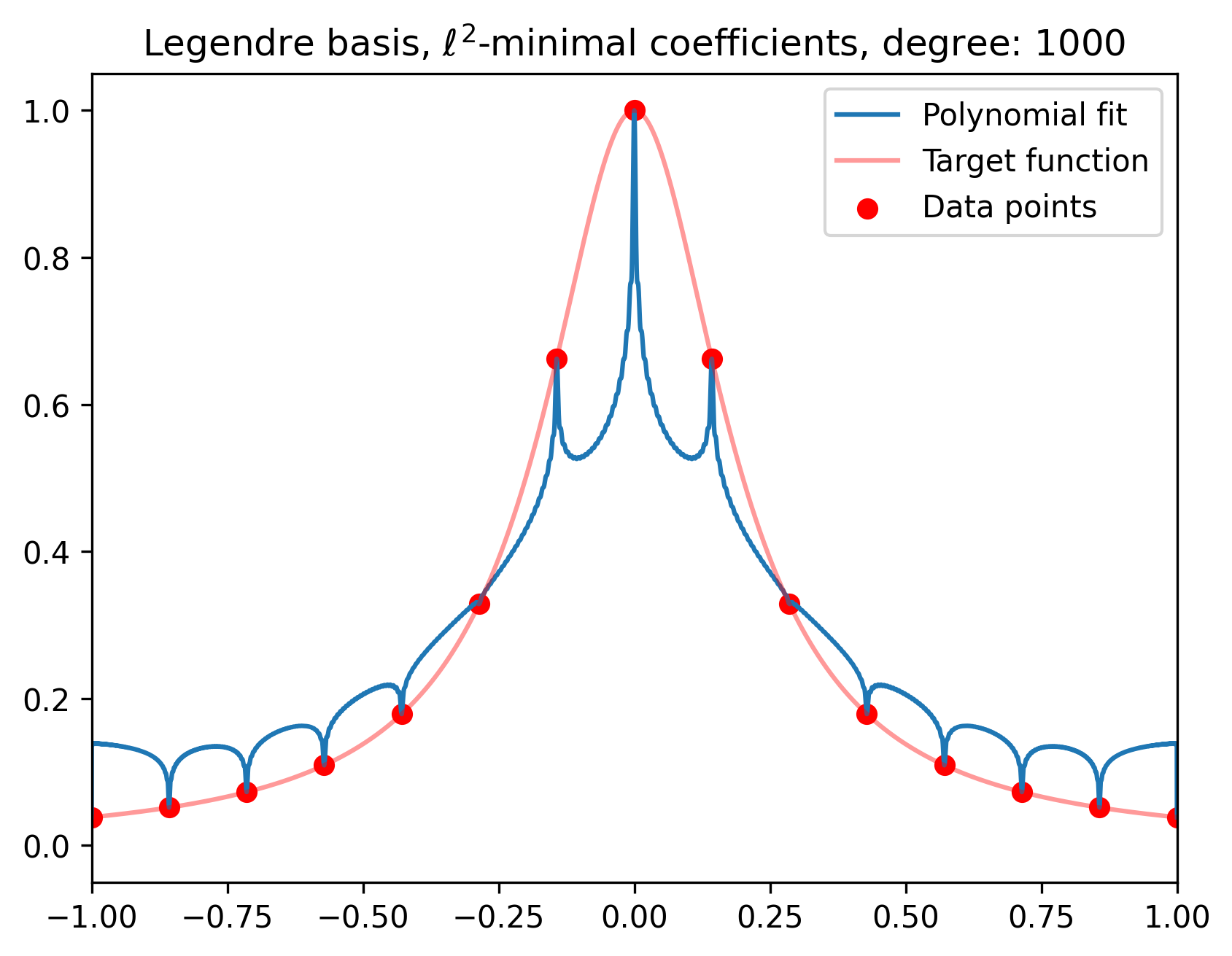}
\includegraphics[width=.32\textwidth]{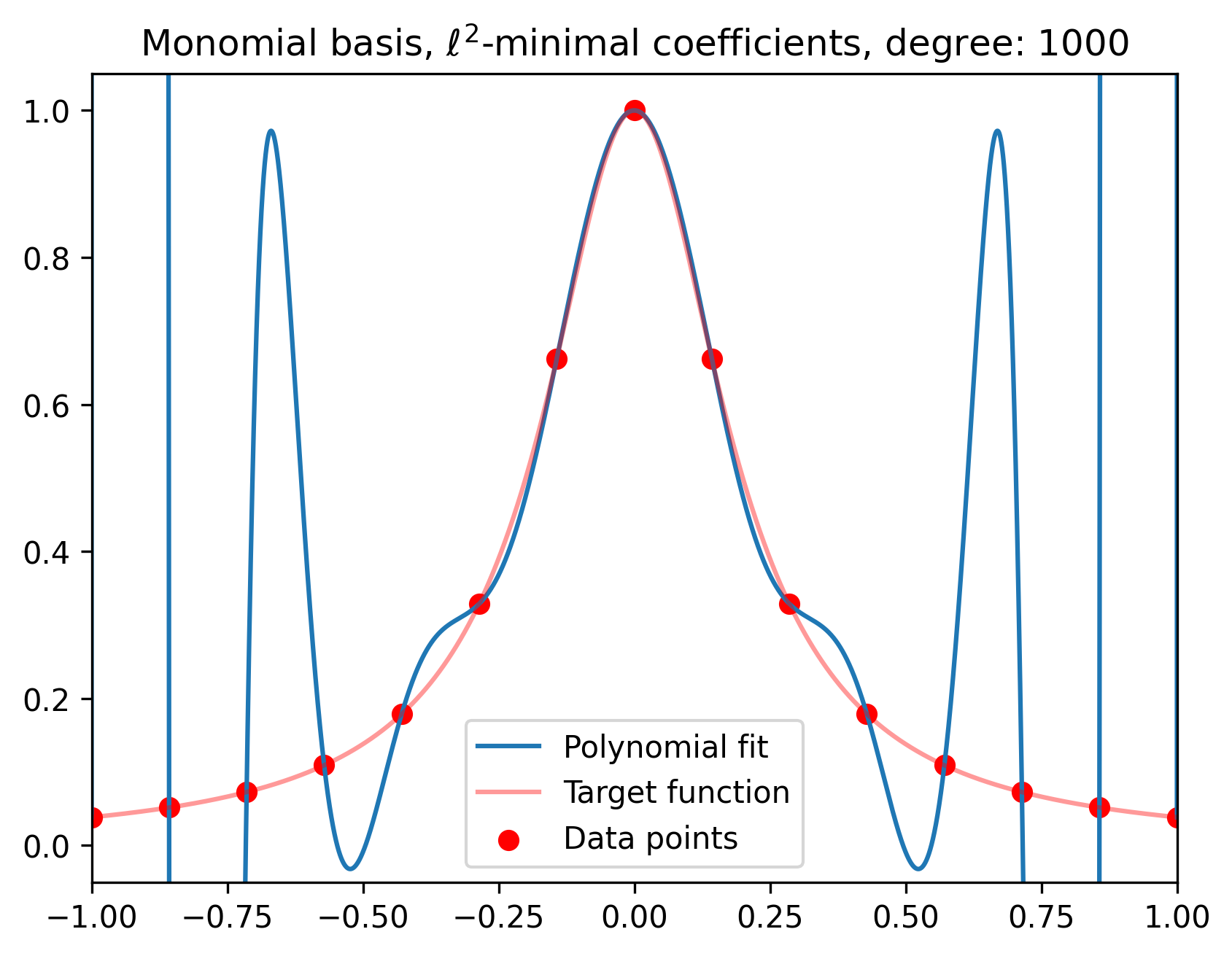}
\vspace{3mm}
    \caption{\label{figure basis comparison}
    Interpolating input-output pairs $x_i,\ y_i = f(x_i)$ at 15 equidistant points in $[-1,1]$ by polynomials of degree $d = 1,000$ with $\ell^2$-minimal coefficients for various functions and bases. First row: $f(x) = \exp(3x-3)$, second row: $f(x) = Li_2(x^{15}) + Li_2(x^6)$, third row: $f(x) = |x|$, fourth row: $f(x) = \frac1{1+25x^2}$. First column: Chebyshev basis, second column: Legendre basis, third column: Monomial basis.\\ 
    Interpolants in Chebyshev basis are always far from $f$, but do not blow up. Interpolants in Legendre basis are not close, but always roughly follow the function for degree $d=1,000$. Interpolants in monomial basis are spectacularly good approximants for the exponential series, even if the coefficient vector is $\ell^2$-minimal, not $\ell^1$-minimal. For the sum of dilogarithms, on the other hand, we find a small `boundary oscillation' where the $\ell^1$-optimal interpolant would have none.}
\end{figure}

\subsection{\texorpdfstring{Monomial Basis: $\ell^2$-coefficient norm}{Monomial Basis: l2-coefficient norm}}
\label{section monomial ell2}

We pursue the complex analytic perspective we took in the introduction for power series with summable coefficients.
As with $\ell^1$-regularization, if $a\in \ell^2(\mathbb C)$ is a sequence with square summable complex entries, the power series $f(z) = \sum_{k=0}^\infty a_k z^k$ has radius of convergence $r\geq 1$ since
\[
\sum_{k=0}^\infty |a_kz^k| = \sum_{k=0}^\infty |a_k|\,|z|^k \leq \left(\sum_{k=0}^\infty |a_k|^2\right)^{1/2}\left(\sum_{k=0}^\infty |z|^{2k}\right)^{1/2} = \frac{1}{\sqrt{1- |z|^2}} \,\|a\|_{\ell^2},
\]
i.e.\ the power series converges absolutely (and hence, converges) for $|z|<1$.
We consider this a natural class for monomial basis coefficient $\ell^2$-regularization.

Denote by $\mathbb D$ be the unit disk in the complex plane $\mathbb C$. Given a collection of data pairs $(z_i, w_i) \in \mathbb D^\circ \times \mathbb C$ with $i = 0,\dots, n$, we consider the problem
\begin{equation}\label{eq hardy space}
a = \argmin \left\{ \sum_{k=0}^\infty |a_k|^2  : a\in \ell^2(\mathbb C) \text{ s.t.\ } g_a(z_i) = w_i \text{ for all } i = 1,\dots, n\right\} \quad \text{where}\quad g_a(z) = \sum_{k=0}^\infty a_k z^k.
\end{equation}
The function $g_a$ associated to the parameter vector $a$ is the interpolating function for the given data which has minimal norm in the Hardy space $H^2$. The minimization problem is well-posed if the data points $z_i$ are inside the open unit disk where Hardy functions are analytic and a fortiori continuous. We consider the approximation of the Hardy space problem by polynomials
\begin{equation}\label{eq hardy polynomial}
a_d = \argmin \left\{ \sum_{k=0}^d |a_k|^2  : a\in \mathbb C^{d+1} \text{ s.t.\ } g_a(z_i) = w_i \text{ for all } i = 1,\dots, n\right\} \qquad \text{where}\quad g_a(z) = \sum_{k=0}^d a_k z^k.
\end{equation}
We establish well-posedness for a fixed dataset as the polynomial degree becomes unbounded.

\begin{theorem}\label{theorem hardy}
Assume that $|z_i|< 1$ for all $i= 1,\dots, n$ and $d\in \mathbb N$. Then there exist a unique solutions $a^*, a_d$ to the minimization problems \eqref{eq hardy space} and \eqref{eq hardy polynomial} respectively. 
\begin{itemize}
\item With the embedding $I_d: \mathbb C^{d+1}\to \ell^2(\mathbb C)$, $I_da = (a_0, a_1, \dots, a_d, 0, 0, \dots)$, we have that $I_da_d \to a^*$ strongly in $\ell^2(\mathbb C)$ and $g_{a_d}\to g_{a^*}$ uniformly on the closed disk $\overline{B_r(0)}\subset \mathbb C$ for all $r\in (0,1)$.

\item The minimum norm interpolants have the Szeg\H{o} type kernel representation
\[
g_{a_d}(z) = \sum_{i=0}^n b_{d,i}\,\frac{1 - (z\bar z_i)^{d+1}}{1- z\bar z_i}, \qquad g_{a^*}(z) = \sum_{i=0}^n b_{i}^*\,\frac{1}{1- z\bar z_i}
\]
for some $b_d, b^*\in \R^{n+1}$.

\item If $z_i, w_i \in \R$ for all $i=1,\dots, n$, then so are $a^*, a_d, b^*$ and $b_d$.
\end{itemize}
\end{theorem}

The proof relies on the weak topology of reflexive and separable Hilbert spaces, see e.g.\ \cite[Chapter 3]{brezis2011functional}. As in the proof of Theorem \ref{theorem no sign condition}, we denot $z^\N = (1,z,z^2,\dots) \in \ell^1(\C)$.

\begin{proof}
The existence of a solution $a^*\in \ell^2(\C)$ follows from the direct method of the calculus of variations: We use that
\begin{enumerate}
\item the squared norm is lower semi-continuous under weak convergence,
\item bounded subsets of $\ell^2(\mathbb C)$ are weakly sequentially compact, and
\item the interpolation condition $\langle a, z_i^{\mathbb N}\rangle = y_i$ is preserved under weak convergence.
\end{enumerate}
A minimizing sequence hence converges to a limit point (in the weak topology) up to subsequence, and the limit point is a minimizer. A variation of the proof is given with more detail for Theorem \ref{theorem no sign condition}. The uniqueness of $a^*$ follows from the strict convexity of the function $a\mapsto \|a\|_2^2$.

By construction, we see that the sequence of norms $\|I_da_d\|_2^2$ is monotone decreasing in $d$ and $\|I_da_d\|_2^2 \geq \|a^*\|_2^2$ for all $d\in \mathbb N$. On the other hand, we find that $\lim_{d\to\infty} \|a_d\|_2^2 \leq \|a^*\|^2_2$ by the same arguments as in Lemma \ref{existence lemma}.

From the weak compactness of the unit ball in $\ell^2(\mathbb C)$, we find that any subsequence of $I_da_d$ has a further subsequence which converges to a limit point $\tilde a$ such that
\begin{itemize}
    \item $g_{\tilde a}(z_i) =\lim_{d\to\infty} g_d(z_i) = \lim_{d\to\infty} \langle I_da_d, z_i^\N\rangle = w_i$ for all $i$ and
    \item $\|\tilde a\|_2^2 \leq \|a^*\|_2^2$
\end{itemize}
since the norm is weakly lower semi-continuous and the linear functionals $a\mapsto \langle a, z^{\mathbb N}\rangle$ are continuous on $\ell^2(\mathbb C)$ for $|z|<1$. Since $a$ is the unique such point, we conclude that every subsequence of $I_da_d$ has a further subsequence which converges weakly to $a^*$. Since weak convergence is convergence in the weak topology, we conclude that the whole series converges.

Since additionally $\|a^*\|_2 = \lim_{d\to\infty} \|I_da_d\|_2$, the sequence in fact converges strongly. The uniform convergence on compact disks follows as in Theorem \ref{theorem no sign condition} with a minor modification.

The representation formula can be derived from the representer theorem or `kernel trick' \cite[Theorem 16.1]{shalev2014understanding} with feature map
\[
\phi(x) = \begin{pmatrix}1\\x\\\vdots\\ x^d\end{pmatrix} \qquad \text{ and kernel } K(x,x') = \phi(x)^T \phi(x') = \sum_{k=0}^d (xx')^d = \frac{ 1- (xx')^d}{1-xx'}.
\]
The fact that the weights $a,b$ can be chosen to be real follows from the fact that $Re(a),\ Re(b)$ also satisfy the interpolation conditions if $w$ is real, but have strictly lower norm if $a, b$ have a non-zero imaginary part.
\end{proof}

In this note, we do not explore rigorously whether a Runge type phenomenon arises in this setting, but numerical evidence in Figure \ref{figure basis comparison} suggests that it does at times when the target function is not in $H^2$. This is potentially unsurprising as $-1, 1$ are among the data points here and the limiting Szeg\H{o} kernel becomes unbounded at the boundary (even as $K_d$ remains finite for $d\in\mathbb N$). On the positive side, we note the following.

\begin{theorem}\label{theorem hardy 2}
Assume that $w_i = f(z_i)$ and the sets $\mathcal X_n = \{z_{0,n}, \dots, z_{n,n}\}$ converge to a compact limit $K\subseteq \mathbb D$ in Hausdorff distance which has an accumulation point in the open disk $\mathbb D^\circ$. Let $a_n$ the $\ell^2$-minimal element of $\ell^2(\mathbb C)$ which interpolates $f$ on $\mathcal X_n$. Then $a_n$ remains bounded in $\ell^2(\mathbb C)$ if and only if there exists $f\in H^2$ such that $f^*\equiv f$ on $K$. If $f^*\in H^2$, then $f_{a_n} \to f^*$ strongly in $H^2$.
\end{theorem}

However, to the best of our knowledge it remains open for which functions $f_{a_d}\to f^*$ in $L^2$ or uniformly despite the fact that $a_d$ may diverge in $\ell^2$.

\begin{proof}
The proof follows that of Theorem \ref{theorem no sign condition} with minor modifications. As previously, strong convergence follows from the fact that $a_n \wto a^*$ weekly and $\|a_n\| \to \|a\|$. The main difference is that $f$ no longer satisfies an $L^\infty$-bound on $\partial \mathbb D$, but only satisfies an $L^2$-bound due to the orthonormality of $\{\phi\mapsto e^{id\phi}\}_{d\in\mathbb N_0}$.
\end{proof}

\subsection{Legendre Basis: An RKHS approach} \label{section legendre}
The Legendre basis $\{p_0, \dots, p_d\}$ is a class of polynomials such that the first $d+1$ basis functions span the same space as the first $d+1$ monomials, i.e.\ the space $\mathcal P_d$ of polynomials of degree $d$ and such that
\[
\langle p_d, p_r\rangle = \int_{-1}^1p_d(x)\,p_r(x)\dx = \frac{2\,\delta_{dr}}{2d+1}
\]
with respect to the $L^2(-1,1)$-inner product. We pursue the same reproducing kernel Hilbert space (RKHS) strategy as for the monomial basis with square summable coefficients, but entirely on the real line.
Specifically, considering feature maps $\phi_d:(-1,1)\to \R^{d+1},\ \phi_d(x) = (p_0(x),\dots,p_d(x))$, the approximation is characterized by the kernel
\[
K_d(x, x') =\phi(x)^T \phi(x') = \sum_{k=0}^d p_k(x)\,p_k(x').
\]
In $L^2((-1,1)^2)$, we observe that 
\begin{align*}
\|K_d\|_{L^2((-1,1)^2)}^2 &= \sum_{k=0}^d \sum_{l=0}^d\int_{-1}^1\int_{-1}^1 p_k(x)p_k(x')p_l(x)p_l(x')\dx\dx'\\
    &= \sum_{k, l=0}^d \left(\int_{-1}^1 p_k(x)p_l(x)\dx\right) \left(\int_{-1}^1p_k(x')p_l(x')\dx'\right)\\
    &= \sum_{k,l=0}^d \left(\frac {2\delta_{kl}}{2k+1}\right)^2 
    = \sum_{k=0}^d\frac4{(2k+1)^2} \leq 4 + \sum_{k=1}^\infty \frac1{k^2} = 4 + \frac{\pi^2}6<+\infty.
\end{align*}
By the same computation, we could have shown that 
\[
\|K_D - K_d\|_{L^2} \leq \sqrt{\sum_{k=d+1}^D \frac4{(2k+1)^2}} \to 0
\]
as $d,D\to +\infty$, i.e.\ the kernels $K_d$ are Cauchy and thus converge to a limit $K^*$ strongly in $L^2((-1,1)^2)$. On the other hand, the kernel {\em diverges} along the diagonal:
\[
\|K_d(x,x)\|_{L^1(-1,1)}^2 = \int_{-1}^1 \sum_{k=0}^d p_k(x)^2\dx = \sum_{k=0}^d \|p_k\|_{L^2(-1,1)}^2 = \sum_{k=0}^d\frac2{2d+1} \geq 1 + \sum_{k=1}^d \frac 1k\sim \log(d)
\]
by the orthogonality of the Legendre polynomials. The Laplace-Heine formula \cite[Theorem 8.21.2]{szeg1939orthogonal} for the asymptotic behavior of Legendre polynomials
\begin{equation}\label{eq laplace heine} 
p_k(\cos \theta) = \sqrt{\frac2{\pi k\,\sin \theta}} \,\cos\big((k+1/2)\theta - \pi/4\big) + O(k^{-3/2})\qquad \text{for }\theta \in (0,\pi)
\end{equation}
indicates that the kernels $K_d(x,x)$ along the diagonal even diverge in $L^2(a,b)$ for any sub-interval $(a,b)$ as the kernels (roughly) correspond to oscillations of increasing frequency $\sim k$ and slowly decreasing amplitude $\sim k^{-1/2}$. Thus, for a `generic' choice of points $x_1,\dots, x_n$ (for instance, independent uniform random sample points), we expect that 
\begin{itemize}
\item $K_d(x, x_i)$ converges in $L^2(-1,1)$ for all $i$,
\item $K_d(x_i, x_j)$ converges for all $i\neq j$ (i.e.\ at a generic pair of points), and
\item $K_d(x_i, x_i) \to +\infty$.
\end{itemize}

If these assumptions apply for the set $\{x_1,\dots, x_n\}$ and $f_d(x) = \sum_{i=1}^n a_{d,i} \,K_d(x,x_i)$ is the RKHS interpolant for data $\{(x_i, y_i) : i=1,\dots, n\}$, then
\[
\begin{pmatrix}
K_d(x_1,x_1) &\cdots& K_d(x_n,x_1)\\ \vdots &\ddots &\vdots\\
K_d(x_1,x_n) &\cdots & K_d(x_n,x_n)
\end{pmatrix}\begin{pmatrix} a_{d,1}\\ \vdots\\ a_{d,n}\end{pmatrix} = \begin{pmatrix} y_1\\ \vdots \\ y_n\end{pmatrix} \qquad\Ra\quad a_{d,i} \approx \frac{y_i}{K_d(x_i,x_i)} \approx \frac{y_i}{\log d}
\]
since the matrix is heavily dominated by the diagonal for sufficiently large $d$. If $K_d(x,x_i)$ remains bounded in $L^2(-1,1)$ and $a_{d,i}$ converges to zero for all $i$, we conclude that $f_d\to 0$ in $L^2(-1,1)$. Thus, we expect Legendre polynomials to asymptotically produce the same behavior as Chebyshev polynomials.

\begin{conjecture}
For $d\in\mathbb N_0$, let $p_0,\dots, p_d$ the Legendre basis for the space $\mathcal P_d$ of polynomials of degree at most $d$. Let $x_1, \dots, x_n \in [-1,1]$ and $y_1, \dots, y_n \in \R$. For $d\geq \mathbb N$, let
\[
a_d = \argmin \left\{\sum_{k=0}^d a_k^2 : \sum_{k=0}^d a_kp_k(x_i) = y_i \text{ for all } i = 1,\dots, n\right\}.
\]
Then the sequence of polynomials $P_d(x) = \sum_{k=0}^d a_{k,d}\,p_k(x)$ converges to zero strongly in $L^2(-1,1)$.
\end{conjecture}

However, as the divergence in $K_d(x,x)$ is logarithmically slow, we expect this deterioration to happen very slowly -- see Figure \ref{figure orthogonal bases} for a comparison between the Legendre and Chebyshev basis in interpolating constant functions. 

A similar `slow deterioration' phenomenon in the context of double descent was explored in \cite{ma2020slow}, where the authors show that the phenomenon which affects absolute minimizers of the least squares energy does not affect solutions found by gradient descent optimization over a time scale spanning several orders of magnitude. Such partial regularization is observed in Figure \ref{figure basis comparison} for heavy overparametrization of degree $d=1,000$ with $n= 15$ data points.

\begin{figure}
    \centering
\includegraphics[width=.32\textwidth]{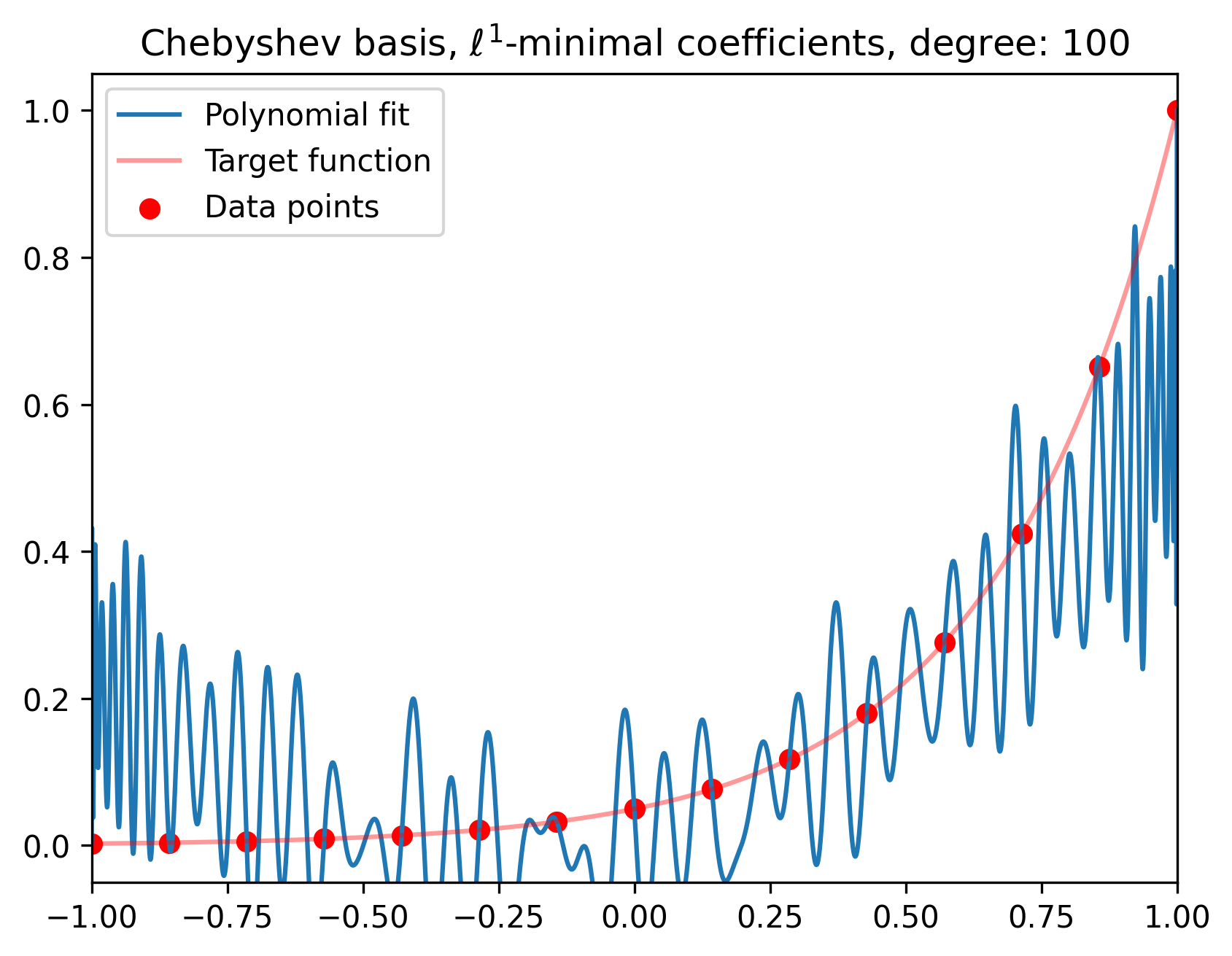}
\includegraphics[width=.32\textwidth]{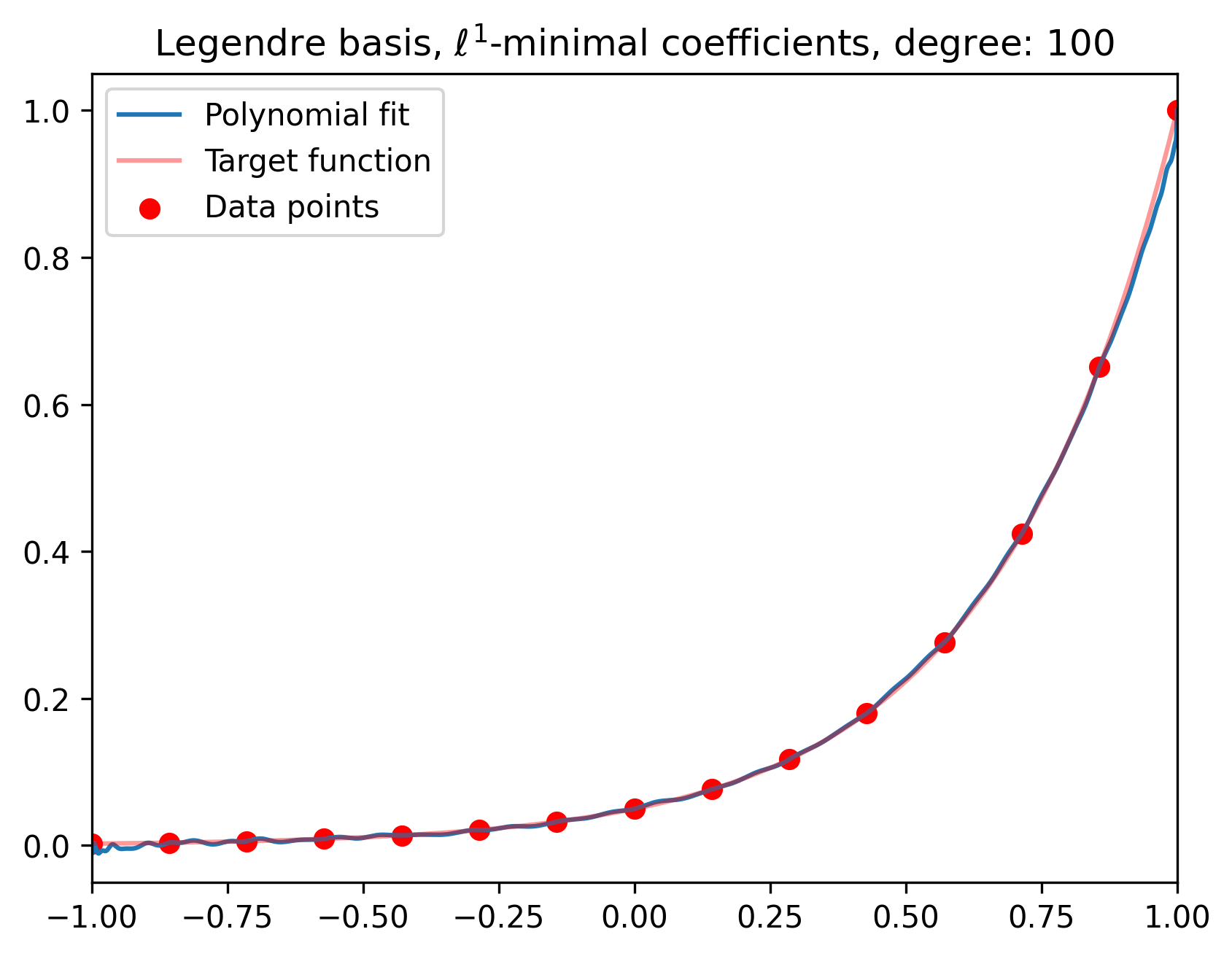}
\includegraphics[width=.32\textwidth]{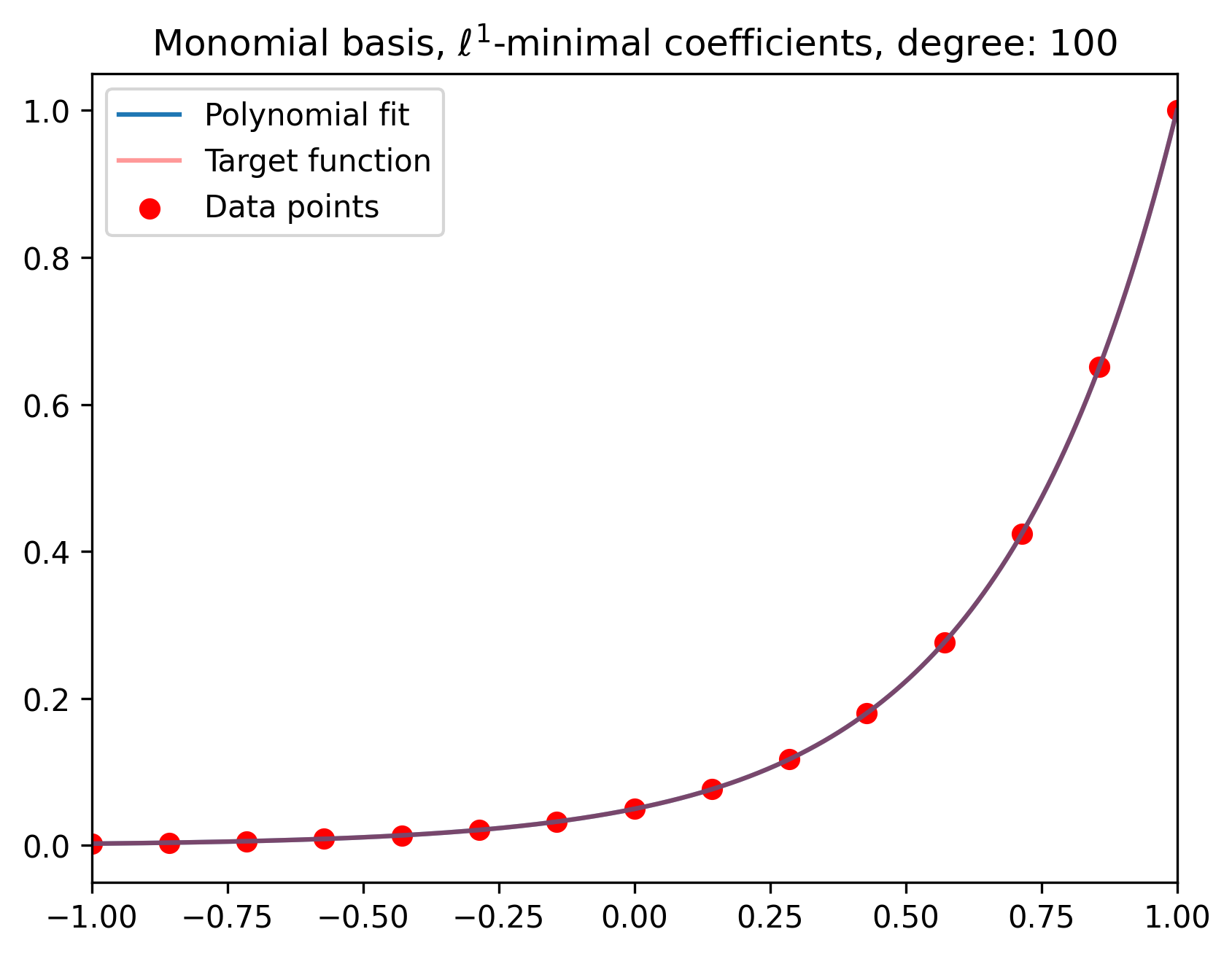}
\vspace{3mm}

\includegraphics[width=.32\textwidth]{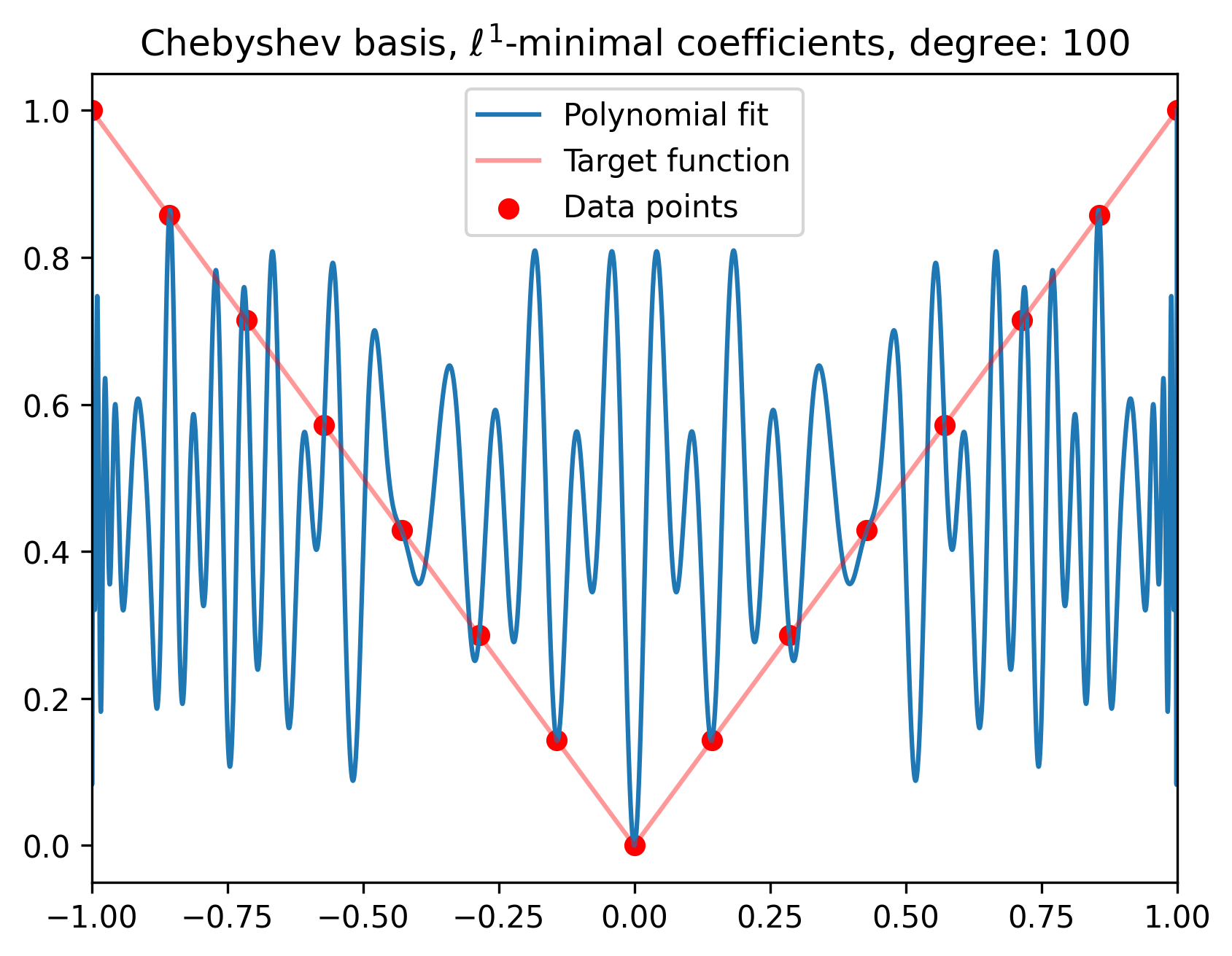}
\includegraphics[width=.32\textwidth]{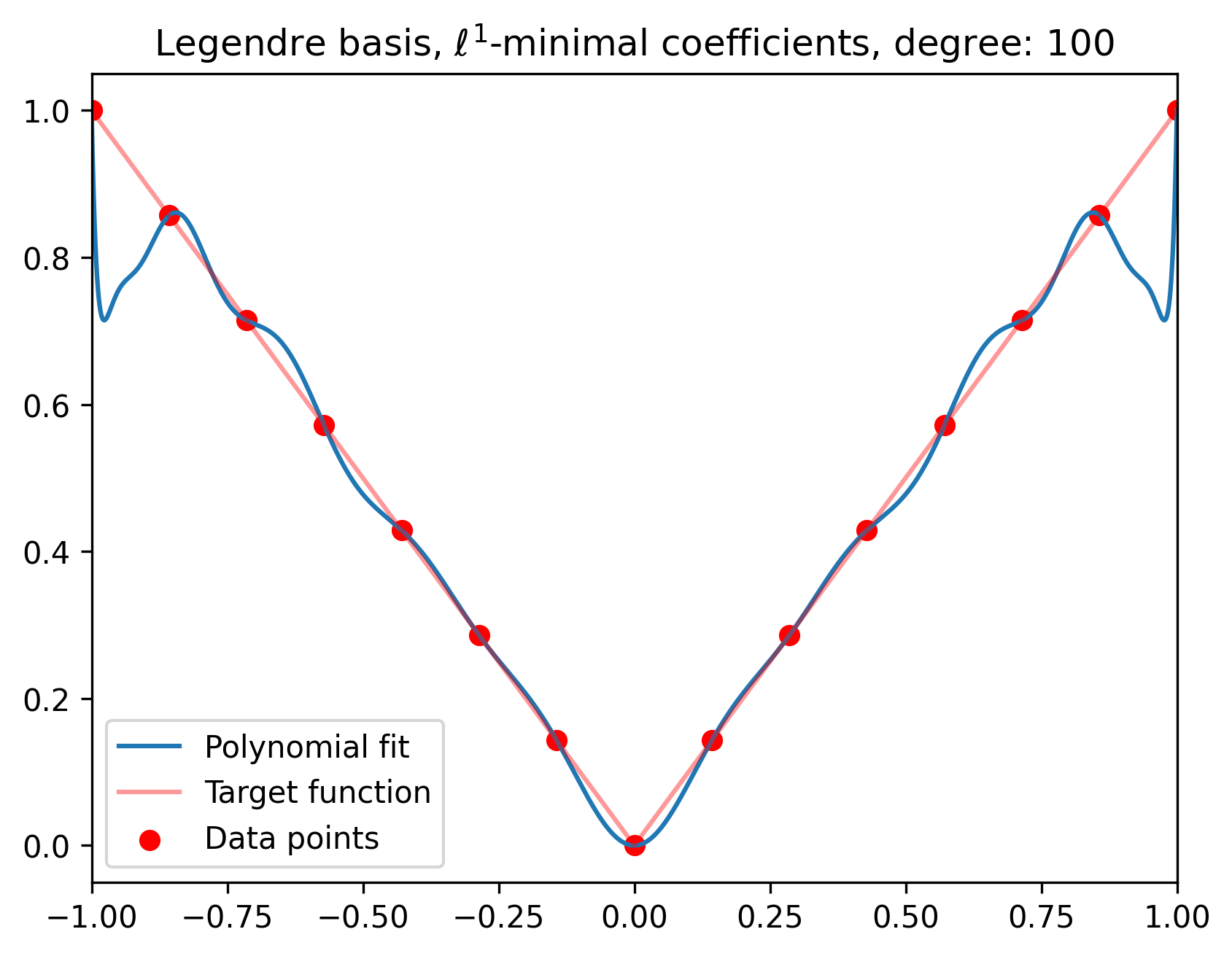}
\includegraphics[width=.32\textwidth]{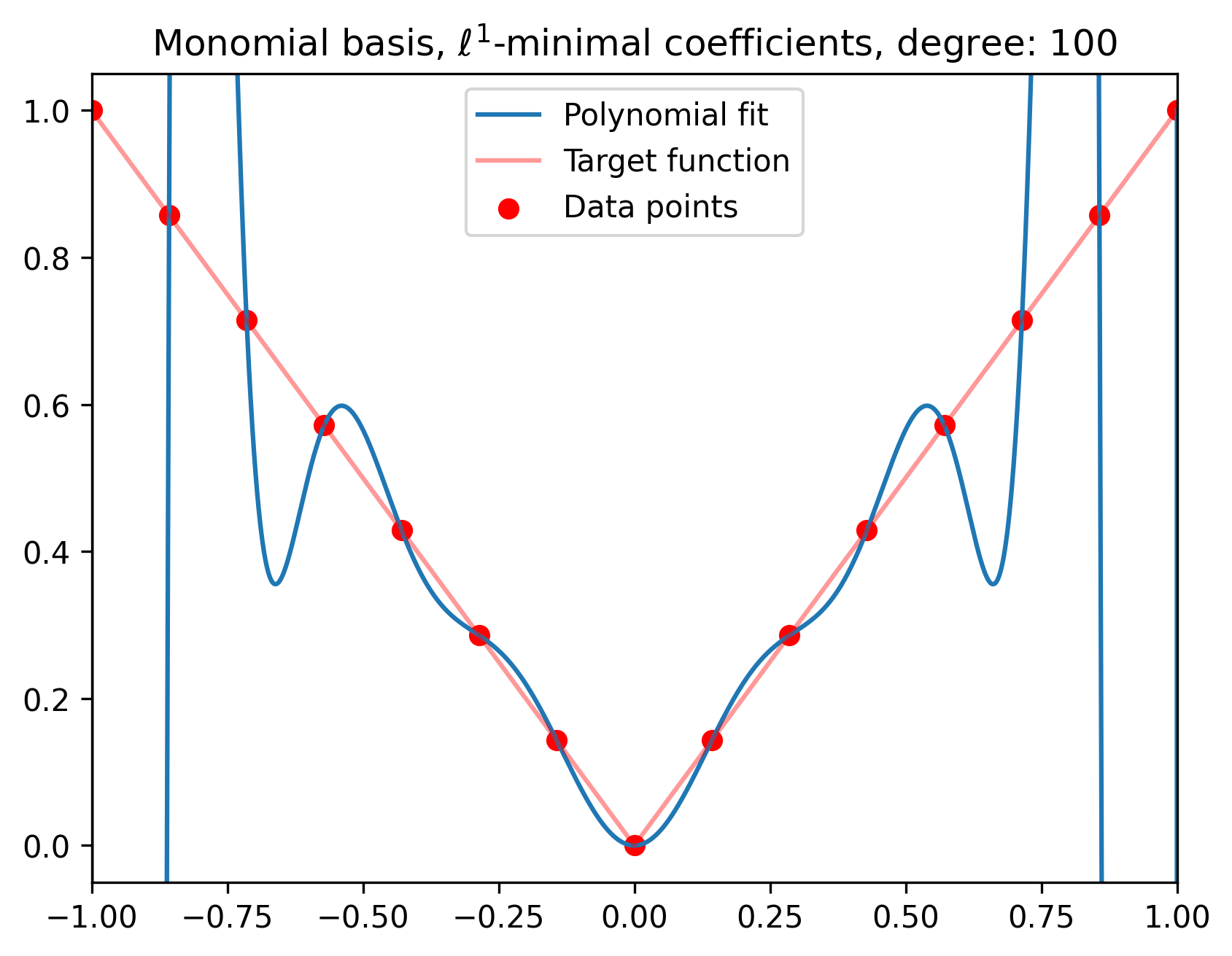}
\vspace{3mm}

\includegraphics[width=.32\textwidth]{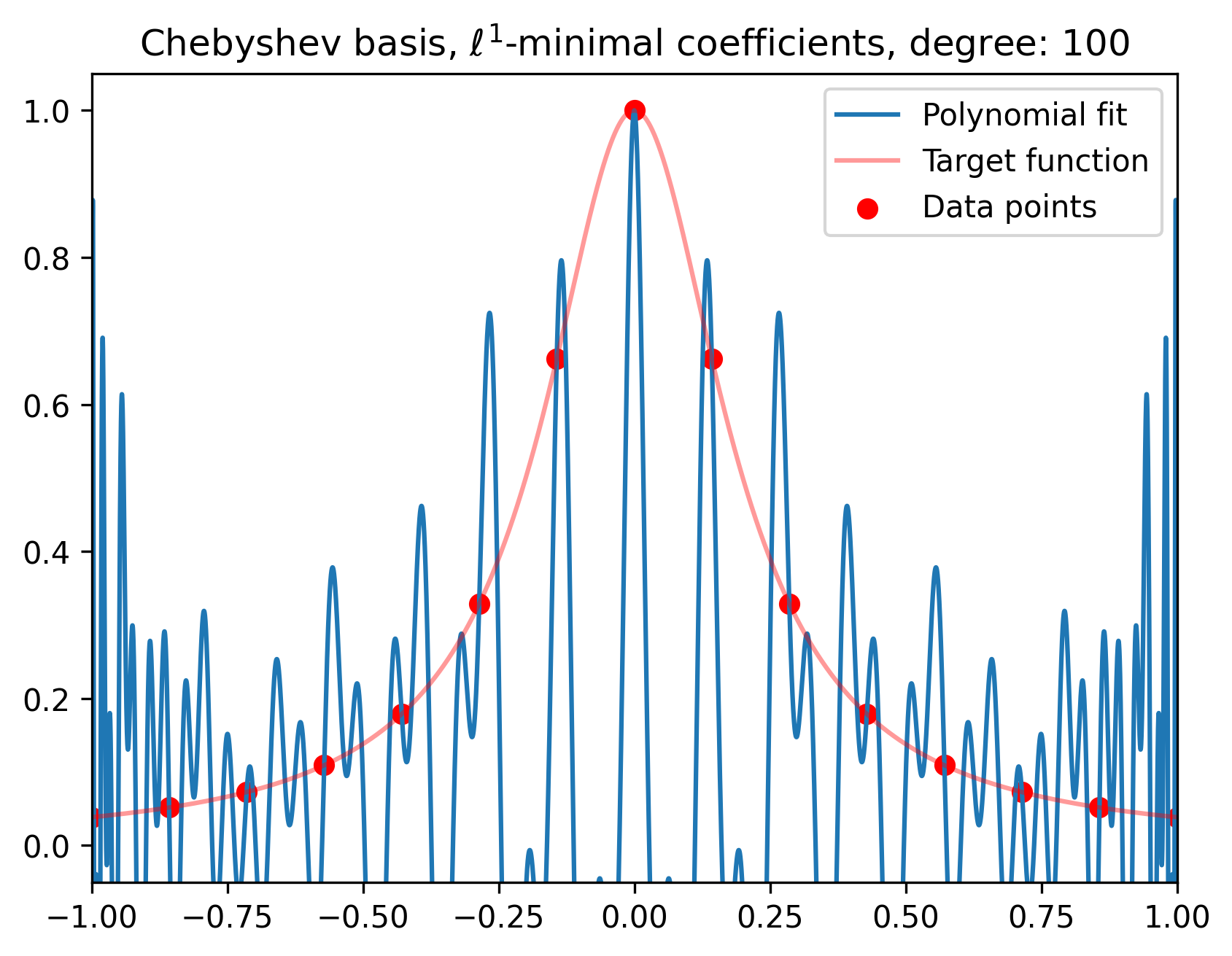}
\includegraphics[width=.32\textwidth]{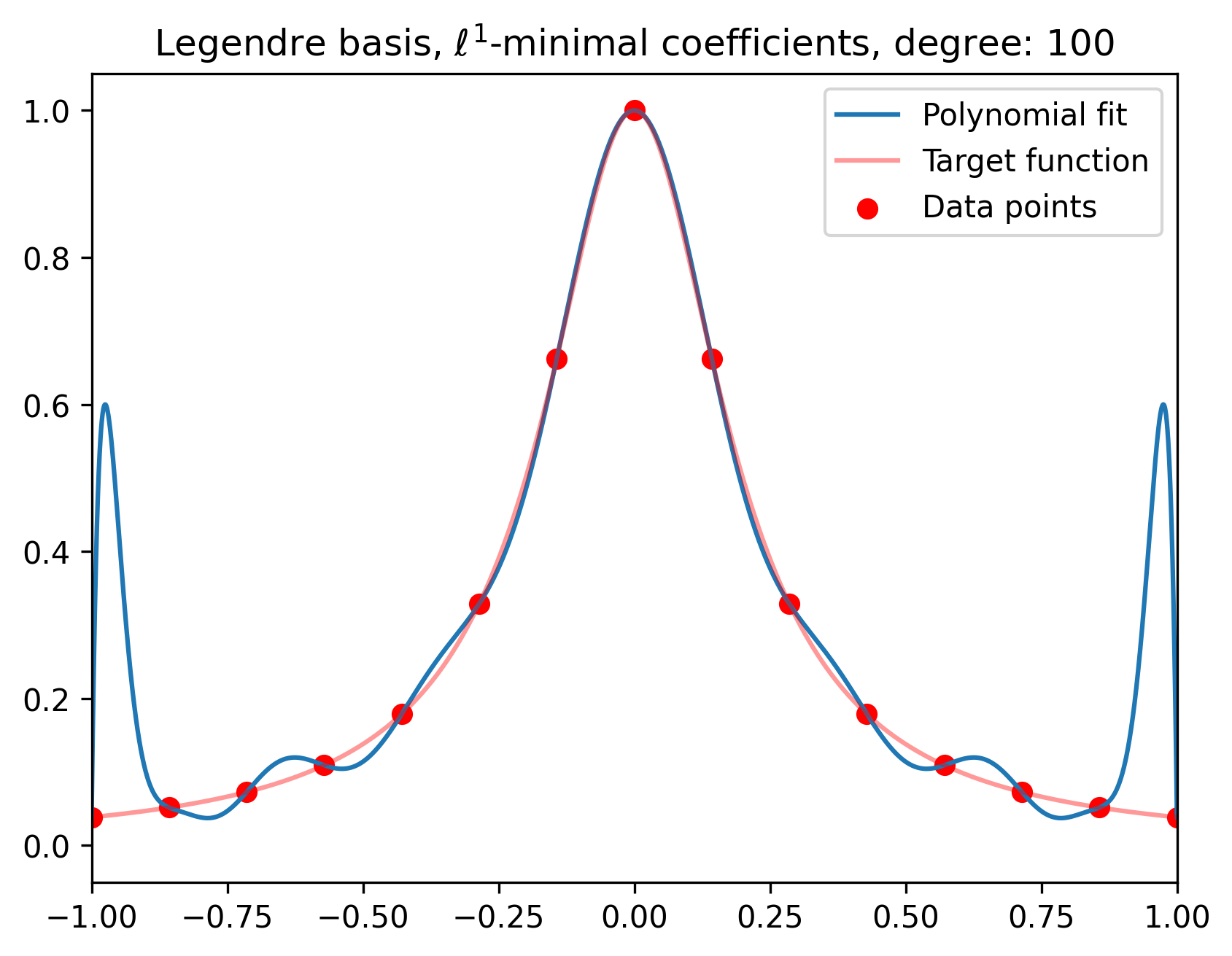}
\includegraphics[width=.32\textwidth]{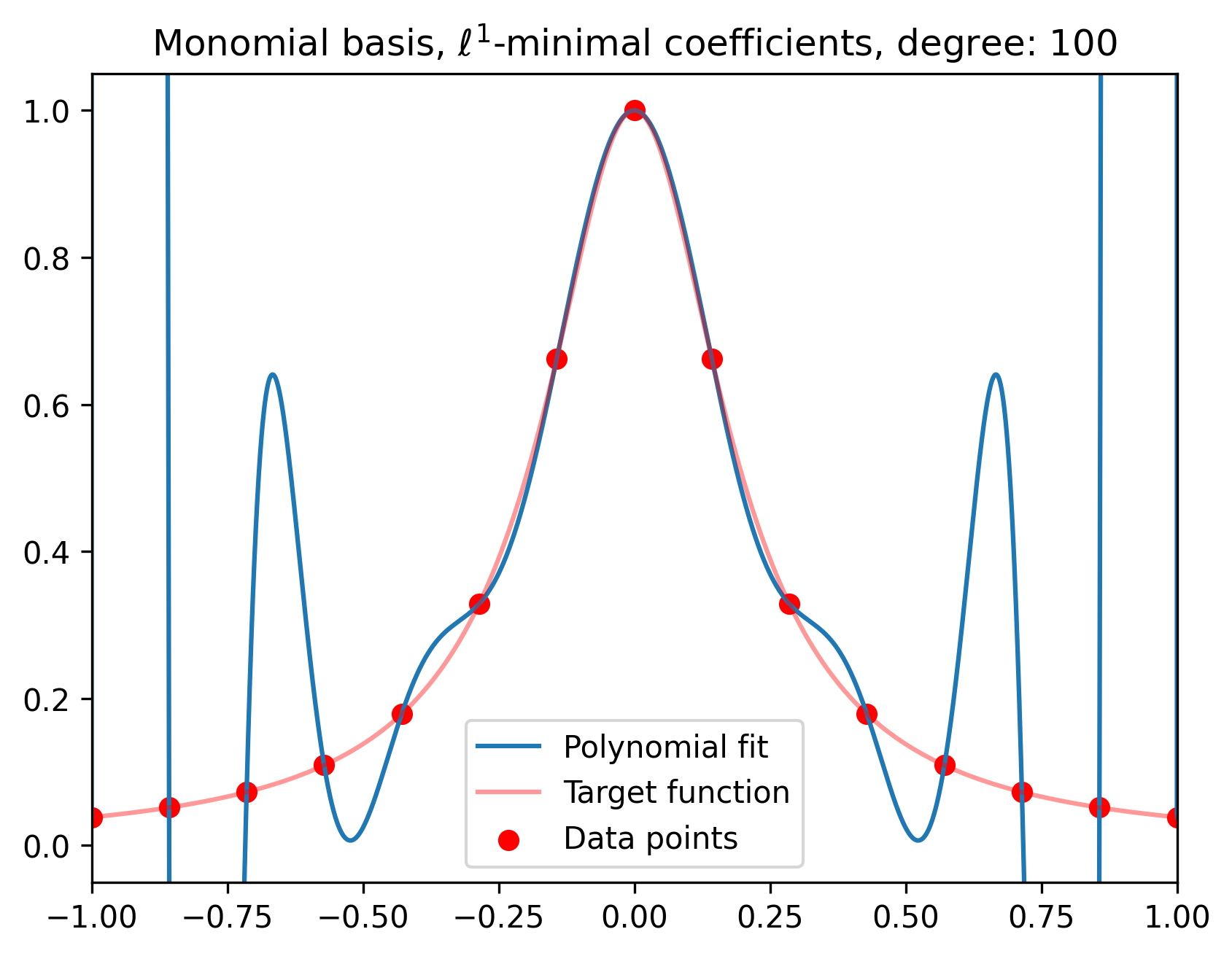}
\vspace{3mm}
    \caption{\label{figure basis comparison ell1}
    Interpolating input-output pairs $x_i, f(x_i)$ at 15 equidistant points in $[-1,1]$ by polynomials of degree $d = 100$ with $\ell^1$-minimal coefficients for various functions and bases. For higher degree as in Figure \ref{figure basis comparison}, the Chebyshev basis interpolant becomes so oscillatory as to obscure the entire picture. First row: $f(x) = \exp(3x-3)$, second row: $f(x) = |x|$, third row: $f(x) = \frac1{1+25x^2}$. First column: Chebyshev basis, second column: Legendre basis, third column: Monomial basis.\\ 
    Interpolants in Chebyshev basis are highly oscillatory. Interpolants in Legendre basis follow $f$ surprisingly closely. We do not provide an analytic justification. Interpolants in monomial basis are spectacularly good for power series with positive coefficients, but exhibit a Runge type phenomenon for the non-smooth absolute value target function and the Runge function.}
\end{figure}

\subsection{Legendre Basis: Analogy to Fractional Order Sobolev Spaces} \label{section legendre sobolev}

The previous argument was only applicable for a fixed dataset $\{x_1,\dots, x_n\}$ of `generic' points and thus could not easily be generalized to focus on equidistant points or Chebyshev nodes, or to exploring a joint regime taking $n, d_n\to \infty$ simultaneously and exploring whether e.g.\ a polynomial scaling $d_n \sim n^p$ produces different results for some $p>1$. In this section, we will explore this in more detail, but remaining on a heuristic level. Specifically, we exploit the similarity between Legendre polynomials and trigonometric functions due to \eqref{eq laplace heine} and utilize existing work on fractional Sobolev spaces to draw an analogy. For further context and background on fractional Sobolev spaces, we recommend \cite{di2012hitchhiker, leoni2023first}.

\begin{example}[Fractional Sobolev spaces]
Functions in $L^2(0,2\pi)$ can be represented uniquely as a Fourier sum $f_a(x) = \sum_{k\in \mathbb Z} a_k e^{i\, kx}$ and
\[
\|f_a\|_{L^2(0,2\pi)}^2 = 2\pi\sum_{k\in \mathbb Z} |a_k|^2, \qquad \|f_a'\|_{L^2} = 2\pi \sum_{k\in\mathbb Z} k^2|a_k|^2, \qquad \big\|f_a^{(m)}\big\|_{L^2(0,2\pi)}^2 = 2\pi \sum_{k\in\mathbb Z} |k|^{2m} a_k^2
\]
if the series converge. If the weighted series of square coefficients converges for $m\geq 1$, the functions $f$ are periodic and continuous, including at $x\in\{0,2\pi\}$. The expression
\[
\left\|\sum_{k\in\Z}a_ke^{ikx}\right\|^2_{H^m(S^1)} = \sum_{k\in\mathbb Z} \big(1+ |k|^{2m}\big)|a_k|^2
\]
defines a norm on the Sobolev space $H^m(S^1)$ of $2\pi$-periodic Sobolev functions on $\R$ with $m$ weak derivatives in $L^2$. 

These notions are generalized to Sobolev spaces of fractional order as
\[
H^s(S^1) = \left\{ f(x) = \sum_{k\in\mathbb Z} a_k\,e^{kix} : \sum_{k\in\mathbb Z} \big(1 + |k|^{2s}\big)a_k^2 < +\infty\right\}, \qquad \|f\|_{H^s(S^1)}^2 = \sum_{k\in\mathbb Z} \big(1 + |k|^{2s}\big)a_k^2 .
\]
In this scale of spaces with continuous index $s$, the regularity of functions is well-understood:
\begin{enumerate}
\item If $s>1/2$, all functions in $H^s(S^1)$ are continuous.
\item If $s< 1/2$, functions in $H^s(S^1)$ may have discontinuities. For instance, if $I\subseteq (0,2\pi)$ is an interval, the indicator function $1_I$ is an element of $H^s(S^1)$ with jump discontinuities.
\end{enumerate}
The most subtle case is $s=1/2$. Elements of $H^{1/2}(S^1)$ cannot have jump discontinuities, but they need not be bounded, let alone continuous. More precisely, a function $f:S^1\to \R$ is in $H^{1/2}(S^1)$ if and only if there exists $u\in H^1(\mathbb D)$ such that $u = f$ on $S^1 = \partial\mathbb D$ in the sense of traces. The $H^{1/2}$-norm above is equivalent to
\[
\|f\|_{H^{1/2}(S^1)}' = \inf\left\{\|u\|_{H^1(\mathbb D)} : \mathrm{tr}\,u = f\right\}.
\]
Thus, to understand the possible singularities of $f\in H^{1/2}(S^1)$, it suffices to understand the singularities of Sobolev functions in the integer order space $H^1(\mathbb D)$. The sequence
\[
u_m(x) = \begin{cases} 1 & \|x\|\leq e^{-m}\\ \frac{\log \|x\|}{m} & e^{-m} \leq \|x\|\leq 1\\ 0 &\|x\|\geq 1\end{cases} ,\qquad
 \int \|\nabla u_m\|^2\dx = \frac{2\pi} {m^2} \int_{e^{-m}}^1 \frac1{r^2}\,r\d r = \frac{2\pi}m
\]
satisfies $u_m(0) = 1$ for all $m\in\mathbb N$ and $u_m \to 0$ in $H^1(\R^2)$. As a corollary, using $U_{n,m}(x) := \sum_{i=1}^n y_i\,u_m(x-x_i)$ we see that it is possible to prescribe any values $y_i$ at $x_i \in S^1$ with arbitrarily small $H^{1/2}(S^1)$-norm. The interpolation problem
\[
\min\left\{\|u\|_{H^s(S^1)} : u(x_i) = y_i\text{ for all }i =0, \dots, n\right\}
\]
is therefore well-defined if and only if $s>1/2$.
\end{example}

The Laplace-Heine formula \eqref{eq laplace heine} indicates that
\[
\sum_{k=0}^\infty a_k \,p_k(\cos \theta) \approx \sum_{k=0}^K a_k\,p_k(\cos \theta) + \sqrt{\frac2{\pi\sin\theta}}\sum_{k=K+1}^\infty \frac{a_k}{\sqrt k}\,\cos\left(\left(k+\frac12\right)\theta - \frac\pi 4\right).
\]
The cosine series behaves similarly to a complex Fourier series with basis functions oscillating at analogous frequency $\sim k$. 
For $\theta\approx \pi/2$, the factor $1/\sin\theta$ is a smooth function and the change of variables $\theta\mapsto\cos \theta$ is diffeomorphic, so a reasonable hypothesis is that the smoothness $s$ of $\sum_k a_k p_k(x)$ in $[-r,r]$ for $r<1$ is governed by the summability of
\[
\sum_{k=1}^\infty (1+k^{2s}) \left(\frac {a_k}{\sqrt k}\right)^2 \leq 2 \sum_{k=0}^\infty k^{2s-1}\,a_k^2.
\]
In our setting, this weighted sum of squares is only guaranteed to be finite for $s\leq 1/2$, i.e.\ in the regime where continuity is not guaranteed. We conjecture that 
\[
\left\|\sum_{k=0}^\infty a_k p_k\right\| := \left(\sum_{k=0}^\infty a_k^2\right)^{1/2}
\]
defines a norm which {\em does not} control the $C^0$-norm. In particular, we that the interpolation problem is {\em not} well-defined in the setting of Legendre series and that the polynomial interpolation problem with basis coefficient $\ell^2$-norm regularization degenerates as the degree approaches infinity and converges to the zero function, i.e.\ the unique function with norm zero.

In the companion article \cite{second_article}, we explore a fractional Sobolev space perspective for a rescaled version of the Chebyshev basis, which is somewhat easier to analyze.

\section*{Acknowledgements}

JW's work was partially supported by a Brackenridge Summer Fellowship of the David C.\ Frederick Honors College at the University of Pittsburgh.
Neil Slavishak contributed to early numerical experiments in the context of this work as an undergraduate researcher.
SW is grateful to Rishi Sonthalia and Guido Mont\'ufar for insightful conversations on polynomial interpolation and double descent.

\bibliographystyle{alpha}
\bibliography{bibliography.bib}

\end{document}